\documentclass[12pt]{article}
\usepackage[amsmath]{e-jc}

\usepackage{graphicx}

\usepackage[all]{xy}
\usepackage{amsthm, amscd, latexsym, multicol, verbatim, enumerate, graphicx,xy, color}
\usepackage{adjustbox}
\usepackage{mathrsfs}
\usepackage{tikz,stackrel}
\usepackage{mathdots}
\usetikzlibrary{calc}
\usetikzlibrary{matrix, 
arrows, decorations.pathmorphing}
\usetikzlibrary{intersections}
\usetikzlibrary{decorations.markings}
\usetikzlibrary{external}

\usepackage{tikz-cd}
\usetikzlibrary{decorations.markings}
\makeatletter
\tikzcdset{
open/.code={\tikzcdset{hook, circled};},
closed/.code={\tikzcdset{hook, slashed};},
circled/.code={\tikzcdset{markwith={\draw (0,0) circle (.375ex);}};},
slashed/.code={\tikzcdset{markwith={\draw[-] (-.4ex,-.4ex) -- (.4ex,.4ex);}};},
markwith/.code={
\pgfutil@ifundefined{tikz@library@decorations.markings@loaded}%
{\pgfutil@packageerror{tikz-cd}{You need to say %
\string\usetikzlibrary{decorations.markings} to use arrow with markings}{}}{}%
\pgfkeysalso{/tikz/postaction={/tikz/decorate,
/tikz/decoration={
markings,
mark = at position 0.5 with
{#1}}}}},
}
\makeatother
\makeatletter
\tikzset{
circled/.code={\tikzset{markwith={\draw (0,0) circle (.375ex);}};},
slashed/.code={\tikzset{markwith={\draw[-] (-.4ex,-.4ex) -- (.4ex,.4ex);}};},
markwith/.code={
\pgfutil@ifundefined{tikz@library@decorations.markings@loaded}%
{\pgfutil@packageerror{tikz-cd}{You need to say %
\string\usetikzlibrary{decorations.markings} to use arrow with markings}{}}{}%
\pgfkeysalso{/tikz/postaction={/tikz/decorate,
/tikz/decoration={
markings,
mark = at position 0.5 with
{#1}}}}},
}
\makeatother

\makeatletter
\providecommand{\leftsquigarrow}{%
  \mathrel{\mathpalette\reflect@squig\relax}%
}
\newcommand{\reflect@squig}[2]{%
  \reflectbox{$\m@th#1\rightsquigarrow$}%
}
\makeatother

\newcommand{\hourglass}{\rotatebox[origin=c]{90}{$\bowtie$}}
\newcommand{\ie}{{\em i.e.}}

\dateline{TBD}{TBD}{TBD}

\MSC{05A19, 05B45, 05C70, 15B36}

\Copyright{The authors. Released under the CC BY-ND license (International 4.0).}

\title{Combinatorial connections in snake graphs: \\
Tilings, lattice paths, and perfect matchings}

\author{Carolina Melo\authornote{1}
}

\authortext{1}{Departamento de Matemáticas, Universidad Nacional de Colombia, Bogotá, Colombia
   (\email{amelol@unal.edu.co})}
\begin{document}

\maketitle

% ABSTRACT
% E-JC papers must include an abstract. The abstract should consist of a
% succinct statement of background followed by a listing of the principal
% new results that are to be found in the paper. The abstract should be
% informative, clear, and as complete as possible. Phrases like
% "we investigate..." or "we study..." should be kept to a minimum in
% favor of "we prove that..."  or "we show that...".  Do not include equation
% numbers, unexpanded citations (such as "[23]"), or any other references
% to things in the paper that are not defined in the abstract. The abstract
% may be distributed without the rest of the paper so it must be entirely
% self-contained.  Try to include all words and phrases that someone
% might search for when looking for your paper.

\begin{abstract}
Snake graphs and their perfect matchings play a key role in the description of cluster variables of cluster algebras associated to surfaces. In this paper, we introduce triangular snake graphs and establish a bijection between their routes (non-intersecting lattice paths), perfect matchings of their underlying snake graphs, and tilings. As an application, we show that the number of perfect matchings in straight snake graphs can be expressed in terms of determinants of Hankel matrices with Catalan number entries. Moreover, we prove that the number of perfect matchings in snake graphs can be expressed as a sum of products of Fibonacci numbers, and we show how Fibonacci and Pell sequences arise from determinants of matrices with Fibonacci entries.
\end{abstract}

\section{Introduction}

In this paper, we introduce triangular snake graphs, a directed acyclic graph structure naturally associated to snake graphs, and use them as a framework to study perfect matchings and their algebraic properties. Our main contributions are:
\begin{itemize}
  \item We establish a bijection between perfect matchings of a snake graph and k-routes (non-intersecting lattice paths) in its associated triangular snake graph, providing a combinatorial interpretation of classical path-counting results.
  \item A connection between snake graphs and determinants of Hankel matrices. We establish a connection between the number of perfect matchings in snake graphs and determinants of path matrices involving Catalan, Fibonacci, and Pell numbers, bridging tiling problems with algebraic identities.
\end{itemize}

\subsection{Overview} \label{def:catalan numbers} 
To place our results in context, we first review classical problems in tiling theory and their connections to combinatorics. Given a region (usually a subset of the Euclidean plane) and a collection of shapes (tiles), a \emph{tiling} of that region is defined as a set of non-overlapping tiles whose union is the region. The study of tilings and their combinatorial properties has been a central topic in discrete mathematics, with deep connections to perfect matchings, continued fractions, and determinant identities; see, for example, \cite{AS10, BEO20, EKLP92, Kas61}. This research often addresses key questions: Is it possible to cover a given region completely using specific tiles? How many distinct ways can such a region be covered without overlapping tiles? Can these tilings be connected to other combinatorial objects through bijections, and what combinatorial properties emerge from these configurations? \par
As the study of larger regions and more complex tiles progresses, the complexity of these questions increases due to the rapid growth in the number of possible configurations. Nevertheless, progress has been made by focusing on specific families of regions and tiles, yielding precise answers to these questions. A well-known example involves regions defined as subsets of the plane, with tiles being $1\times 2$ rectangles (commonly referred to as dominoes or domino tiles), placed on lattice points. The problem of counting the number of tilings (also called domino tilings) within a given region has been widely studied in mathematics; see, for example, \cite{FT61, Kas61}. For instance, the number of domino tilings for a $2\times n$ rectangle corresponds to the $(n+1)$-th Fibonacci number; see \cite{KP80}.
%As the study of larger regions and more complex tiles progresses, the complexity of these questions increases due to the rapid growth in the number of possible configurations. Nevertheless, progress has been made by focusing on specific families of regions and tiles, yielding precise answers to these questions. A well-known example involves regions defined as subsets of the plane, with tiles being $1\times 2$ rectangles (commonly referred to as dominoes or domino tiles), placed on lattice points. The problem of counting the number of tilings (also called domino tilings) within a given region has been widely studied in mathematics. For instance, the number of domino tilings of an $m\times n$ rectangle, calculated independently by Kasteleyn \cite{Kas61} and Fisher and Temperley \cite{FT61}, is given by
%\[ 
%\prod^{\lceil m/2\rceil}_{j=1}\prod^{\lceil  n/2\rceil}_{k=1}\left(4\cos^2\left(\frac{j\pi}{m + 1}\right)+ 4\cos^2\left(\frac{k\pi}{n + 1}\right)\right).
%\]

%\label{rem:Fibonacci number}In particular, using identity (1) from \cite{GR08}, the number of domino tilings for a $2\times n$ rectangle corresponds to the $(n+1)$-th Fibonacci number. However, for rectangles with odd dimensions in both width and height (\ie, $(2k+1)\times (2k^\prime+1)$ where $k$ and $k^\prime$ are non-negative integers), the number of domino tilings is zero. \par
Moreover, combinatorial proofs based on tilings have been used to establish and extend a variety of algebraic identities. Examples include connections with Fibonacci determinants \cite{BCQ05}, Hankel determinants \cite{EF05}, and the Fibonacci and Pell numbers \cite{KS09}. The work of Elkies, Kuperberg, Larsen, and Propp in \cite{EKLP92} has been foundational in this area. They introduced a family of planar regions called Aztec diamonds and studied their tilings with dominoes. Among their key contributions is the Aztec diamond theorem, which states that the number of domino tilings for this shape is exactly $2^{n(n+1)/2}$. Furthermore, they developed a method to transform tiling problems into the framework of counting paths on graphs. Central to this approach is the concept of a \textit{$k$-route}, which represents a set of $k$ non-intersecting lattice paths. \par
Building on the work of Gessel and Viennot \cite{GV89}, Eu and Fu established in \cite{EF05} connections between determinants of \textit{Hankel matrices} (square matrices in which each ascending skew-diagonal from left to right is constant) and the enumeration of $k$-tuples of non-intersecting large and small Schröder paths. These bijections establish a link between the domino tilings of Aztec diamonds and non-intersecting lattice paths. An alternative approach to the Aztec diamond theorem, distinct from the methods introduced by Elkies, Kuperberg, Larsen, and Propp, is presented in \cite{Ciu96}. In this work, Ciucu obtains the recurrence relation $a_n =2^{n}a_{n-1}$ using perfect matchings of cellular graphs. In general, a domino tiling of a region corresponds to a perfect matching in the grid graph obtained by placing a vertex at the center of each square of the region and connecting two vertices whenever their corresponding squares are adjacent.\par
A \textit{matching} in a graph is defined as a subset of independent edges, \ie, edges that share no common vertices. This concept has been widely studied in mathematics and computer science due to its broad applications, including graph coloring, flow networks, and neural networks; see, for example, \cite{LP86}. However, several open problems remain, such as finding a closed formula for the number of matchings in a graph. In particular, determining the number of perfect matchings, denoted as $m(G)$, in an arbitrary graph $G$ remains an unresolved problem. While exact closed formulas are generally unavailable, asymptotic approximations and exact results exist for specific classes of graphs; see, for example, \cite{Bre73, LP86}.\par 
%For example, in the case of complete graphs $K_n$ with an even number $n$ of vertices, the number of perfect matchings is given by
%\[m(K_n)=\frac{n!}{2^{n/2}\left(\frac{n}{2}\right)!}.\]  
%Brégman provided in \cite{Bre73} an upper bound for the number of perfect matchings in a balanced bipartite graph in terms of the degrees $d(v)$ of the vertices. Additionally, Kahn and Lovász (unpublished) extended Brégman's bound to an arbitrary graph, yielding the inequality. 
%\[
%m(G)\leq \prod_{v\in V}(d(v)!)^{1/2d(v)}.
%\] 
On the other hand, in the context of cluster algebras associated to surfaces, Musiker, Schiffler, and Williams in \cite{MSW11, MSW13} established a connection between cluster variables in cluster algebras and perfect matchings of snake graphs (connected planar graphs consisting of a finite sequence of square tiles). Their work provided a formula for computing cluster variables combinatorially. Recently, following the ideas of Eu and Fu, we showed in \cite{Mel19} how to associate a perfect matching of a ladder graph (or straight snake graph, where all tiles lie in a single column or row) to a $k$-route formed by Schröder paths. In this paper, we extend these ideas by constructing a bijection between perfect matchings and non-intersecting lattice paths, revealing new determinant identities involving some numerical sequences. \par 
%A snake graph $\mathcal{G}$ is a connected planar graph consisting of a finite sequence of square tiles $G_1, G_2,\dots, G_d$ with $d\geq 1$, such that $G_i$ and $G_{i+1}$ share exactly one edge $e_i$. This edge is either the north edge of $G_i$ and the south edge of $G_{i+1}$, or the east edge of $G_i$ and the west edge of $G_{i+1}$, for each $i=1,...,d-1$. Recently, following the ideas of Eu and Fu, we showed in \cite{Mel19} how to associate a perfect matching of a ladder graph (or straight snake graph, where all tiles lie in a single column or row) to a $k$-route formed by Schröder paths. \par

%The following synthesis encapsulates Proposition~\ref{prop:snake_graph_tilings_and_perfect_matchings}, Lemma~\ref{lemma:Perfect_matchings_to_routes}, and Theorem~\ref{the:Perfect_matchings_and_Routes} below.

We define the region $T(\mathcal{G})$, referred to as the snake graph cover, as the union of all unit squares centered at the vertices of a snake graph $\mathcal{G}$ with $d$ square tiles. Using the domino tilings of $T(\mathcal{G})$, the triangular snake graph $\mathcal{T}_\mathcal{G}$ is defined as an acyclic-directed graph with $d$ triangular tiles. Alternatively, employing the edge contraction operation, we provide an alternative method to construct $\mathcal{T}_\mathcal{G}$ (see Definition~\ref{def:contracted_snake_graph}). Leveraging Lemma~\ref{le:equality_of_sources_and_sinks} and Definition~\ref{def:k-vertices}, the set of $k$-paths $p_i$ from $s_i$ to $t_i$, where ${s_1, \dots, s_k}$ is the set of sources and ${t_1, \dots, t_k}$ is the set of sinks, can be considered. One may then pose the question: How many of these sets of $k$ paths are $k$-routes? The following synthesis encapsulates Proposition~\ref{prop:snake_graph_tilings_and_perfect_matchings}, Lemma~\ref{lemma:Perfect_matchings_to_routes}, and Theorem~\ref{the:Perfect_matchings_and_Routes} below, providing an answer to this question.  \label{principal results 2} % These results are also among the principal findings of this research.

\begin{theorem}Let $\mathcal{G}$ be a snake graph. 
\begin{enumerate}[(1)]
    \item There exists a bijection between the domino tilings of the region $T(\mathcal{G})$ and the set $\text{Match}(\mathcal{G})$ of perfect matchings of $\mathcal{G}$. 
    \item The set $\text{Match}(\mathcal{G})$ is in bijection with the set of $k$-routes from $s$ to $t$ in $\mathcal{T}_\mathcal{G}$, where $s=(s_1,\dots,s_k)$ and $t=(t_1,\dots,t_k)$.
\end{enumerate}
\end{theorem}

%Since Çanakçi and Schiffler established in \cite{CS18} that the number of perfect matchings $m(\mathcal{G})$ is equal to the numerator of the continued fraction associated to $\mathcal{G}$, therefore the item $(1)$ of the previous theorem allows us to know the number of the domino tilings of the region $T(\mathcal{G})$. Moreover, since the Lindström-Gessel-Viennot lemma says that the number of $k$-routes from $s$ to $t$ in $\mathcal{T}_\mathcal{G}$ is equal to the determinant of the path matrix $M_{st}$ (see Lemma~\ref{lem:LGV}), then item (2) tells us that this determinant can be interpreted as the numerator of the continued fraction and, vice versa, the continued fraction associated to $\mathcal{G}$ can be re-interpreted as a quotient of determinants of path matrices. \par \bigskip

Çanakçi and Schiffler established in \cite{CS18} that the number of perfect matchings of $\mathcal{G}$ is equal to the numerator of the continued fraction associated with $\mathcal{G}$. As a consequence, item $(1)$ of the preceding theorem provides a method for determining the number of domino tilings within the region $T(\mathcal{G})$. Moreover, by the Lindström-Gessel-Viennot lemma, which states that the number of $k$-routes from $s$ to $t$ in $\mathcal{T}_\mathcal{G}$ equals the determinant of the corresponding path matrix $M_{st}$ (as seen in Lemma~\ref{lem:LGV}), item $(2)$ implies that this determinant can be interpreted as the numerator of the associated continued fraction. Conversely, the continued fraction associated to $\mathcal{G}$ can also be expressed as a quotient of determinants of path matrices. \par 
To establish an explicit connection between snake graphs and Hankel matrices, it is necessary to describe how these combinatorial structures are linked through the study of path matrices and matchings. This exploration reveals a relationship between the entries of such path matrices and well-known numerical sequences, such as the Fibonacci sequence and the Catalan numbers. \par
The Catalan number sequence, denoted by $C$, is one of the most important sequences in combinatorics. It has been extensively studied and is associated with over 200 different combinatorial objects; see \cite{Sta15}. This sequence appears in the On-Line Encyclopedia of Integer Sequences (or OEIS) \cite{Sloane} with the identifier $A000108$ and is recursively defined as follows:
\[
C_0=1, \quad C_{n+1}=\sum_{i=0}^{n}C_iC_{n-i}, \quad \text{with} \quad n\geq 0.
\]

Examples of combinatorial objects counted by the Catalan numbers include triangulations of convex polygons, binary trees, and Dyck paths; see \cite[Corollary 6.2.3]{Sta99}. \par
Similarly, the Fibonacci sequence, denoted by $F$, is a well-known sequence defined recursively by:
\[
F_0 = 0, \quad F_1 = 1, \quad F_{n+1} = F_n + F_{n-1}, \quad \text{for} \quad n > 0.
\]
Fibonacci numbers have a wide range of applications in mathematics, computer science, and natural phenomena. They appear in areas such as the golden ratio, population dynamics, and the structure of plants; see, for example, \cite{Kos01}. \par
We explore the combinatorial nature of both Catalan and Fibonacci numbers by connecting them to tiling problems, perfect matchings on snake graphs and Hankel matrices. By extending the concepts of $k$-routes and perfect matchings, we establish a relationship between these sequences and determinants of matrices associated with snake graphs. In particular, we show how Hankel matrices can be used to study these combinatorial sequences, providing explicit connections between their entries and the number of perfect matchings in specific graph structures.  The following synthesis, encompassing Propositions~\ref{prop:Catalan_Fibonacci},~\ref{prop:Ladder_Fibonacci}, and~\ref{prop:Hankel matrix of Catalan numbers}, highlights these relationships:

\begin{theorem} Let $C$ be the sequence of Catalan numbers $C_n$, and let $F_n$ be the $n$-th Fibonacci number. Then the following relationship holds:
\[
F_n=\det(H(C))=\det (M_{st}), \]
where $H(C)$ is a matrix defined in terms of the Hankel matrices of the Catalan numbers, and $M_{st}=(m_{ij})$ is the path matrix associated to the triangular snake graph $\mathcal{T}_{L_{n-2}}$ which corresponds to the vertical ladder graph $L_{n-2}$.
\end{theorem}

Additionally, Proposition~\ref{prop:Ladder_Fibonacci} and
Proposition~\ref{prop:Hankel matrix of Catalan numbers} not only address Hankel matrices with entries from the sequence of Catalan numbers but also introduce matrices whose entries are Fibonacci numbers, with determinants matching those of the earlier Hankel matrices. In \S~\ref{sec:Some identities in not straight snake graphs} we explore this property in greater generally, showing that for snake graphs, the number of perfect matchings can be expressed as a sum of products of Fibonacci numbers, described by the determinant of a matrix with entries given in terms of Fibonacci numbers.

\subsection{Organization}
The paper is structured as follows. In \S~\ref{sec:Background} we provide the background on lattice paths, Aztec diamonds, Hankel matrices, and snake graphs, and set up the notation that will be used throughout the text. In \S~\ref{sec:Perfect matchings, tilings and non-intersecting lattice paths} we give a bijection between non-intersecting lattice paths (or routes) in an acyclic-directed graph and perfect matchings of the snake graph. We also introduce the triangular snake graph $\mathcal{T}_\mathcal{G}$ and describe how each perfect matching corresponds to a $k$-route of $\mathcal{T}_\mathcal{G}$ (\S~\ref{sec:Lattice_paths_and_perfect_matchings}). Finally, in \S~\ref{sec:Fibonacci identities and Hankel matrices} we provide key identities in terms of Fibonacci numbers that allow the computation of determinants of matrices, either Hankel matrices associated with Catalan numbers in the case of straight snake graphs, or more generally, determinants of path matrices.

%The reconstruction conjecture states that the multiset of unlabeled vertex-deleted subgraphs of a graph determines the graph provided it has at least three vertices.  This problem was independently introduced by Kelly~\cite{Kelly} and Ulam~\cite{Ulam}.   The reconstruction conjecture is widely studied \cite{Bollobas, FGH, HHRT, KSU, Stockmeyer, WS} and is very interesting. See \cite{BH} for more about the reconstruction conjecture.

%\begin{definition} 
 % A graph is \emph{fabulous} if \emph{rest of definition here}.
%\end{definition}

%\begin{theorem}\label{Thm:FabGraphs}
 % All planar graphs are fabulous.
%\end{theorem}

%%%%%%%%%%%%%%%%%%%%%%%%%%%%%%%%%%%%%%%%%%%%%%%%%%%
\section{Background} \label{sec:Background}

\subsection{Lattice paths and Aztec diamonds}\label{subsection:Aztec diamonds and Schröder paths} 

Let $S$ be a subset of $\mathbb{Z}^d$. A \emph{lattice path} $L$ in $\mathbb{Z}^d$ of length $k$ with steps in $S$ is a sequence $v_0, v_1,\cdots , v_k \in \mathbb{Z } ^d$ such that each consecutive difference $v_i - v_{i-1}$ belongs to $S$. $L$ is said to start at $v_0$ and end at $v_k$. The enumeration of these paths often leads to interesting numerical sequences. For example, the set of \emph{Dyck paths} $\mathcal{D}_n$ consists of all lattice paths in $\mathbb{Z}^2$ from $(0, 0)$ to $(2n, 0)$ for some $n> 0$, using steps $S=\{(1,1), (1,- 1)\}$ without crossing below the $x$-axis. Their number corresponds to the $n$-th Catalan number. The following $5$ lattice paths make up $\mathcal{D}_3$. \par \bigskip

\begin{center}
\begin{tikzpicture}[y=.3cm, x=.3cm,font=\normalsize, scale=1.5]
%\draw[black,thick,->] (0,0)--(6.5,0) node[right] {}; % Eje x
% Enumeración del eje x
%\foreach \x/\xtext in {1/1, 2/2, 3/3, 4/4, 5/5, 6/6} 
%\draw[shift={(\x,0)},black] (0pt,2pt)--(0pt,-2pt) node[below] {$\xtext$};

%\draw[black,thick,->] (0,0)--(0,3.5) node[left,above] {}; % Eje y
% Enumeración del eje y
%\foreach \y/\ytext in {1/1, 2/2, 3/3} 
%\draw[shift={(0,\y)},black] (2pt,0pt)--(-2pt,0pt) node[left] {$\ytext$};

\foreach \x in {0,1,2}{
\draw[->, >=latex,gray] (\x,\x) -- (\x+1,\x+1);
\draw[->, >=latex,gray] (3+\x,3-\x) -- (3+\x+1,3-\x-1);
} 
%::::::::::::::::::::::::::::::::::::::::::::::::::::::::::
\foreach \x in {0,1}{\draw[->, >=latex,gray] (\x+10,\x) -- (\x+1+10,\x+1);}
\foreach \x in {1,2}{\draw[->, >=latex,gray] (3+\x+10,3-\x) -- (3+\x+1+10,3-\x-1);}
\draw[->, >=latex,gray] (12,2) -- (13,1);
\draw[->, >=latex,gray] (13,1) -- (14,2);

%::::::::::::::::::::::::::::::::::::::::::::::::::::::::::
\draw[->, >=latex,gray] (-5,-3) -- (-4,-2);
\draw[->, >=latex,gray] (-4,-2) -- (-3,-3);
\draw[->, >=latex,gray] (-3,-3) -- (-2,-2);
\draw[->, >=latex,gray] (-2,-2) -- (-1,-1);
\draw[->, >=latex,gray] (-1,-1) -- (0,-2);
\draw[->, >=latex,gray] (0,-2) -- (1,-3);
%::::::::::::::::::::::::::::::::::::::::::::::::::::::::::
\draw[->, >=latex,gray] (15,-3) -- (16,-2);
\draw[->, >=latex,gray] (16,-2) -- (17,-1);
\draw[->, >=latex,gray] (17,-1) -- (18,-2);
\draw[->, >=latex,gray] (18,-2) -- (19,-3);
\draw[->, >=latex,gray] (19,-3) -- (20,-2);
\draw[->, >=latex,gray] (20,-2) -- (21,-3);
%::::::::::::::::::::::::::::::::::::::::::::::::::::::::::
\draw[->, >=latex,gray] (5,-3) -- (6,-2);
\draw[->, >=latex,gray] (6,-2) -- (7,-3);
\draw[->, >=latex,gray] (7,-3) -- (8,-2);
\draw[->, >=latex,gray] (8,-2) -- (9,-3);
\draw[->, >=latex,gray] (9,-3) -- (10,-2);
\draw[->, >=latex,gray] (10,-2) -- (11,-3);
\end{tikzpicture}
\end{center} \par \bigskip

Similarly, the (large) Schröder number $r_n$ counts the number of lattice paths from $(0, 0)$ to $(2n, 0)$, for some $n> 0$, using steps $S=\{(1,1), (1,- 1), (2,0)\}$ and staying above the $x$-axis. For a fixed $n$, let $S_n$ denote the set of these \emph{Schröder paths}. The following $6$ lattice paths constitute $S_2$. \par \bigskip

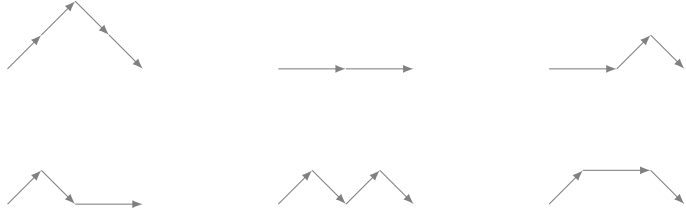
\begin{center}
\begin{tikzpicture}[y=.3cm, x=.3cm,font=\normalsize, scale=1.5]
\foreach \x in {0,1}{
\draw[->, >=latex,gray] (\x,\x) -- (\x+1,\x+1);
\draw[->, >=latex,gray] (2+\x,2-\x) -- (2+\x+1,2-\x-1);
} 
%:::::::::::::::::::::::::::::::::::::::::::::::
\draw[->, >=latex,gray] (8,0) -- (10,0);
\draw[->, >=latex,gray] (10,0) -- (12,0);

%:::::::::::::::::::::::::::::::::::::::::::::::
\draw[->, >=latex,gray] (16,0) -- (18,0);
\draw[->, >=latex,gray] (18,0) -- (19,1);
\draw[->, >=latex,gray] (19,1) -- (20,0);

%:::::::::::::::::::::::::::::::::::::::::::::::
\draw[->, >=latex,gray] (0,-4) -- (1,-3);
\draw[->, >=latex,gray] (1,-3) -- (2,-4);
\draw[->, >=latex,gray] (2,-4) -- (4,-4);

%:::::::::::::::::::::::::::::::::::::::::::::::
\draw[->, >=latex,gray] (8,-4) -- (9,-3);
\draw[->, >=latex,gray] (9,-3) -- (10,-4);
\draw[->, >=latex,gray] (10,-4) -- (11,-3);
\draw[->, >=latex,gray] (11,-3) -- (12,-4);

%:::::::::::::::::::::::::::::::::::::::::::::::
\draw[->, >=latex,gray] (16,-4) -- (17,-3);
\draw[->, >=latex,gray] (17,-3) -- (19,-3);
\draw[->, >=latex,gray] (19,-3) -- (20,-4);
\end{tikzpicture}
\end{center} \par \bigskip

These numbers are intimately connected to tiling problems. The \emph{Aztec diamond} of order $n$, denoted by $Az(n)$, is defined as the union of all unit squares whose corners have integer coordinates $(x,y)$ that satisfy $|x|+|y|\leq n+1$. A \emph{domino tile} is a rectangle of size 1 by 2 or 2 by 1 with corners with integer coordinates. A \emph{domino tiling} of $Az(n)$ is a set of non-overlapping dominoes whose union is $Az(n)$.
In \cite{EKLP92}, it was first proven that the number of domino tilings for the Aztec diamond $Az(n)$ is $2^{n(n+1)/2}$. This significant outcome, commonly referred to as the Aztec Diamond Theorem, has been proved in different ways. One such proof was presented by Eu and Fu in \cite{EF05} using Hankel determinants of the large and small Schröder numbers. This proof relies on a bijection connecting domino tilings of an Aztec diamond to non-intersecting lattice paths. The essential idea behind this bijection involves labeling the rows of $Az(n)$ with numbers from $1$ to $2n$ (from bottom to top). For each $1\leq i\leq n$, a path $p_i$ is defined, starting from the center of the left edge of the $i$-th row and ending at the center of the right edge of the same row. Whenever the path reaches a high domino, it traverses through the domino, switching from top to bottom or vice versa. For wide dominoes, the path moves horizontally. The dominoes above the highest path are exclusively wide. Therefore, to establish a correspondence between Aztec diamonds and Schröder paths, specific decorated dominoes are introduced. Each domino is depicted as follows: \par \bigskip

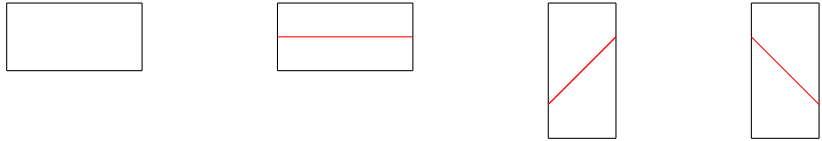
\begin{center}
\begin{tikzpicture}[domain=0:4, scale=0.9]
  \draw (-5,1) -- (-3,1);
  \draw (-5,1) -- (-5,0);
  \draw (-3,0) -- (-3,1);
  \draw (-3,0) -- (-5,0);
  
  \draw (-1,1) -- (1,1);
  \draw (-1,1) -- (-1,0);
  \draw (1,0) -- (1,1);
  \draw[red] (-1,0.5) -- (1,0.5);
  \draw (1,0) -- (-1,0);
 
  \draw (3,1) -- (4,1);
  \draw (3,1) -- (3,-1);
  \draw (4,-1) -- (3,-1);
  \draw (4,-1) -- (4,1);
  \draw[fill][red] (3,-0.5) -- (4,0.5);
  
  \draw (6,1) -- (7,1);
  \draw (6,1) -- (6,-1);
  \draw (7,-1) -- (6,-1);
  \draw (7,-1) -- (7,1);
  \draw[red] (6,0.5) -- (7,-0.5);
  \end{tikzpicture}
\end{center}

\subsection{Hankel matrices and lattice paths}\label{subsec:Hankel_path}
A well-known problem in linear algebra is the study of \emph{Hankel matrices}, which are $n\times n$ matrices $B=(b_{ij})$ such that, for $i\leq j$, we have $b_{ij}=b_{i+k,j-k}$, for all $k\in [j-i]$. In combinatorics specifically, we study Hankel matrices associated with a sequence $\mathcal{B}=\{b_n\}_{n\in\mathbb{N}}$, which are defined in \cite[p.53]{Ard15} as \par
\[
H_n(\mathcal{B}) =
\left( \begin{array}{cccc}
 b_0 & b_1 & \cdots & b_{n-1} \\ 
 b_1 & b_2 & \cdots & b_{n} \\
 \vdots & \vdots & \ddots & \vdots \\
 b_{n-1} & b_{n} & \cdots & b_{2n-2}
\end{array} \right)   \quad  \text{and} \quad  
H_n^{\prime}(\mathcal{B}) =
\left( \begin{array}{cccc}
 b_1 & b_2 & \cdots & b_n \\ 
 b_2 & b_3 & \cdots & b_{n+1} \\
 \vdots & \vdots & \ddots & \vdots \\
 b_n & b_{n+1} & \cdots & b_{2n-1}
\end{array} \right).
\]  \par \bigskip

In particular, the \textit{Hankel determinant} of order $n$ of $\mathcal{B}$ (sometimes called the \textit{Hankel transform}), is the determinant of the corresponding Hankel matrix of order $n$. The Hankel
determinant was first introduced in OEIS with Sloane’s sequence A055878 and later studied by Layman \cite{Lay01}. Many interesting properties of Hankel determinants are known. For instance, for the Catalan sequence, we obtain that \par
\[ 
\det \; H_n^{\prime}(C) =
\left| \begin{array}{cccc}
 C_1 & C_2 & \cdots & C_n \\ 
 C_2 & C_3 & \cdots & C_{n+1} \\
 \vdots & \vdots & \ddots & \vdots \\
 C_n & C_{n+1} & \cdots & C_{2n-1}
\end{array} \right| =1.
\] \par \bigskip

To find some properties for the determinants of the Hankel matrices associated with the sequence of Catalan numbers we can use the Lindström-Gessel-Viennot lemma. To state this lemma, we first introduce the concepts of routes and path matrices. Let $ G = (V, E) $ be a directed graph and let $ n \in \mathbb {N} $. A \emph{$n$-vertex} $ s=(s_1, \dots, s_n) $ of $ G $ is an $n$-tuple where $ s_i \in V $, for all $ i \in [n] $. If $ s = (s_1, \dots, s_n) $ and $ t = (t_1, \dots, t_n) $ are $n$-vertices of $ G $, define a \emph{$ n$-path} (also called \emph{path system}) $ P = (p_1, \dots, p_n) $ from $ s $ to $ t $, where $ p_i $ is a path from $ s_i $ to $ t_i $. 

If the paths $p_i$ and $ p_j $ do not have vertices in common, for $ 1 \leq i, j \leq n $, they are  \emph{vertex-disjoint paths}. The $n$-path $ R = (p_1, \dots, p_n) $ is a \emph{$n$-route} if $ p_i $ and $ p_j $ are vertex-disjoint, for all $ i \neq j $. Let $A$ be a commutative ring and let $w: E \to A$ be a weight function on the edges of a directed graph $G$.  
%The \emph{weight} of a path $p_i$, denoted as $w(p_i)$, is given by the product of the weights of the edges in the path:
%\[ w(p_i)=\prod_{e\in p_i} w(e).\]

The \emph{weight} of a $n$-route $R$, denoted as $w(R)$, is the product of the weights of the edges in the paths:
\[
w(R)=\prod_{i=1}^n \prod_{e\in p_i} w(e).
\]

\label{def:path matrix}The \emph{path matrix} $ M_{st} = (m_{ij})_ {1 \leq i, j \leq n}$ between two $n$-vertices $ s = (s_1, \dots, s_n)$ and $ t = (t_1, \dots, t_n)$ is defined as follows:
\[
m_{ij}=\sum_{p:s_i\to t_j}w(p).
\]

In particular, if $w(e)=1$ for all edges $e\in E$ (unweighted graph), the path matrix $ M_{st} = (m_{ij})_ {1 \leq i, j \leq n}$ from $ s = (s_1, \dots, s_n)$ to $ t = (t_1, \dots, t_n)$ is such that $m_{ij}$ is equal to the number of paths from $s_i$ to $t_j$.

\begin{lemma}[Lindström \cite{Lin73}, Gessel-Viennot {\cite[Theorem 1]{GV89}}] \label{lem:LGV}
Let $ G $ be an acyclic directed unweighted graph, and let $ s = (s_1, \dots, s_n) $ and $ t = (t_1, \dots, t_n) $ be two $n$-vectors of $ G $, then
\[
\det \; M_{st}=|\{R: R \; \text{is  a  $n$-route from $s$ to $t$}\}|,
\]
where $ M_{st} $ is the path matrix between $ s $ and $ t $.
\end{lemma}
An example relating the previous concepts is given by the graph below with $n$-vertices $ s = (s_1, \dots, s_n) $ and $ t = (t_1, \dots, t_n) $

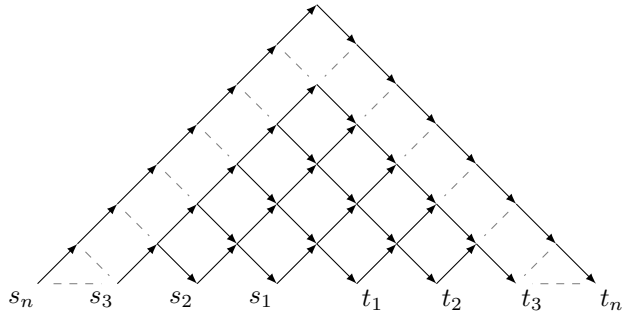
\begin{center}
\begin{tikzpicture}[y=.3cm, x=.3cm,font=\normalsize, scale=0.5,rotate=-45]

\draw[gray,dashed] (29,14) -- (26,11);
\draw[gray,dashed] (-4,-19) -- (-1,-16);
\draw[->, >=latex] (-5,15) -- (0,15);

\foreach \y in {-15,-10,...,10}
 {\draw[gray,dashed] (-4,\y) -- (-1,\y);}
 
\foreach \y in {-10,-5,...,15,20}
 {\draw[->, >=latex] (-5,\y-10) -- (-5,\y-5);} 
 
\foreach \x in {0,5,...,25}
 {\draw[gray,dashed] (\x,14) -- (\x,11);} 
 %::::::::::::::::::::::::::::::::::::::::::::::::::::::::::::::::: 

  \foreach \x in {0,5,...,25}
 {         \draw[->, >=latex] (\x,15) -- (\x+5,15);    
         } 
         
  \foreach \x in {0,5,...,20}
 {         \draw[->, >=latex] (\x,10) -- (\x+5,10);
           \draw[->, >=latex] (\x,5) -- (\x,10);
      
         } 
         
  \foreach \x in {0,5,...,15}
 {         \draw[->, >=latex] (\x,5) -- (\x+5,5);
           \draw[->, >=latex] (\x,0) -- (\x,5);
      
         } 
           
  \foreach \x in {0,5,10}
 {         \draw[->, >=latex] (\x,0) -- (\x+5,0);
           \draw[->, >=latex] (\x,-5) -- (\x,0);
      
         } 
         
  \foreach \x in {0,5}
 {         \draw[->, >=latex] (\x,-5) -- (\x+5,-5);
           \draw[->, >=latex] (\x,-10) -- (\x,-5);
      
         }          

 \draw[->, >=latex] (0,-10) -- (5,-10);
 \draw[->, >=latex] (0,-15) -- (0,-10);

\node (4) at (32,15) {$t_n$};
\node (5) at (27,10) {$t_3$};
\node (6) at (22,5) {$t_2$};
\node (7) at (17,0) {$t_1$};
\node (8) at (10,-7) {$s_1$};
\node (9) at (5,-12) {$s_2$};
\node (10) at (0,-17) {$s_3$};
\node (11) at (-5,-22) {$s_n$};
\end{tikzpicture}
\end{center}

According to the picture, the paths from $ s_i $ to $ t_k $ are Dyck paths, in this way, the path matrix is $ M_{st}=(m_{i j}) $, where $ m_{ij} = C_{i + j-1} $, which is the Hankel matrix $ H_n^{\prime} (C) $ associated to the sequence $ C $ of Catalan numbers. Additionally, only the $ n$-route is the following

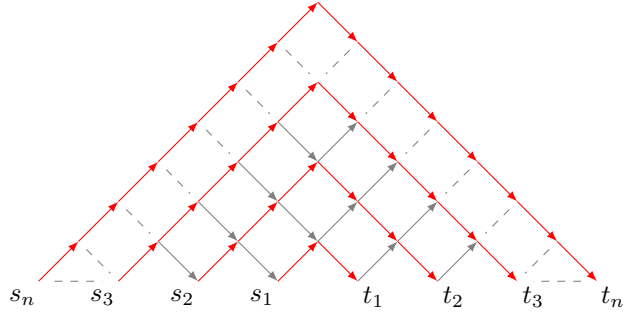
\begin{center}
\begin{tikzpicture}[y=.3cm, x=.3cm,font=\normalsize, scale=0.5,rotate=-45]

\draw[gray,dashed] (29,14) -- (26,11);
\draw[gray,dashed] (-4,-19) -- (-1,-16);

\foreach \y in {-15,-10,...,10}
 {\draw[gray,dashed] (-4,\y) -- (-1,\y);}
 
\foreach \y in {-10,-5,...,15,20}
 {\draw[->, >=latex,red] (-5,\y-10) -- (-5,\y-5);} 
 
\foreach \x in {0,5,...,25}
 {\draw[gray,dashed] (\x,14) -- (\x,11);} 
 %::::::::::::::::::::::::::::::::::::::::::::::::::::::::::::::::: 

  \foreach \x in {-5,0,5,...,25}
 {         \draw[->, >=latex,red] (\x,15) -- (\x+5,15);    
         } 
         
  \foreach \x in {0,5,...,20}
 {         \draw[->, >=latex,red] (\x,10) -- (\x+5,10);
         } 
  \foreach \x in {5,10,...,20}
 {   \draw[->, >=latex,gray] (\x,5) -- (\x,10);}
 \draw[->, >=latex,red] (0,5) -- (0,10);
         
  \foreach \x in {5,10,15}
 {         \draw[->, >=latex,red] (\x,5) -- (\x+5,5);} 
\draw[->, >=latex,gray] (0,5) -- (5,5);

\draw[->, >=latex,red] (0,0) -- (0,5);
\draw[->, >=latex,red] (5,0) -- (5,5);
\draw[->, >=latex,gray] (10,0) -- (10,5);
\draw[->, >=latex,gray] (15,0) -- (15,5);

  \foreach \x in {0,5}
 {         \draw[->, >=latex,gray] (\x,0) -- (\x+5,0);
           \draw[->, >=latex,red] (\x,-5) -- (\x,0);
      
         } 
      \draw[->, >=latex,red] (10,-5) -- (10,0);
      \draw[->, >=latex,red] (10,0) -- (15,0);
         
  \foreach \x in {0,5}
 {         \draw[->, >=latex,gray] (\x,-5) -- (\x+5,-5);
           \draw[->, >=latex,red] (\x,-10) -- (\x,-5);
      
         }          

 \draw[->, >=latex,gray] (0,-10) -- (5,-10);
 \draw[->, >=latex,red] (0,-15) -- (0,-10);

\node (4) at (32,15) {$t_n$};
\node (5) at (27,10) {$t_3$};
\node (6) at (22,5) {$t_2$};
\node (7) at (17,0) {$t_1$};
\node (8) at (10,-7) {$s_1$};
\node (9) at (5,-12) {$s_2$};
\node (10) at (0,-17) {$s_3$};
\node (11) at (-5,-22) {$s_n$};
\end{tikzpicture}
\end{center}

Therefore, using the Lemma~\ref{lem:LGV}, we conclude that $\det \; H_n^{\prime}(C) = 1$.

\subsection{Perfect matchings and snake graphs}\label{subsection:Perfect matchings and snake graphs}
In \cite{CS13} and \cite{CS18} the authors introduced a useful relationship between the continued fraction associated with a snake graph, the number of perfect matchings of such graph and a cluster algebra. These relationships became a tool to simplify some problems related to each of these objects, and they are one of the main interests of the present work. 

\label{def:perfect matchings} Let $G=(V, E)$ be a graph with $|V|=n$ and let $S\subset E$. The set $S$ is said to be a \emph{matching} on $G$, if these edges do not have vertices in common. If each vertex in $G$ is incident to an edge of $S$, then $S$ is a \emph{perfect matching} and $|S|=\frac{n}{2}$. We denote by $\text{Match} (G)$ the
set of perfect matchings of $G$. 

We consider a tile $ G $ as a graph that has the following form:
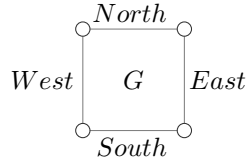
\begin{center}
\begin{tikzpicture}[y=.3cm, x=.3cm,font=\normalsize, scale=1]
\draw[gray] (0,0) -- (0,5);
\draw[gray] (5,0) -- (5,5);
\draw[gray] (5,5) -- (0,5);
\draw[gray] (0,0) -- (5,0);

\filldraw[fill=white!40,draw=black!80] (0,0) circle (3pt);
\filldraw[fill=white!40,draw=black!80] (0,5) circle (3pt);
\filldraw[fill=white!40,draw=black!80] (5,0) circle (3pt);
\filldraw[fill=white!40,draw=black!80] (5,5) circle (3pt);

\filldraw(4.6,5.8)     node[anchor=east] {$North$};
\filldraw(4.6,-0.8)     node[anchor=east] {$South$};
\filldraw(8.5,2.5)     node[anchor=east] {$East$};
\filldraw(0,2.5)     node[anchor=east] {$West$};

\filldraw(2.5,2.5)     node[anchor=center] {$G$};

\end{tikzpicture}
\end{center}

\label{def:snake graph} A \emph{snake graph} $\mathcal{G}$ is a connected planar graph consisting of a finite sequence of tiles $G_1, G_2, \dots, G_d$, with $d \geq 1$, in such a way that for each $i \in [d-1]$, the $G_i$ and $G_{i + 1}$ tiles share exactly one edge $e_i$, which is the North edge of $G_i$ and the South edge of $G_{i + 1}$, or the East edge of $G_i$ and the West edge of $G_{i + 1}$.  The edges $e_i$ that are contained in two tiles are called \emph{interior edges}. For a snake graph made up of tiles $G_1, G_2, \dots, G_d$, the interior edges are $e_1, e_2, \dots, e_{d-1}$. All of the other edges are called \emph{boundary edges}. 

Each snake graph $\mathcal{G}$ admits a sign function $f: E (\mathcal{G}) \longrightarrow \{+,-\}$ such that for each tile $ G_i $ in $\mathcal {G}$ the following conditions are met:

\begin{itemize}
     \item The north and west edges of $G_i$ have the same sign,
     \item The south and east edges of $G_i$ have the same sign,
     \item The sign of the north edge is opposite to the sign of the south edge.
\end{itemize}

These signs determine a continued fraction associated to $G$, encoding combinatorial properties of its perfect matchings. If $\mathcal{G}$ is a snake graph with $d$ tiles, then we define the sequence of signs
\[
[f(e_0), f (e_1), \cdots, f(e_{d-1}), f(e_d)],
\]
where $e_1, \cdots, e_{d-1}$ are the interior edges of $\mathcal{G}$, $e_0$ is the south edge of $G_1$ and $e_d$ is the north edge of $G_d$.
If $\epsilon \in \{+, - \}$ we have 
\begin{equation} \label{eq:sign sequence}
    [f(e_0), f(e_1), \cdots, f(e_{d-1}), f(e_d)] = [ \underbrace{\epsilon, \cdots, \epsilon}_{a_1}, \underbrace{- \epsilon, \cdots, - \epsilon}_{a_2}, \cdots, \underbrace {\pm \epsilon, \cdots, \pm \epsilon}_{a_n}],
\end{equation}

then the continued fraction $ [a_1, \cdots, a_n] $ associated to $ \mathcal{G}$ is obtained, where \label{def:continued fraction} 
\[
\displaystyle{[a_1, \cdots, a_n]}=\displaystyle{a_1+\frac{1}{a_2+\displaystyle{\frac{1}{a_3+\displaystyle{\frac{1}{\ddots+\displaystyle{\frac{1}{a_n}}}}}}}}.
\]

\label{def:snake graph continued fraction}Now let $[a_1, \dots, a_n]$ be a positive continued fraction (\ie, if each $a_i \in \mathbb{Z}_{>0}$). The snake graph associated to the continued fraction $[a_1, \dots, a_n]$ is denoted by $\mathcal{G}[a_1, \dots, a_n]$ and is the snake graph with $d=a_1+a_2+\cdots+ a_n - 1$ tiles determined by the sign sequence \ref{eq:sign sequence}. The following theorem relates the continued fractions of a snake graph with its perfect matchings.
%[Theorem 3.4]
\begin{theorem}[\cite{CS18}] \label{frac}
If $m(\mathcal{G})$ denotes the number of perfect matchings of $\mathcal{G}$ then
$$[a_1, a_2, \cdots, a_n]=\frac{m(\mathcal{G}[a_1, a_2, \cdots, a_n])} {m(\mathcal{G}[a_2, \cdots, a_n ])}.$$
\end{theorem}
As an application of Theorem~\ref{frac}, we note that the numerator of the continued fraction of $[a_1, \dots, a_n]$  corresponds to the number of perfect matchings in the snake graph $\mathcal{G}[a_1, \dots, a_n]$. Similarly, the numerator of $[a_n, \dots , a_1]$ represents the number of perfect matchings in $\mathcal{G}[a_n, \dots , a_1]$. Since these two snake graphs are identical up to a $180$-degree rotation, they must have the same number of perfect matchings. This leads to the following result.
\begin{theorem}[\cite{CS18}]\label{Continued fraction rotation}
The continued fractions $[a_1, \dots , a_n]$ and $[a_n, \dots , a_1]$ have the same numerator.
\end{theorem}
Observe that if $[a_1, \dots , a_n]$ is a continued fraction with  $a_n=1$, then $[a_1, \dots , a_n]=[a_1, \dots , a_{n-1}+1]$. This property, along with Theorem~\ref{Continued fraction rotation}, will be used in \S~\ref{sec:Some identities in not straight snake graphs}. 

\section{Perfect matchings, tilings and non-intersecting lattice paths}\label{sec:Perfect matchings, tilings and non-intersecting lattice paths}
In this section, we will construct an acyclic-directed graph corresponding to each snake graph. This construction will establish a bijection between the routes (non-intersecting lattice paths) in the resulting acyclic-directed graph and the perfect matchings of the snake graph. This construction was motivated by the bijection between domino tilings of an Aztec diamond 
and non-intersecting lattice paths (large Schröder paths satisfying some particular conditions) given by Eu and Fu in \cite{EF05}. %In \S~\ref{sec:Lattice_paths_and_perfect_matchings} we introduce the region $T(\mathcal{G})$ associated to a snake graph $\mathcal{G}$ and show a bijection between the domino tilings of $T(\mathcal{G})$ (called the snake graph tilings) and the perfect matchings of $\mathcal{G}$. Using decorated domino tiles and the operation of edge contraction, we define the triangular snake graph $\mathcal{T}_\mathcal{G}$ and associate each perfect matching to a $k$-route of $\mathcal{T}_\mathcal{G}$. In \S~\ref{sec:ladder_graphs_hankel_and_general_formula}, we explore the connections between Fibonacci numbers and path matrices associated to ladder graphs. %Finally, in \S~\ref{sec:triangular_snake_graph_poset}, we define the poset of $k$-routes of $\mathcal{T}_\mathcal{G}$ and show how to associate it to an $F$-polynomial. In particular, we study the case of Dynkin quivers $Q$ of type $A$ and associate a triangular snake graph for every indecomposable $kQ$-module.

\subsection{Tilings and snake graphs}\label{sec:Lattice_paths_and_perfect_matchings}
Without loss of generality and for an easier exposition, let us consider a snake graph $\mathcal{G}$ formed by a sequence of unit squares $G_1, G_2,\dots, G_{d}$, where $d\geq 1$.

\begin{definition}\label{def:Snake graph cover}
The \emph{cover} $T(G_i)$ of tile $G_i$ in a snake graph $\mathcal{G}$ is the union of all unit squares centered at the vertices of $G_i$. The \emph{snake graph cover} $T(\mathcal{G})$ of a snake graph $\mathcal{G}$ is the union of the covers of all tiles $G_i$ in $\mathcal{G}$.

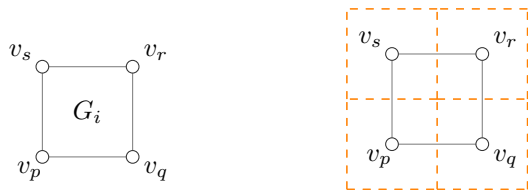
\begin{figure}[ht]
\begin{center}
\begin{tikzpicture}[y=.3cm, x=.3cm,font=\normalsize, scale=0.8]
\draw[gray] (0,0) -- (0,5);
\draw[gray] (5,0) -- (5,5);
\draw[gray] (5,5) -- (0,5);
\draw[gray] (0,0) -- (5,0);

\filldraw[fill=white!40,draw=black!80] (0,0) circle (3pt);
\filldraw[fill=white!40,draw=black!80] (0,5) circle (3pt);
\filldraw[fill=white!40,draw=black!80] (5,0) circle (3pt);
\filldraw[fill=white!40,draw=black!80] (5,5) circle (3pt);

\filldraw(0.5,-0.8)     node[anchor=east] {$v_{p}$};
\filldraw(7.5,-0.8)     node[anchor=east] {$v_{q}$};
\filldraw(7.5,5.8)     node[anchor=east] {$v_{r}$};
\filldraw(0,5.8)     node[anchor=east] {$v_{s}$};

\filldraw(2.5,2.5)     node[anchor=center] {$G_i$};

\end{tikzpicture}
\hspace{2cm} \begin{tikzpicture}[y=.3cm, x=.3cm,font=\normalsize, scale=0.8]
\draw[gray] (0,0) -- (0,5);
\draw[gray] (5,0) -- (5,5);
\draw[gray] (5,5) -- (0,5);
\draw[gray] (0,0) -- (5,0);

\filldraw[fill=white!40,draw=black!80] (0,0) circle (3pt);
\filldraw[fill=white!40,draw=black!80] (0,5) circle (3pt);
\filldraw[fill=white!40,draw=black!80] (5,0) circle (3pt);
\filldraw[fill=white!40,draw=black!80] (5,5) circle (3pt);

\filldraw(0.5,-0.8)     node[anchor=east] {$v_{p}$};
\filldraw(7.5,-0.8)     node[anchor=east] {$v_{q}$};
\filldraw(7.5,5.8)     node[anchor=east] {$v_{r}$};
\filldraw(0,5.8)     node[anchor=east] {$v_{s}$};

\draw[orange, dashed] (-2.5,-2.5) -- (-2.5,2.5);
\draw[orange, dashed] (2.5,-2.5) -- (2.5,2.5);
\draw[orange, dashed] (2.5,2.5) -- (-2.5,2.5);
\draw[orange, dashed] (-2.5,-2.5) -- (2.5,-2.5);

\draw[orange, dashed] (7.5,-2.5) -- (7.5,2.5);
\draw[orange, dashed] (2.5,-2.5) -- (7.5,-2.5);

\draw[orange, dashed] (-2.5,2.5) -- (-2.5,7.5);
\draw[orange, dashed] (2.5,2.5) -- (2.5,7.5);
\draw[orange, dashed] (2.5,7.5) -- (-2.5,7.5);

\draw[orange, dashed] (7.5,2.5) -- (7.5,7.5);
\draw[orange, dashed] (7.5,7.5) -- (2.5,7.5);
\draw[orange, dashed] (2.5,2.5) -- (7.5,2.5);

\draw[orange, dashed] (-2.5,-2.5) -- (-2.5,2.5);
\draw[orange, dashed] (2.5,-2.5) -- (2.5,2.5);
\draw[orange, dashed] (2.5,2.5) -- (-2.5,2.5);
\draw[orange, dashed] (-2.5,-2.5) -- (2.5,-2.5);

\draw[orange, dashed] (7.5,-2.5) -- (7.5,2.5);
\draw[orange, dashed] (2.5,-2.5) -- (7.5,-2.5);

\draw[orange, dashed] (-2.5,2.5) -- (-2.5,7.5);
\draw[orange, dashed] (2.5,2.5) -- (2.5,7.5);
\draw[orange, dashed] (2.5,7.5) -- (-2.5,7.5);

\draw[orange, dashed] (7.5,2.5) -- (7.5,7.5);
\draw[orange, dashed] (7.5,7.5) -- (2.5,7.5);
\draw[orange, dashed] (2.5,2.5) -- (7.5,2.5);
\end{tikzpicture}
\end{center}
\caption{From left to right, tile $G_i$ of a snake graph $\mathcal{G}$ and its cover $T(G_i)$.}
\end{figure}
\end{definition}
The unit square centered at the southwest vertex of the tile $G_1$ is called the \emph{initial tile} of $T(\mathcal{G})$.
\begin{definition}\label{def:snake graph tiling}
A \emph{domino tile} is formed by the union of any two unit squares sharing an edge. A \emph{snake graph tiling} of $\mathcal{G}$ is a set of non-overlapping domino tiles whose union forms $T(\mathcal{G})$. We denote the set of all snake graph tilings of $\mathcal{G}$ by $\text{Til}(\mathcal{G})$.
\end{definition}

\begin{example}\label{ex:SnakeGraphTilings} Consider the snake graph 
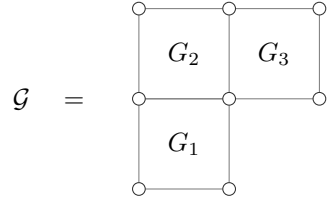
\begin{center}
\begin{tikzpicture}[scale=1.2]

\filldraw(-1,1) node[anchor=center] {$\mathcal{G} \quad =$};

\draw[gray] (0,0) -- (0,1);
\draw[gray] (1,0) -- (1,1);
\draw[gray] (1,1) -- (0,1);
\draw[gray] (0,0) -- (1,0);

\filldraw(0.5,0.5)     node[anchor=center] {$G_1$};

\draw[gray] (0,1) -- (0,2);
\draw[gray] (1,1) -- (1,2);
\draw[gray] (1,2) -- (0,2);
\draw[gray] (0,1) -- (1,1);

\filldraw(0.5,1.5)     node[anchor=center] {$G_{2}$};

\draw[gray] (2,1) -- (2,2);
\draw[gray] (1,2) -- (2,2);
\draw[gray] (2,1) -- (1,1);

\filldraw(1.5,1.5)     node[anchor=center] {$G_{3}$};

\filldraw[fill=white!40,draw=black!80] (0,0) circle (2pt);
\filldraw[fill=white!40,draw=black!80] (1,0) circle (2pt);
\filldraw[fill=white!40,draw=black!80] (0,1) circle (2pt);
\filldraw[fill=white!40,draw=black!80] (1,1) circle (2pt);
\filldraw[fill=white!40,draw=black!80] (2,1) circle (2pt);
\filldraw[fill=white!40,draw=black!80] (0,2) circle (2pt);
\filldraw[fill=white!40,draw=black!80] (1,2) circle (2pt);
\filldraw[fill=white!40,draw=black!80] (2,2) circle (2pt);
\end{tikzpicture}
\end{center}

The snake graph cover $T(\mathcal{G})$ corresponding to it is
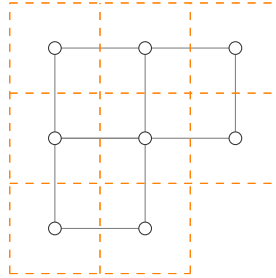
\begin{center}
\begin{tikzpicture}[scale=1.2]

\draw[gray] (0,0) -- (0,1);
\draw[gray] (1,0) -- (1,1);
\draw[gray] (1,1) -- (0,1);
\draw[gray] (0,0) -- (1,0);

%\filldraw(0.5,0.5)     node[anchor=center] {$G_1$};

\draw[gray] (0,1) -- (0,2);
\draw[gray] (1,1) -- (1,2);
\draw[gray] (1,2) -- (0,2);
\draw[gray] (0,1) -- (1,1);

%\filldraw(0.5,1.5)     node[anchor=center] {$G_{2}$};

\draw[gray] (2,1) -- (2,2);
\draw[gray] (1,2) -- (2,2);
\draw[gray] (2,1) -- (1,1);

%\filldraw(1.5,1.5)     node[anchor=center] {$G_{3}$};

\filldraw[fill=white!40,draw=black!80] (0,0) circle (2pt);
\filldraw[fill=white!40,draw=black!80] (1,0) circle (2pt);
\filldraw[fill=white!40,draw=black!80] (0,1) circle (2pt);
\filldraw[fill=white!40,draw=black!80] (1,1) circle (2pt);
\filldraw[fill=white!40,draw=black!80] (2,1) circle (2pt);
\filldraw[fill=white!40,draw=black!80] (0,2) circle (2pt);
\filldraw[fill=white!40,draw=black!80] (1,2) circle (2pt);
\filldraw[fill=white!40,draw=black!80] (2,2) circle (2pt);

\draw[orange, dashed] (-0.5,-0.5) -- (1.5,-0.5);
\draw[orange, dashed] (-0.5,0.5) -- (2.5,0.5);
\draw[orange, dashed] (-0.5,1.5) -- (2.5,1.5);
\draw[orange, dashed] (-0.5,2.5) -- (2.5,2.5);

\draw[orange, dashed] (-0.5,-0.5) -- (-0.5,2.5);
\draw[orange, dashed] (0.5,-0.5) -- (0.5,2.5);
\draw[orange, dashed] (1.5,-0.5) -- (1.5,2.5);
\draw[orange, dashed] (2.5,0.5) -- (2.5,2.5);

\draw[orange, dashed] (-0.5,-0.5) -- (1.5,-0.5);
\draw[orange, dashed] (-0.5,0.5) -- (2.5,0.5);
\draw[orange, dashed] (-0.5,1.5) -- (2.5,1.5);
\draw[orange, dashed] (-0.5,2.5) -- (2.5,2.5);

\draw[orange, dashed] (-0.5,-0.5) -- (-0.5,2.5);
\draw[orange, dashed] (0.5,-0.5) -- (0.5,2.5);
\draw[orange, dashed] (1.5,-0.5) -- (1.5,2.5);
\draw[orange, dashed] (2.5,0.5) -- (2.5,2.5);

\end{tikzpicture}
\end{center}

There exist precisely four snake graph tilings of $\mathcal{G}$
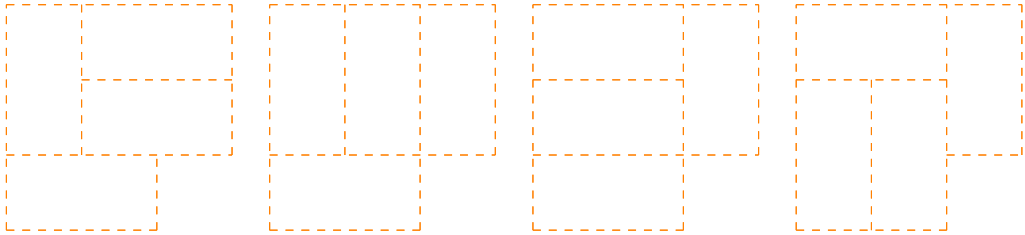
\begin{center}
\begin{tikzpicture}[scale=1]

% 1

\draw[orange, dashed] (-0.5,-0.5) -- (1.5,-0.5);
\draw[orange, dashed] (-0.5,0.5) -- (2.5,0.5);
\draw[orange, dashed] (0.5,1.5) -- (2.5,1.5);
\draw[orange, dashed] (-0.5,2.5) -- (2.5,2.5);

\draw[orange, dashed] (-0.5,-0.5) -- (-0.5,2.5);
\draw[orange, dashed] (0.5,0.5) -- (0.5,2.5);
\draw[orange, dashed] (1.5,-0.5) -- (1.5,0.5);
\draw[orange, dashed] (2.5,0.5) -- (2.5,2.5);

\draw[orange, dashed] (-0.5,-0.5) -- (1.5,-0.5);
\draw[orange, dashed] (-0.5,0.5) -- (2.5,0.5);
\draw[orange, dashed] (0.5,1.5) -- (2.5,1.5);
\draw[orange, dashed] (-0.5,2.5) -- (2.5,2.5);

\draw[orange, dashed] (-0.5,-0.5) -- (-0.5,2.5);
\draw[orange, dashed] (0.5,0.5) -- (0.5,2.5);
\draw[orange, dashed] (1.5,-0.5) -- (1.5,0.5);
\draw[orange, dashed] (2.5,0.5) -- (2.5,2.5);

% y_3

\draw[orange, dashed] (3,-0.5) -- (5,-0.5);
\draw[orange, dashed] (3,0.5) -- (6,0.5);
\draw[orange, dashed] (3,2.5) -- (6,2.5);

\draw[orange, dashed] (3,-0.5) -- (3,2.5);
\draw[orange, dashed] (4,0.5) -- (4,2.5);
\draw[orange, dashed] (5,-0.5) -- (5,2.5);
\draw[orange, dashed] (6,0.5) -- (6,2.5);

\draw[orange, dashed] (3,-0.5) -- (5,-0.5);
\draw[orange, dashed] (3,0.5) -- (6,0.5);
\draw[orange, dashed] (3,2.5) -- (6,2.5);

\draw[orange, dashed] (3,-0.5) -- (3,2.5);
\draw[orange, dashed] (4,0.5) -- (4,2.5);
\draw[orange, dashed] (5,-0.5) -- (5,2.5);
\draw[orange, dashed] (6,0.5) -- (6,2.5);

% y_2y_3

\draw[orange, dashed] (6.5,-0.5) -- (8.5,-0.5);
\draw[orange, dashed] (6.5,0.5) -- (9.5,0.5);
\draw[orange, dashed] (6.5,1.5) -- (8.5,1.5);
\draw[orange, dashed] (6.5,2.5) -- (9.5,2.5);

\draw[orange, dashed] (6.5,-0.5) -- (6.5,2.5);
\draw[orange, dashed] (8.5,-0.5) -- (8.5,2.5);
\draw[orange, dashed] (9.5,0.5) -- (9.5,2.5);

\draw[orange, dashed] (6.5,-0.5) -- (8.5,-0.5);
\draw[orange, dashed] (6.5,0.5) -- (9.5,0.5);
\draw[orange, dashed] (6.5,1.5) -- (8.5,1.5);
\draw[orange, dashed] (6.5,2.5) -- (9.5,2.5);

\draw[orange, dashed] (6.5,-0.5) -- (6.5,2.5);
\draw[orange, dashed] (8.5,-0.5) -- (8.5,2.5);
\draw[orange, dashed] (9.5,0.5) -- (9.5,2.5);

% y_1y_2y_3

\draw[orange, dashed] (10,-0.5) -- (12,-0.5);
\draw[orange, dashed] (12,0.5) -- (13,0.5);
\draw[orange, dashed] (10,1.5) -- (12,1.5);
\draw[orange, dashed] (10,2.5) -- (13,2.5);

\draw[orange, dashed] (10,-0.5) -- (10,2.5);
\draw[orange, dashed] (11,-0.5) -- (11,1.5);
\draw[orange, dashed] (12,-0.5) -- (12,2.5);
\draw[orange, dashed] (13,0.5) -- (13,2.5);

\draw[orange, dashed] (10,-0.5) -- (12,-0.5);
\draw[orange, dashed] (12,0.5) -- (13,0.5);
\draw[orange, dashed] (10,1.5) -- (12,1.5);
\draw[orange, dashed] (10,2.5) -- (13,2.5);

\draw[orange, dashed] (10,-0.5) -- (10,2.5);
\draw[orange, dashed] (11,-0.5) -- (11,1.5);
\draw[orange, dashed] (12,-0.5) -- (12,2.5);
\draw[orange, dashed] (13,0.5) -- (13,2.5);
\end{tikzpicture}
\end{center}
\end{example}

\begin{proposition}\label{prop:snake_graph_tilings_and_perfect_matchings}
Let $\mathcal{G}$ be a snake graph. There exists a bijection between the set $\text{Til}(\mathcal{G})$ of snake graph tilings of $\mathcal{G}$ and the set $\text{Match}(\mathcal{G})$ of perfect matchings of $\mathcal{G}$. 
\end{proposition}

\begin{proof}
Given $P \in \text{Match}(\mathcal{G})$, we associate a snake graph tiling as follows: For an edge $e_{ij}$ in $P$ connecting the vertices $v_i$ and $v_j$ in $\mathcal{G}$, we create a domino tile $D_{ij}$ associated to $e_{ij}$ by linking the unit squares centered at $v_i$ and $v_j$. If two distinct edges $e_{ij}$ and $e_{kl}$ exist in $P$, the intersection of the domino tiles $D_{ij}$ and $D_{kl}$ is empty. Otherwise, the existence of a common vertex between $e_{ij}$ and $e_{kl}$ would contradict the assumption of $P$ being a matching of $\mathcal{G}$. Moreover,
\[
\bigcup_{e_{ij}} D_{ij} = T(\mathcal{G}),
\]
where the union is taken over all edges $e_{ij}$ in $P$. This holds because for each edge $e_{kl}$ in $P$, there exists a tile $G_i$ in $\mathcal{G}$ such that $e_{kl}$ is an edge of $G_i$, and thus $D_{kl}\subset T(G_i) \subset T(\mathcal{G})$. Furthermore, as $P$ is a perfect matching, for each cover $T(G_i)$ there exist edges $e_{pq}$ in $P$ such that $T(G_i)\subseteq \bigcup_{e_{pq}}D_{ pq}$. 
Conversely, considering a snake graph tiling of $\mathcal{G}$, for every vertex $v_p$ of $\mathcal{G}$, there exists a unique domino tile such that $v_p$ is the center of one of the unit squares contained in that domino tile. Let $v_q$ be the center of the other unit square. We construct a perfect matching $P$ where the edge $e_{pq}$ from $v_p$ to $v_q$ is an edge of $P$. 

This proves the bijection between snake graph tilings and perfect matchings in $\mathcal{G}$.
\end{proof}

\begin{example}\label{ex:snake_graph_tiling_matching} Let $\mathcal{G}$ the snake graph in the Example~\ref{ex:SnakeGraphTilings}. Then, the perfect matchings associated to the four snake graph tilings are

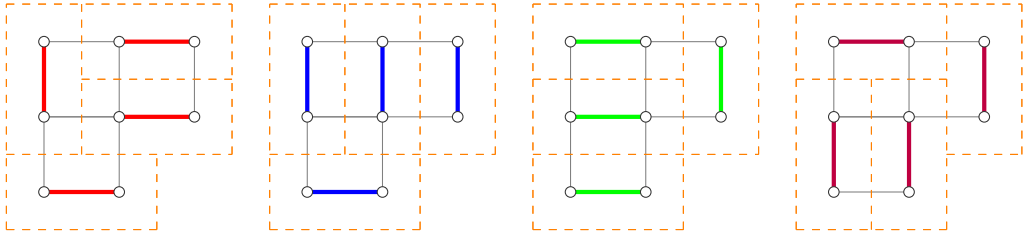
\begin{center}
\begin{tikzpicture}[scale=1]

% 1

\draw[gray] (0,0) -- (0,1);
\draw[gray] (1,0) -- (1,1);
\draw[gray] (1,1) -- (0,1);
\draw[red, ultra thick] (0,0) -- (1,0);

%\filldraw(0.5,0.5)     node[anchor=center] {$G_1$};

\draw[red, ultra thick] (0,1) -- (0,2);
\draw[gray] (1,1) -- (1,2);
\draw[gray] (1,2) -- (0,2);
\draw[gray] (0,1) -- (1,1);

%\filldraw(0.5,1.5)     node[anchor=center] {$G_{2}$};

\draw[gray] (2,1) -- (2,2);
\draw[red, ultra thick] (1,2) -- (2,2);
\draw[red, ultra thick] (2,1) -- (1,1);

%\filldraw(1.5,1.5)     node[anchor=center] {$G_{3}$};

\draw[orange, dashed] (-0.5,-0.5) -- (1.5,-0.5);
\draw[orange, dashed] (-0.5,0.5) -- (2.5,0.5);
\draw[orange, dashed] (0.5,1.5) -- (2.5,1.5);
\draw[orange, dashed] (-0.5,2.5) -- (2.5,2.5);

\draw[orange, dashed] (-0.5,-0.5) -- (-0.5,2.5);
\draw[orange, dashed] (0.5,0.5) -- (0.5,2.5);
\draw[orange, dashed] (1.5,-0.5) -- (1.5,0.5);
\draw[orange, dashed] (2.5,0.5) -- (2.5,2.5);

\draw[orange, dashed] (-0.5,-0.5) -- (1.5,-0.5);
\draw[orange, dashed] (-0.5,0.5) -- (2.5,0.5);
\draw[orange, dashed] (0.5,1.5) -- (2.5,1.5);
\draw[orange, dashed] (-0.5,2.5) -- (2.5,2.5);

\draw[orange, dashed] (-0.5,-0.5) -- (-0.5,2.5);
\draw[orange, dashed] (0.5,0.5) -- (0.5,2.5);
\draw[orange, dashed] (1.5,-0.5) -- (1.5,0.5);
\draw[orange, dashed] (2.5,0.5) -- (2.5,2.5);

\filldraw[fill=white!40,draw=black!80] (0,0) circle (2pt);
\filldraw[fill=white!40,draw=black!80] (1,0) circle (2pt);
\filldraw[fill=white!40,draw=black!80] (0,1) circle (2pt);
\filldraw[fill=white!40,draw=black!80] (1,1) circle (2pt);
\filldraw[fill=white!40,draw=black!80] (2,1) circle (2pt);
\filldraw[fill=white!40,draw=black!80] (0,2) circle (2pt);
\filldraw[fill=white!40,draw=black!80] (1,2) circle (2pt);
\filldraw[fill=white!40,draw=black!80] (2,2) circle (2pt);

% y_3

\draw[gray] (3.5,0) -- (3.5,1);
\draw[gray] (4.5,0) -- (4.5,1);
\draw[gray] (4.5,1) -- (3.5,1);
\draw[blue, ultra thick] (3.5,0) -- (4.5,0);

%\filldraw(4,0.5)     node[anchor=center] {$G_1$};

\draw[blue, ultra thick] (3.5,1) -- (3.5,2);
\draw[blue, ultra thick] (4.5,1) -- (4.5,2);
\draw[gray] (4.5,2) -- (3.5,2);
\draw[gray] (3.5,1) -- (4.5,1);

%\filldraw(4,1.5)     node[anchor=center] {$G_{2}$};

\draw[blue, ultra thick] (5.5,1) -- (5.5,2);
\draw[gray] (4.5,2) -- (5.5,2);
\draw[gray] (5.5,1) -- (4.5,1);

%\filldraw(5,1.5)     node[anchor=center] {$G_{3}$};

\draw[orange, dashed] (3,-0.5) -- (5,-0.5);
\draw[orange, dashed] (3,0.5) -- (6,0.5);
\draw[orange, dashed] (3,2.5) -- (6,2.5);

\draw[orange, dashed] (3,-0.5) -- (3,2.5);
\draw[orange, dashed] (4,0.5) -- (4,2.5);
\draw[orange, dashed] (5,-0.5) -- (5,2.5);
\draw[orange, dashed] (6,0.5) -- (6,2.5);

\draw[orange, dashed] (3,-0.5) -- (5,-0.5);
\draw[orange, dashed] (3,0.5) -- (6,0.5);
\draw[orange, dashed] (3,2.5) -- (6,2.5);

\draw[orange, dashed] (3,-0.5) -- (3,2.5);
\draw[orange, dashed] (4,0.5) -- (4,2.5);
\draw[orange, dashed] (5,-0.5) -- (5,2.5);
\draw[orange, dashed] (6,0.5) -- (6,2.5);

\filldraw[fill=white!40,draw=black!80] (3.5,0) circle (2pt);
\filldraw[fill=white!40,draw=black!80] (4.5,0) circle (2pt);
\filldraw[fill=white!40,draw=black!80] (3.5,1) circle (2pt);
\filldraw[fill=white!40,draw=black!80] (4.5,1) circle (2pt);
\filldraw[fill=white!40,draw=black!80] (5.5,1) circle (2pt);
\filldraw[fill=white!40,draw=black!80] (3.5,2) circle (2pt);
\filldraw[fill=white!40,draw=black!80] (4.5,2) circle (2pt);
\filldraw[fill=white!40,draw=black!80] (5.5,2) circle (2pt);

% y_2y_3

\draw[gray] (7,0) -- (7,1);
\draw[gray] (8,0) -- (8,1);
\draw[green, ultra thick] (8,1) -- (7,1);
\draw[green, ultra thick] (7,0) -- (8,0);

%\filldraw(7.5,0.5)     node[anchor=center] {$G_1$};

\draw[gray] (7,1) -- (7,2);
\draw[gray] (8,1) -- (8,2);
\draw[green, ultra thick] (8,2) -- (7,2);

%\filldraw(7.5,1.5)     node[anchor=center] {$G_{2}$};

\draw[green, ultra thick] (9,1) -- (9,2);
\draw[gray] (8,2) -- (9,2);
\draw[gray] (9,1) -- (8,1);

%\filldraw(8.5,1.5)     node[anchor=center] {$G_{3}$};

\draw[orange, dashed] (6.5,-0.5) -- (8.5,-0.5);
\draw[orange, dashed] (6.5,0.5) -- (9.5,0.5);
\draw[orange, dashed] (6.5,1.5) -- (8.5,1.5);
\draw[orange, dashed] (6.5,2.5) -- (9.5,2.5);

\draw[orange, dashed] (6.5,-0.5) -- (6.5,2.5);
\draw[orange, dashed] (8.5,-0.5) -- (8.5,2.5);
\draw[orange, dashed] (9.5,0.5) -- (9.5,2.5);

\draw[orange, dashed] (6.5,-0.5) -- (8.5,-0.5);
\draw[orange, dashed] (6.5,0.5) -- (9.5,0.5);
\draw[orange, dashed] (6.5,1.5) -- (8.5,1.5);
\draw[orange, dashed] (6.5,2.5) -- (9.5,2.5);

\draw[orange, dashed] (6.5,-0.5) -- (6.5,2.5);
\draw[orange, dashed] (8.5,-0.5) -- (8.5,2.5);
\draw[orange, dashed] (9.5,0.5) -- (9.5,2.5);

\filldraw[fill=white!40,draw=black!80] (7,0) circle (2pt);
\filldraw[fill=white!40,draw=black!80] (8,0) circle (2pt);
\filldraw[fill=white!40,draw=black!80] (7,1) circle (2pt);
\filldraw[fill=white!40,draw=black!80] (8,1) circle (2pt);
\filldraw[fill=white!40,draw=black!80] (9,1) circle (2pt);
\filldraw[fill=white!40,draw=black!80] (7,2) circle (2pt);
\filldraw[fill=white!40,draw=black!80] (8,2) circle (2pt);
\filldraw[fill=white!40,draw=black!80] (9,2) circle (2pt);

% y_1y_2y_3

\draw[purple, ultra thick] (10.5,0) -- (10.5,1);
\draw[purple, ultra thick] (11.5,0) -- (11.5,1);
\draw[gray] (11.5,1) -- (10.5,1);
\draw[gray] (10.5,0) -- (11.5,0);

%\filldraw(9.5,0.5)     node[anchor=center] {$G_1$};

\draw[gray] (10.5,1) -- (10.5,2);
\draw[gray] (11.5,1) -- (11.5,2);
\draw[purple, ultra thick] (11.5,2) -- (10.5,2);
\draw[gray] (10.5,1) -- (11.5,1);

%\filldraw(9.5,1.5)     node[anchor=center] {$G_{2}$};

\draw[purple, ultra thick] (12.5,1) -- (12.5,2);
\draw[gray] (11.5,2) -- (12.5,2);
\draw[gray] (12.5,1) -- (11.5,1);

%\filldraw(10.5,1.5)     node[anchor=center] {$G_{3}$};

\draw[orange, dashed] (10,-0.5) -- (12,-0.5);
\draw[orange, dashed] (12,0.5) -- (13,0.5);
\draw[orange, dashed] (10,1.5) -- (12,1.5);
\draw[orange, dashed] (10,2.5) -- (13,2.5);

\draw[orange, dashed] (10,-0.5) -- (10,2.5);
\draw[orange, dashed] (11,-0.5) -- (11,1.5);
\draw[orange, dashed] (12,-0.5) -- (12,2.5);
\draw[orange, dashed] (13,0.5) -- (13,2.5);

\draw[orange, dashed] (10,-0.5) -- (12,-0.5);
\draw[orange, dashed] (12,0.5) -- (13,0.5);
\draw[orange, dashed] (10,1.5) -- (12,1.5);
\draw[orange, dashed] (10,2.5) -- (13,2.5);

\draw[orange, dashed] (10,-0.5) -- (10,2.5);
\draw[orange, dashed] (11,-0.5) -- (11,1.5);
\draw[orange, dashed] (12,-0.5) -- (12,2.5);
\draw[orange, dashed] (13,0.5) -- (13,2.5);

\filldraw[fill=white!40,draw=black!80] (10.5,0) circle (2pt);
\filldraw[fill=white!40,draw=black!80] (11.5,0) circle (2pt);
\filldraw[fill=white!40,draw=black!80] (10.5,1) circle (2pt);
\filldraw[fill=white!40,draw=black!80] (11.5,1) circle (2pt);
\filldraw[fill=white!40,draw=black!80] (12.5,1) circle (2pt);
\filldraw[fill=white!40,draw=black!80] (10.5,2) circle (2pt);
\filldraw[fill=white!40,draw=black!80] (11.5,2) circle (2pt);
\filldraw[fill=white!40,draw=black!80] (12.5,2) circle (2pt);
\end{tikzpicture}
\end{center}

\end{example}

\subsection{Decorated snake graph tilings and paths}
We intend to establish a correspondence between routes and snake graph tilings. To achieve this, we define a grid in $(\mathbb{R}_{\geq 0})^2$, placing black points at coordinates $(2n, 2m+0.5)$ and $(2n+1, 2m+1.5)$, where $n$ and $m$ are elements of $\mathbb{N}$.

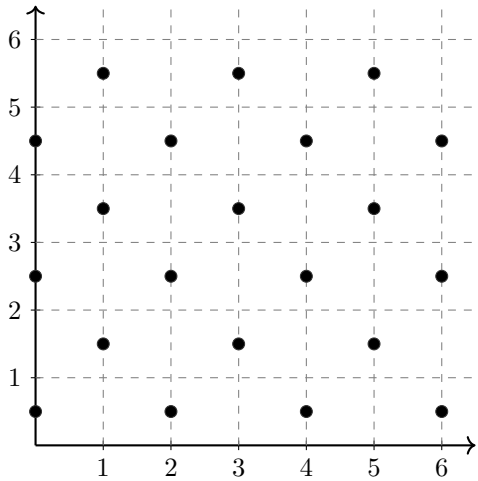
\begin{center}
\begin{tikzpicture}[scale=0.9]
\draw[black,thick,->] (0,0)--(6.5,0) node[right] {}; % Eje x
% Enumeración del eje x
\foreach \x/\xtext in {1/1, 2/2, 3/3, 4/4, 5/5, 6/6} 
\draw[shift={(\x,0)},black] (0pt,2pt)--(0pt,-2pt) node[below] {$\xtext$};
% Enumeración del eje y
\foreach \y/\ytext in {1/1, 2/2, 3/3, 4/4, 5/5, 6/6} 
\draw[shift={(0,\y)},black] (2pt,0pt)--(-2pt,0pt) node[left] {$\ytext$};
\draw[black,thick,->] (0,0)--(0,6.5) node[left,above] {}; % Eje y
%gray grid
\foreach \x in {1, 2, 3, 4, 5, 6} 
\draw[gray, dashed] (\x,0)--(\x,6.5);
\foreach \x in {1, 2, 3, 4, 5, 6} 
\draw[gray, dashed] (0,\x)--(6.5,\x);
%\black points
\foreach \x in {0, 2, 4, 6} 
\filldraw[fill=black!100,draw=black!80] (\x,0.5) circle (2.5pt)    node[anchor=north] {\small };
\foreach \x in {0, 2, 4, 6}
\filldraw[fill=black!100,draw=black!80] (\x,2.5) circle (2.5pt)    node[anchor=north] {\small };
\foreach \x in {0, 2, 4, 6}
\filldraw[fill=black!100,draw=black!80] (\x,4.5) circle (2.5pt)    node[anchor=north] {\small };

\foreach \x in {1, 3, 5}
\filldraw[fill=black!100,draw=black!80] (\x,1.5) circle (2.5pt)    node[anchor=north] {\small };
\foreach \x in {1, 3, 5}
\filldraw[fill=black!100,draw=black!80] (\x,3.5) circle (2.5pt)    node[anchor=north] {\small };
\foreach \x in {1, 3, 5}
\filldraw[fill=black!100,draw=black!80] (\x,5.5) circle (2.5pt)    node[anchor=north] {\small };
\end{tikzpicture}
\end{center}

Subsequently, we position the snake graph cover $T(\mathcal{G})$ of $\mathcal{G}$ on the grid, ensuring that the southwest vertex of the initial tile of $T(\mathcal{G})$ aligns with $(0,0)$, thus establishing a decoration using black points within the snake graph cover. We can observe that the snake graph cover is decorated with the following two unit squares

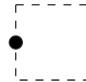
\begin{center}
\begin{tikzpicture}[domain=0:4, scale=1]

  \draw[dashed] (-5,1) -- (-4,1);
  \draw[dashed] (-5,1) -- (-5,0);
  \draw[dashed] (-4,0) -- (-4,1);
  \draw[dashed] (-4,0) -- (-5,0);
  
  \draw[dashed] (1,1) -- (2,1);
  \draw[dashed] (1,1) -- (1,0);
  \draw[dashed] (2,0) -- (2,1);
  \draw[dashed] (2,0) -- (1,0);
 
 \filldraw[fill=black!100,draw=black!80] (-4,0.5) circle (2.5pt) node[anchor=north] {\small } node[above=13pt] {\textbf{Right-decorated square~~~~~~~}};
\filldraw[fill=black!100,draw=black!80] (1,0.5) circle (2.5pt) node[anchor=north] {\small } node[above=13pt] {\textbf{~~~~~~~Left-decorated square}};
  \end{tikzpicture}
\end{center} 
in such a way that the initial tile of $T(\mathcal{G})$ is left-decorated, and each square that shares an edge with a left-decorated square is right-decorated, and vice versa. In this way, any domino tile in a snake graph tiling will be decorated in one of the following ways

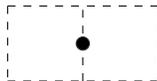
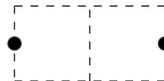
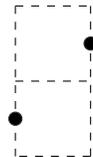
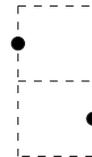
\begin{center}
\begin{tikzpicture}[domain=0:4, scale=1]

  \draw[dashed] (-5,1) -- (-3,1);
  \draw[dashed] (-5,1) -- (-5,0);
  \draw[dashed] (-3,0) -- (-3,1);
  \draw[dashed] (-4,0) -- (-4,1);
  \draw[dashed] (-3,0) -- (-5,0);
  
  \draw[dashed] (-1,1) -- (1,1);
  \draw[dashed] (-1,1) -- (-1,0);
  \draw[dashed] (0,0) -- (0,1);
  \draw[dashed] (1,0) -- (1,1);
  \draw[dashed] (1,0) -- (-1,0);
 
  \draw[dashed] (3,1) -- (4,1);
  \draw[dashed] (3,0) -- (4,0);
  \draw[dashed] (3,1) -- (3,-1);
  \draw[dashed] (4,-1) -- (3,-1);
  \draw[dashed] (4,-1) -- (4,1);
  
  \draw[dashed] (6,1) -- (7,1);
  \draw[dashed] (6,0) -- (7,0);
  \draw[dashed] (6,1) -- (6,-1);
  \draw[dashed] (7,-1) -- (6,-1);
  \draw[dashed] (7,-1) -- (7,1);
  
  \filldraw[fill=black!100,draw=black!80] (-4,0.5) circle (2.5pt)    node[anchor=north] {\small };
  
  \filldraw[fill=black!100,draw=black!80] (-1,0.5) circle (2.5pt)    node[anchor=north] {\small };
  \filldraw[fill=black!100,draw=black!80] (1,0.5) circle (2.5pt)    node[anchor=north] {\small };
  
   \filldraw[fill=black!100,draw=black!80] (4,0.5) circle (2.5pt)    node[anchor=north] {\small };
  \filldraw[fill=black!100,draw=black!80] (3,-0.5) circle (2.5pt)    node[anchor=north] {\small };
  \filldraw[fill=black!100,draw=black!80] (6,0.5) circle (2.5pt)    node[anchor=north] {\small };
  \filldraw[fill=black!100,draw=black!80] (7,-0.5) circle (2.5pt)    node[anchor=north] {\small };
  
  \end{tikzpicture}
\end{center} 

So decorated dominoes can be naturally associated to these figures as follows:

\begin{figure}[ht]
\begin{center}
\begin{tikzpicture}[domain=0:4, scale=1]

  \draw[orange] (-5,1) -- (-3,1);
  \draw[orange] (-5,1) -- (-5,0);
  \draw[orange] (-3,0) -- (-3,1);
  \draw[orange] (-3,0) -- (-5,0);
  
  \draw[orange] (-1,1) -- (1,1);
  \draw[orange] (-1,1) -- (-1,0);
  \draw[orange] (1,0) -- (1,1);
  \draw[red, fill, ->, >=latex] (-1,0.5) -- (0.9,0.5);
  \draw[orange] (1,0) -- (-1,0);
 
  \draw[orange] (3,1) -- (4,1);
  \draw[orange] (3,1) -- (3,-1);
  \draw[orange] (4,-1) -- (3,-1);
  \draw[orange] (4,-1) -- (4,1);
  \draw[red, fill, ->, >=latex] (3,-0.5) -- (3.9,0.4);
  
  \draw[orange] (6,1) -- (7,1);
  \draw[orange] (6,1) -- (6,-1);
  \draw[orange] (7,-1) -- (6,-1);
  \draw[orange] (7,-1) -- (7,1);
  \draw[red, fill, ->, >=latex] (6,0.5) -- (6.9,-0.4);

  \filldraw[fill=black!100,draw=black!80] (-4,0.5) circle (2.5pt)    node[anchor=north] {\small } node[above=13pt] {\textbf{Empty domino}};
  
  \filldraw[fill=black!100,draw=black!80] (-1,0.5) circle (2.5pt)    node[anchor=north] {\small } node[above=13pt] {\textbf{~~~~~~~~~~~~~~$(2,0)$ domino}};
  \filldraw[fill=black!100,draw=black!80] (1,0.5) circle (2.5pt)    node[anchor=north] {\small };
  
   \filldraw[fill=black!100,draw=black!80] (4,0.5) circle (2.5pt)    node[anchor=north] {\small }node[above=13pt] {\textbf{$(1,1)$ domino~~~~~~~~~}};
  \filldraw[fill=black!100,draw=black!80] (3,-0.5) circle (2.5pt)    node[anchor=north] {\small };
  
  \filldraw[fill=black!100,draw=black!80] (6,0.5) circle (2.5pt)    node[anchor=north] {\small } node[above=13pt] {\textbf{~~~~~~~$(1,-1)$ domino}};
  \filldraw[fill=black!100,draw=black!80] (7,-0.5) circle (2.5pt)    node[anchor=north] {\small };
  
  \end{tikzpicture}
\end{center} 
\caption{Decorated dominoes.}
\label{fig:decorated_dominoes}
\end{figure}
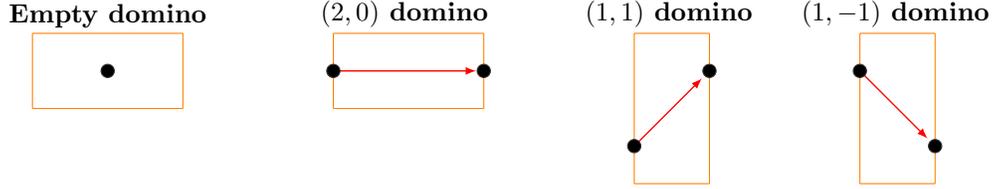

\begin{remark}
While the coordinates of the black points in the grid are not aligned with integers on the $y$-axis, a simple translation by 0.5 units downwards shows that the resulting grid does indeed possess integer coordinates. Based on this observation, we classify each of the paths derived from the decorated snake graph tilings as a lattice path with steps $S=\{(2,0),(1,1),(1,-1)\}$.
\end{remark}

\begin{definition}\label{def:triangular_snake_graph_1}
    The \emph{triangular snake graph} $\mathcal{T}_\mathcal{G}$ associated to $\mathcal{G}$ is a connected acyclic-directed graph derived from the decorated snake graph cover. It is constructed by placing an arrow between two distinct black points if the corresponding unit squares share an edge as shown in Figure~\ref{fig:decorated_dominoes}.
\end{definition}

\begin{remark}\label{re:Basic_properties_T_G}
Notably, the black points that belong to only one square in $T(\mathcal{G})$ are sources or sinks. Considering that the arrows in the triangular snake graph are oriented from left to right, the sources (resp. the sinks) reside on a domino tile's left-hand side (resp. right-hand side). Additionally, note that every tile $G_j$ of $\mathcal{G}$ contributes a triangular tile labeled by the number $j$, depending on the parity of $j$, as follows:
    
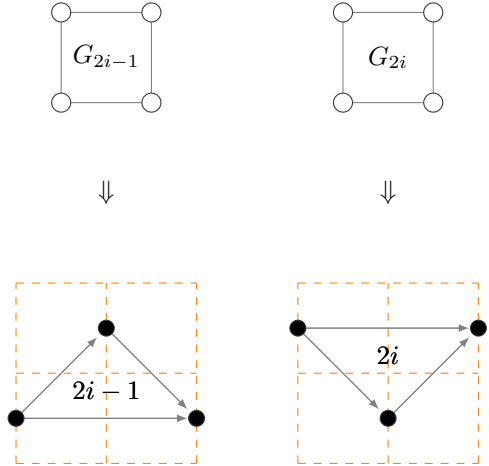
\begin{center}
\begin{tikzpicture}[scale=1.2]

\draw[gray] (0.5,4) -- (0.5,5);
\draw[gray] (1.5,4) -- (1.5,5);
\draw[gray] (1.5,5) -- (0.5,5);
\draw[gray] (0.5,4) -- (1.5,4);

\filldraw[fill=white!40,draw=black!80] (0.5,4) circle (3pt);
\filldraw[fill=white!40,draw=black!80] (0.5,5) circle (3pt);
\filldraw[fill=white!40,draw=black!80] (1.5,4) circle (3pt);
\filldraw[fill=white!40,draw=black!80] (1.5,5) circle (3pt);
\filldraw(1,3)     node[anchor=center] {$\Downarrow$};
\filldraw(1,4.5)     node[anchor=center] {$G_{2i-1}$};

\draw[orange, dashed] (0,0) -- (2,0);
\draw[orange, dashed] (0,1) -- (2,1);
\draw[orange, dashed] (0,2) -- (2,2);

\draw[orange, dashed] (0,0) -- (0,2);
\draw[orange, dashed] (1,0) -- (1,2);
\draw[orange, dashed] (2,0) -- (2,2);

%Arrows

\draw[gray, fill, ->, >=latex] (0,0.5) -- (0.9,1.4);
\draw[gray, fill, ->, >=latex] (1,1.5) -- (1.9,0.6);
\draw[gray, fill, ->, >=latex] (0,0.5) -- (1.9,0.5);

%\black points
\foreach \x in {0, 2} 
\filldraw[fill=black!100,draw=black!80] (\x,0.5) circle (2.5pt)    node[anchor=north] {\small};

\foreach \x in {1}
\filldraw[fill=black!100,draw=black!80] (\x,1.5) circle (2.5pt)    node[anchor=north] {\small };

%\filldraw (1,0.8) node {$\mathbf{G_d}$};
\filldraw (1,0.8) node {$2i-1$};
\filldraw (1,0.8) node {$2i-1$};

\end{tikzpicture} \hspace{1cm}\begin{tikzpicture}[scale=1.2]

\draw[gray] (0.5,4) -- (0.5,5);
\draw[gray] (1.5,4) -- (1.5,5);
\draw[gray] (1.5,5) -- (0.5,5);
\draw[gray] (0.5,4) -- (1.5,4);

\filldraw[fill=white!40,draw=black!80] (0.5,4) circle (3pt);
\filldraw[fill=white!40,draw=black!80] (0.5,5) circle (3pt);
\filldraw[fill=white!40,draw=black!80] (1.5,4) circle (3pt);
\filldraw[fill=white!40,draw=black!80] (1.5,5) circle (3pt);
\filldraw(1,3)     node[anchor=center] {$\Downarrow$};
\filldraw(1,4.5)     node[anchor=center] {$G_{2i}$};

\draw[orange, dashed] (0,0) -- (2,0);
\draw[orange, dashed] (0,1) -- (2,1);
\draw[orange, dashed] (0,2) -- (2,2);

\draw[orange, dashed] (0,0) -- (0,2);
\draw[orange, dashed] (1,0) -- (1,2);
\draw[orange, dashed] (2,0) -- (2,2);

%Arrows

\draw[gray, fill, ->, >=latex] (1,0.5) -- (1.9,1.4);
\draw[gray, fill, ->, >=latex] (0,1.5) -- (0.9,0.6);
\draw[gray, fill, ->, >=latex] (0,1.5) -- (1.9,1.5);

%\black points
\foreach \x in {0, 2} 
\filldraw[fill=black!100,draw=black!80] (\x,1.5) circle (2.5pt)    node[anchor=west] {\small};

\foreach \x in {1}
\filldraw[fill=black!100,draw=black!80] (\x,0.5) circle (2.5pt)    node[anchor=north] {\small };

%\filldraw (1,1.2) node {$\mathbf{G_d}$};

\filldraw (1,1.2) node {$2i$};
\filldraw (1,1.2) node {$2i$};
\end{tikzpicture}
\end{center}

\end{remark}

We will now prove some other not-so-obvious properties about $\mathcal{T}_\mathcal{G}$.

\begin{lemma}\label{le:equality_of_sources_and_sinks}
Let $\mathcal{T}_\mathcal{G}$ be the triangular snake graph associated to $\mathcal{G}$. Then $\mathcal{T}_\mathcal{G}$ possesses an equal number of sources and sinks.
\end{lemma}

\begin{proof}
The proof proceeds by induction on the number of tiles in $\mathcal{G}$. For the base case when $d = 1$, $\mathcal{G}=G_1$, then the triangular snake graph $\mathcal{T}_\mathcal{G}$ associated to $G_1$ is depicted as:

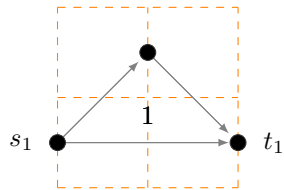
\begin{center}
\begin{tikzpicture}[scale=1.2]

\draw[orange, dashed] (0,0) -- (2,0);
\draw[orange, dashed] (0,1) -- (2,1);
\draw[orange, dashed] (0,2) -- (2,2);

\draw[orange, dashed] (0,0) -- (0,2);
\draw[orange, dashed] (1,0) -- (1,2);
\draw[orange, dashed] (2,0) -- (2,2);

%Arrows

\draw[gray, fill, ->, >=latex] (0,0.5) -- (0.9,1.4);
\draw[gray, fill, ->, >=latex] (1,1.5) -- (1.9,0.6);
\draw[gray, fill, ->, >=latex] (0,0.5) -- (1.9,0.5);

%\black points
\foreach \x in {0, 2} 
\filldraw[fill=black!100,draw=black!80] (\x,0.5) circle (2.5pt)    node[anchor=north] {\small };

\foreach \x in {1}
\filldraw[fill=black!100,draw=black!80] (\x,1.5) circle (2.5pt)    node[anchor=north] {\small };

\filldraw (-0.4,0.5) node {$s_1$};
\filldraw (2.4,0.5) node {$t_1$};

\filldraw (1,0.8) node {$1$};
\filldraw (1,0.8) node {$1$};
\end{tikzpicture}
\end{center}

In this case, there exists one source $s_1$ and one sink $t_1$. Assume the statement holds for any snake graph with $d$ tiles, $d>1$. We will now establish its validity for a snake graph $\mathcal{G}_{d+1}$ containing $d+1$ tiles. 

Consider the snake graph $\mathcal{G}_d$ obtained by removing the last tile $G_{d+1}$ from $\mathcal{G}_{d+1}$. By the induction hypothesis, $\mathcal{T}_{\mathcal{G}_{d}}$ has an equal number of sources and sinks. Now, we examine two possible scenarios for the last part of the triangular snake graph of $\mathcal{G}_d$:

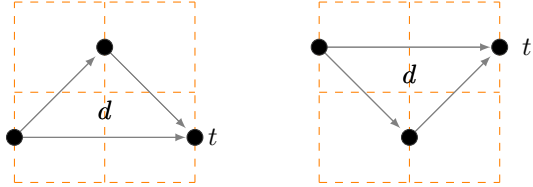
\begin{center}
\begin{tikzpicture}[scale=1.2]

\draw[orange, dashed] (0,0) -- (2,0);
\draw[orange, dashed] (0,1) -- (2,1);
\draw[orange, dashed] (0,2) -- (2,2);

\draw[orange, dashed] (0,0) -- (0,2);
\draw[orange, dashed] (1,0) -- (1,2);
\draw[orange, dashed] (2,0) -- (2,2);

%Arrows

\draw[gray, fill, ->, >=latex] (0,0.5) -- (0.9,1.4);
\draw[gray, fill, ->, >=latex] (1,1.5) -- (1.9,0.6);
\draw[gray, fill, ->, >=latex] (0,0.5) -- (1.9,0.5);

%\black points
\foreach \x in {0, 2} 
\filldraw[fill=black!100,draw=black!80] (\x,0.5) circle (2.5pt)    node[anchor=north] {\small};

\foreach \x in {1}
\filldraw[fill=black!100,draw=black!80] (\x,1.5) circle (2.5pt)    node[anchor=north] {\small };

%\filldraw (1,0.8) node {$\mathbf{G_d}$};
\filldraw (1,0.8) node {$d$};
\filldraw (1,0.8) node {$d$};

\filldraw (2.2,0.5) node {$t$};
\end{tikzpicture} \hspace{1cm}\begin{tikzpicture}[scale=1.2]

\draw[orange, dashed] (0,0) -- (2,0);
\draw[orange, dashed] (0,1) -- (2,1);
\draw[orange, dashed] (0,2) -- (2,2);

\draw[orange, dashed] (0,0) -- (0,2);
\draw[orange, dashed] (1,0) -- (1,2);
\draw[orange, dashed] (2,0) -- (2,2);

%Arrows

\draw[gray, fill, ->, >=latex] (1,0.5) -- (1.9,1.4);
\draw[gray, fill, ->, >=latex] (0,1.5) -- (0.9,0.6);
\draw[gray, fill, ->, >=latex] (0,1.5) -- (1.9,1.5);

%\black points
\foreach \x in {0, 2} 
\filldraw[fill=black!100,draw=black!80] (\x,1.5) circle (2.5pt)    node[anchor=west] {\small};

\foreach \x in {1}
\filldraw[fill=black!100,draw=black!80] (\x,0.5) circle (2.5pt)    node[anchor=north] {\small };

%\filldraw (1,1.2) node {$\mathbf{G_d}$};

\filldraw (1,1.2) node {$d$};
\filldraw (1,1.2) node {$d$};

\filldraw (2.3,1.5) node {$t$};
\end{tikzpicture}
\end{center}

If tiles $G_d$ and $G_{d + 1}$ share the edge $e_d$, where $e_d$ is the north edge of $G_d$ and the south edge of $G_{d + 1}$, we obtain two possible different structures:

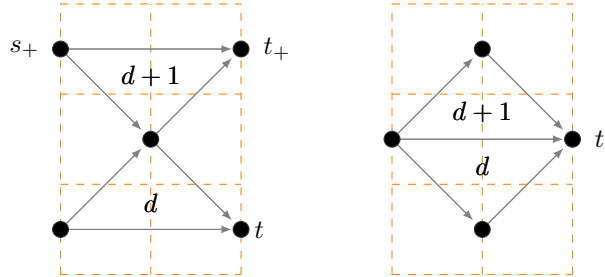
\begin{figure}[ht]
\begin{center}
\begin{tikzpicture}[scale=1.2]

\draw[orange, dashed] (0,0) -- (2,0);
\draw[orange, dashed] (0,1) -- (2,1);
\draw[orange, dashed] (0,2) -- (2,2);
\draw[orange, dashed] (0,3) -- (2,3);

\draw[orange, dashed] (0,0) -- (0,3);
\draw[orange, dashed] (1,0) -- (1,3);
\draw[orange, dashed] (2,0) -- (2,3);

%Arrows

\draw[gray, fill, ->, >=latex] (0,0.5) -- (0.9,1.4);
\draw[gray, fill, ->, >=latex] (1,1.5) -- (1.9,0.6);
\draw[gray, fill, ->, >=latex] (0,0.5) -- (1.9,0.5);

\draw[gray, fill, ->, >=latex] (1,1.5) -- (1.9,2.4);
\draw[gray, fill, ->, >=latex] (0,2.5) -- (0.9,1.6);
\draw[gray, fill, ->, >=latex] (0,2.5) -- (1.9,2.5);

%\black points
\foreach \x in {0, 2} 
\filldraw[fill=black!100,draw=black!80] (\x,0.5) circle (2.5pt)    node[anchor=north] {\small};

\foreach \x in {1}
\filldraw[fill=black!100,draw=black!80] (\x,1.5) circle (2.5pt)    node[anchor=north] {\small };

\foreach \x in {0, 2} 
\filldraw[fill=black!100,draw=black!80] (\x,2.5) circle (2.5pt)    node[anchor=north] {\small};

%\filldraw (1,0.8) node {$\mathbf{G_d}$};
\filldraw (1,0.8) node {$d$};
\filldraw (1,0.8) node {$d$};

\filldraw (1,2.2) node {${d+1}$};
\filldraw (1,2.2) node {${d+1}$};

\filldraw (2.2,0.5) node {$t$};
\filldraw (-0.4,2.5) node {$s_+$};
\filldraw (2.4,2.5) node {$t_+$};

\end{tikzpicture} \hspace{1cm}\begin{tikzpicture}[scale=1.2]

\draw[orange, dashed] (0,0) -- (2,0);
\draw[orange, dashed] (0,1) -- (2,1);
\draw[orange, dashed] (0,2) -- (2,2);
\draw[orange, dashed] (0,3) -- (2,3);

\draw[orange, dashed] (0,0) -- (0,3);
\draw[orange, dashed] (1,0) -- (1,3);
\draw[orange, dashed] (2,0) -- (2,3);

%Arrows

\draw[gray, fill, ->, >=latex] (0,1.5) -- (0.9,2.4);
\draw[gray, fill, ->, >=latex] (1,2.5) -- (1.9,1.6);
\draw[gray, fill, ->, >=latex] (1,0.5) -- (1.9,1.4);
\draw[gray, fill, ->, >=latex] (0,1.5) -- (0.9,0.6);
\draw[gray, fill, ->, >=latex] (0,1.5) -- (1.9,1.5);

%\black points

\foreach \x in {1}
\filldraw[fill=black!100,draw=black!80] (\x,2.5) circle (2.5pt)    node[anchor=north] {\small };

\foreach \x in {0, 2} 
\filldraw[fill=black!100,draw=black!80] (\x,1.5) circle (2.5pt)    node[anchor=west] {\small};

\foreach \x in {1}
\filldraw[fill=black!100,draw=black!80] (\x,0.5) circle (2.5pt)    node[anchor=north] {\small };

%\filldraw (1,1.2) node {$\mathbf{G_d}$};

\filldraw (1,1.2) node {$d$};
\filldraw (1,1.2) node {$d$};
\filldraw (1,1.8) node {${d+1}$};
\filldraw (1,1.8) node {${d+1}$};

\filldraw (2.3,1.5) node {$t$};
\end{tikzpicture}
\end{center}
\caption{Triangular tiles associated to the tiles $G_d$ and $G_{d+1}$ that share the north and south edge.}
    \label{fig:North-South}
\end{figure}

In the first case, the resulting triangular snake graph of $\mathcal{G}_{d+1}$ encompasses all the sources and sinks of $\mathcal{G}_{d}$. Additionally, two new elements are introduced: a source $s_{+}$ and a sink $t_{+}$ as depicted in the left of Figure~\ref{fig:North-South}. Conversely, in the second case, the sources and sinks in both $\mathcal{G}_{d}$ and $\mathcal{G}_{d+1}$ are equal. This distinction arises due to the structural differences between the configurations of adjoining tiles. 

On the other hand, if the tiles $G_d$ and $G_{d + 1}$ share the edge $e_d$, where $e_d$ is the east edge of $G_d$ and the west edge of $G_{d + 1}$, we obtain one of the following scenarios

\begin{figure}[ht]
\begin{center}
\begin{tikzpicture}[scale=1.2]

\draw[orange, dashed] (0,0) -- (3,0);
\draw[orange, dashed] (0,1) -- (3,1);
\draw[orange, dashed] (0,2) -- (3,2);

\draw[orange, dashed] (0,0) -- (0,2);
\draw[orange, dashed] (1,0) -- (1,2);
\draw[orange, dashed] (2,0) -- (2,2);
\draw[orange, dashed] (3,0) -- (3,2);

%Arrows

\draw[gray, fill, ->, >=latex] (0,0.5) -- (0.9,1.4);
\draw[gray, fill, ->, >=latex] (1,1.5) -- (1.9,0.6);
\draw[gray, fill, ->, >=latex] (0,0.5) -- (1.9,0.5);
\draw[gray, fill, ->, >=latex] (2,0.5) -- (2.9,1.4);
\draw[gray, fill, ->, >=latex] (1,1.5) -- (2.9,1.5);

%\black points
\foreach \x in {0, 2} 
\filldraw[fill=black!100,draw=black!80] (\x,0.5) circle (2.5pt)    node[anchor=north] {\small};

\foreach \x in {1,3}
\filldraw[fill=black!100,draw=black!80] (\x,1.5) circle (2.5pt)    node[anchor=north] {\small };

%\filldraw (1,0.8) node {$\mathbf{G_d}$};
\filldraw (1,0.8) node {$d$};
\filldraw (1,0.8) node {$d$};
\filldraw (2,1.2) node {${d+1}$};
\filldraw (2,1.2) node {${d+1}$};

\filldraw[fill=white!100,draw=red!80] (2.3,0.5) circle (4pt)    node[anchor=north] {\small };
\filldraw (2.3,0.5) node {\textcolor{red}{$t$}};
\filldraw (3.3,1.5) node {$t$};
\end{tikzpicture} \hspace{1cm}\begin{tikzpicture}[scale=1.2]

\draw[orange, dashed] (0,0) -- (3,0);
\draw[orange, dashed] (0,1) -- (3,1);
\draw[orange, dashed] (0,2) -- (3,2);

\draw[orange, dashed] (0,0) -- (0,2);
\draw[orange, dashed] (1,0) -- (1,2);
\draw[orange, dashed] (2,0) -- (2,2);
\draw[orange, dashed] (3,0) -- (3,2);

%Arrows

\draw[gray, fill, ->, >=latex] (1,0.5) -- (1.9,1.4);
\draw[gray, fill, ->, >=latex] (0,1.5) -- (0.9,0.6);
\draw[gray, fill, ->, >=latex] (0,1.5) -- (1.9,1.5);
\draw[gray, fill, ->, >=latex] (1,0.5) -- (2.9,0.5);
\draw[gray, fill, ->, >=latex] (2,1.5) -- (2.9,0.6);
%\black points
\foreach \x in {0, 2} 
\filldraw[fill=black!100,draw=black!80] (\x,1.5) circle (2.5pt)    node[anchor=west] {\small};

\foreach \x in {1,3}
\filldraw[fill=black!100,draw=black!80] (\x,0.5) circle (2.5pt)    node[anchor=north] {\small };

%\filldraw (1,1.2) node {$\mathbf{G_d}$};

\filldraw (1,1.2) node {$d$};
\filldraw (1,1.2) node {$d$};
\filldraw (2,0.8) node {${d+1}$};
\filldraw (2,0.8) node {${d+1}$};

\filldraw[fill=white!100,draw=red!80] (2.3,1.5) circle (4pt)    node[anchor=north] {\small };
\filldraw (2.3,1.5) node {\textcolor{red}{$t$}};
\filldraw (3.3,0.5) node {$t$};
\end{tikzpicture}
\end{center}
\caption{Triangular tiles associated to the tiles $G_d$ and $G_{d+1}$ that share the east and west edge.}
    \label{fig:east-west}
\end{figure}
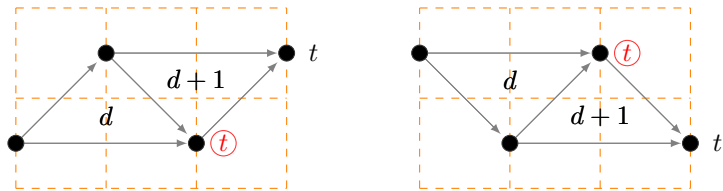

Upon close examination, it becomes evident that the element identified as the sink $t$ in $\mathcal{G}_{d}$ is no longer a sink of $\mathcal{T}_{\mathcal{G}_{d+1}}$. Instead, the southeast black point within the new triangular tile takes over the role previously attributed to the mentioned sink. 

In either case, the count of sources and sinks remains equal, thereby completing the proof.
\end{proof}

The proof of Lemma~\ref{le:equality_of_sources_and_sinks} yields an additional insight: the number of sources and sinks increases with the appearance of a specific subgraph within the triangular snake graph. This subgraph is formally defined in the following definition.

\begin{definition}\label{def:hourglass graph}
An \emph{hourglass graph} is a graph with the following structure:
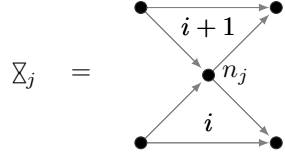
\begin{center}
\begin{tikzpicture}[scale=1]
%Arrows
\draw[gray, fill, ->, >=latex] (0,0.5) -- (0.9,1.4);
\draw[gray, fill, ->, >=latex] (1,1.5) -- (1.9,0.6);
\draw[gray, fill, ->, >=latex] (0,0.5) -- (1.9,0.5);

\draw[gray, fill, ->, >=latex] (1,1.5) -- (1.9,2.4);
\draw[gray, fill, ->, >=latex] (0,2.5) -- (0.9,1.6);
\draw[gray, fill, ->, >=latex] (0,2.5) -- (1.9,2.5);

%\black points
\foreach \x in {0, 2} 
\filldraw[fill=black!100,draw=black!80] (\x,0.5) circle (2.5pt)    node[anchor=north] {\small};

\foreach \x in {1}
\filldraw[fill=black!100,draw=black!80] (\x,1.5) circle (2.5pt)    node[anchor=north] {\small };

\foreach \x in {0, 2} 
\filldraw[fill=black!100,draw=black!80] (\x,2.5) circle (2.5pt)    node[anchor=north] {\small};

%Names
\filldraw (-1.3,1.5) node {$\hourglass_j \quad =$};
\filldraw (1.4,1.5) node {$n_j$};
\filldraw (1,0.8) node {$i$};
\filldraw (1,0.8) node {$i$};

\filldraw (1,2.2) node {${i+1}$};
\filldraw (1,2.2) node {${i+1}$};

\end{tikzpicture}
\end{center}

The vertex $n_j$ will be called the \emph{neck} of the hourglass graph $\hourglass_j$. The triangular subgraphs labeled by $i+1$ and $i$ are the \emph{head} and \emph{body} of $\hourglass_j$, respectively.
\end{definition}
As a direct consequence, we establish the following corollary.

\begin{corollary}\label{corollary:k-1 hourglass graphs}
Let $\mathcal{T}_\mathcal{G}$ be the triangular snake graph associated to $\mathcal{G}$. $\mathcal{T}_\mathcal{G}$ contains $k-1$ hourglass graphs as subgraphs if and only if the number of sources (and sinks) in $\mathcal{T}_\mathcal{G}$ is $k$, where $k\geq 1$.
\end{corollary}

\begin{proof}
The proof follows a similar structure as the proof of Lemma~\ref{le:equality_of_sources_and_sinks} and will only be indicated briefly. We proceed by induction based on the number of sources in $\mathcal{T}_\mathcal{G}$. For the base case, when $\mathcal{T}_\mathcal{G}$ has no hourglass graphs, the number of sources (and sinks) is $1$. When $\mathcal{T}_\mathcal{G}$ has only one hourglass graph, the number of sources (and sinks) is $2$. Assuming the validity of the statement for any triangular snake graph with $k-1$ hourglasses, we aim to establish its validity for a triangular snake graph $\mathcal{T}_\mathcal{G}$ containing $k$ hourglasses. 

Let us consider the triangular snake graph $\mathcal{T}_\mathcal{G}^-$ obtained by removing the triangular tiles labeled by $i$ from $\mathcal{T}_\mathcal{G}$, for $i > j$, where $j$ is the label of the body of the last hourglass $\hourglass_k$. Then, $\mathcal{T}_\mathcal{G}^-$ contains $k-1$ hourglass graphs. According to the induction hypothesis, $\mathcal{T}_\mathcal{G}^-$ possesses $k$ sources (and sinks). A similar analysis to that in the proof of  Lemma~\ref{le:equality_of_sources_and_sinks}, shows that $\mathcal{T}_\mathcal{G}$ has $k+1$ sources (and sinks). 
To prove the reverse implication, assume 
$\mathcal{T}_\mathcal{G}$ contains 
$k$ sources (and sinks). By induction on the number of sources, the same reasoning can be applied to show that 
$\mathcal{T}_\mathcal{G}$ contains 
$k-1$ hourglass graphs as subgraphs.
\end{proof}

Unless stated otherwise, we assume that $\mathcal{T}_\mathcal{G}$ possesses $k$ sources, denoted as $s_1, \dots, s_k$, and $k$ sinks, denoted as $t_1, \dots, t_k$. To describe the $k$-routes of $\mathcal{T}_\mathcal{G}$, these $k$-vertices of $\mathcal{T}_\mathcal{G}$ are defined as follows.

\begin{definition} \label{def:k-vertices}
The \emph{$k$-vertices} $s = (s_1,\dots, s_k)$ and $t = (t_1,\dots, t_k)$ of $\mathcal{T}_\mathcal{G}$ are defined by imposing the following conditions:
\begin{enumerate}
\item $\mathcal{T}_\mathcal{G}$ contains $k-1$ hourglass graphs as subgraphs.
\item $s = (s_1,\dots, s_k)$ and $t = (t_1,\dots, t_k)$ are ordered vectors formed by the sources and sinks in $\mathcal{T}_\mathcal{G}$, respectively.
\item If $s_i$ and $s_j$ (respectively $t_i$ and $t_j$) are vertices such that $i<j$, then the y-coordinate of $s_i$ (respectively $t_i$) is less than the y-coordinate of $s_j$ (respectively $t_j$).
\end{enumerate}
\end{definition}

\begin{example}\label{ex:SnakeGraphTilings2}
    Let $\mathcal{G}$ the snake graph in the Example~\ref{ex:SnakeGraphTilings}. The snake graph cover on the grid is depicted as follows:
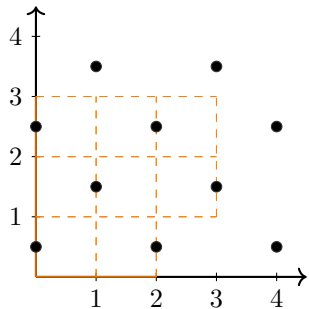
\begin{center}
\begin{tikzpicture}[scale=0.8]
\draw[black,thick,->] (0,0)--(4.5,0) node[right] {}; % Eje x
% Enumeración del eje x
\foreach \x/\xtext in {1/1, 2/2, 3/3, 4/4} 
\draw[shift={(\x,0)},black] (0pt,2pt)--(0pt,-2pt) node[below] {$\xtext$};
% Enumeración del eje y
\foreach \y/\ytext in {1/1, 2/2, 3/3, 4/4} 
\draw[shift={(0,\y)},black] (2pt,0pt)--(-2pt,0pt) node[left] {$\ytext$};
\draw[black,thick,->] (0,0)--(0,4.5) node[left,above] {}; % Eje y
%gray grid
%\foreach \x in {1, 2, 3, 4} 
%\draw[gray, dashed] (\x,0)--(\x,4);
%\foreach \x in {1, 2, 3, 4} 
%\draw[gray, dashed] (0,\x)--(4,\x);
%Snake graph tiling

\draw[orange] (0,0) -- (2,0);
\draw[orange, dashed] (0,1) -- (3,1);
\draw[orange, dashed] (0,2) -- (3,2);
\draw[orange, dashed] (0,3) -- (3,3);

\draw[orange] (0,0) -- (0,3);
\draw[orange, dashed] (1,0) -- (1,3);
\draw[orange, dashed] (2,0) -- (2,3);
\draw[orange, dashed] (3,1) -- (3,3);

\draw[orange] (0,0) -- (2,0);
\draw[orange, dashed] (0,1) -- (3,1);
\draw[orange, dashed] (0,2) -- (3,2);
\draw[orange, dashed] (0,3) -- (3,3);

\draw[orange] (0,0) -- (0,3);
\draw[orange, dashed] (1,0) -- (1,3);
\draw[orange, dashed] (2,0) -- (2,3);
\draw[orange, dashed] (3,1) -- (3,3);

%\black points
\foreach \x in {0, 2, 4} 
\filldraw[fill=black!100,draw=black!80] (\x,0.5) circle (2.5pt)    node[anchor=north] {\small };
\foreach \x in {0, 2, 4}
\filldraw[fill=black!100,draw=black!80] (\x,2.5) circle (2.5pt)    node[anchor=north] {\small };

\foreach \x in {1, 3}
\filldraw[fill=black!100,draw=black!80] (\x,1.5) circle (2.5pt)    node[anchor=north] {\small };
\foreach \x in {1, 3}
\filldraw[fill=black!100,draw=black!80] (\x,3.5) circle (2.5pt)    node[anchor=north] {\small };
\end{tikzpicture} 
\end{center}

The snake graph tilings are

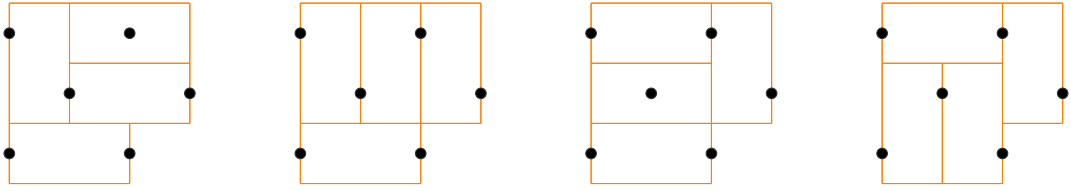
\begin{center}
\begin{tikzpicture}[scale=0.8]
\draw[orange] (0,0) -- (2,0);
\draw[orange] (0,1) -- (3,1);
\draw[orange] (1,2) -- (3,2);
\draw[orange] (0,3) -- (3,3);

\draw[orange] (0,0) -- (0,3);
\draw[orange] (1,1) -- (1,3);
\draw[orange] (2,0) -- (2,1);
\draw[orange] (3,1) -- (3,3);

\draw[orange] (0,0) -- (2,0);
\draw[orange] (0,1) -- (3,1);
\draw[orange] (1,2) -- (3,2);
\draw[orange] (0,3) -- (3,3);

\draw[orange] (0,0) -- (0,3);
\draw[orange] (1,1) -- (1,3);
\draw[orange] (2,0) -- (2,1);
\draw[orange] (3,1) -- (3,3);

%\black points
\foreach \x in {0, 2} 
\filldraw[fill=black!100,draw=black!80] (\x,0.5) circle (2.5pt)    node[anchor=north] {\small };
\foreach \x in {0, 2}
\filldraw[fill=black!100,draw=black!80] (\x,2.5) circle (2.5pt)    node[anchor=north] {\small };

\foreach \x in {1, 3}
\filldraw[fill=black!100,draw=black!80] (\x,1.5) circle (2.5pt)    node[anchor=north] {\small };

\end{tikzpicture} \hspace{1cm} \begin{tikzpicture}[scale=0.8]

\draw[orange] (0,0) -- (2,0);
\draw[orange] (0,1) -- (3,1);
%\draw[orange] (0,2) -- (3,2);
\draw[orange] (0,3) -- (3,3);

\draw[orange] (0,0) -- (0,3);
\draw[orange] (1,1) -- (1,3);
\draw[orange] (2,0) -- (2,3);
\draw[orange] (3,1) -- (3,3);

\draw[orange] (0,0) -- (2,0);
\draw[orange] (0,1) -- (3,1);
%\draw[orange] (0,2) -- (3,2);
\draw[orange] (0,3) -- (3,3);

\draw[orange] (0,0) -- (0,3);
\draw[orange] (1,1) -- (1,3);
\draw[orange] (2,0) -- (2,3);
\draw[orange] (3,1) -- (3,3);

%\black points
\foreach \x in {0, 2} 
\filldraw[fill=black!100,draw=black!80] (\x,0.5) circle (2.5pt)    node[anchor=north] {\small };
\foreach \x in {0, 2}
\filldraw[fill=black!100,draw=black!80] (\x,2.5) circle (2.5pt)    node[anchor=north] {\small };

\foreach \x in {1, 3}
\filldraw[fill=black!100,draw=black!80] (\x,1.5) circle (2.5pt)    node[anchor=north] {\small };
\end{tikzpicture}  \hspace{1cm} \begin{tikzpicture}[scale=0.8]

\draw[orange] (0,0) -- (2,0);
\draw[orange] (0,1) -- (3,1);
\draw[orange] (0,2) -- (2,2);
\draw[orange] (0,3) -- (3,3);

\draw[orange] (0,0) -- (0,3);
%\draw[orange] (1,1) -- (1,3);
\draw[orange] (2,0) -- (2,3);
\draw[orange] (3,1) -- (3,3);

\draw[orange] (0,0) -- (2,0);
\draw[orange] (0,1) -- (3,1);
\draw[orange] (0,2) -- (2,2);
\draw[orange] (0,3) -- (3,3);

\draw[orange] (0,0) -- (0,3);
%\draw[orange] (1,1) -- (1,3);
\draw[orange] (2,0) -- (2,3);
\draw[orange] (3,1) -- (3,3);

%\black points
\foreach \x in {0, 2} 
\filldraw[fill=black!100,draw=black!80] (\x,0.5) circle (2.5pt)    node[anchor=north] {\small };
\foreach \x in {0, 2}
\filldraw[fill=black!100,draw=black!80] (\x,2.5) circle (2.5pt)    node[anchor=north] {\small };

\foreach \x in {1, 3}
\filldraw[fill=black!100,draw=black!80] (\x,1.5) circle (2.5pt)    node[anchor=north] {\small };
\end{tikzpicture} \hspace{1cm} \begin{tikzpicture}[scale=0.8]

\draw[orange] (0,0) -- (2,0);
\draw[orange] (2,1) -- (3,1);
\draw[orange] (0,2) -- (2,2);
\draw[orange] (0,3) -- (3,3);

\draw[orange] (0,0) -- (0,3);
\draw[orange] (1,0) -- (1,2);
\draw[orange] (2,0) -- (2,3);
\draw[orange] (3,1) -- (3,3);

\draw[orange] (0,0) -- (2,0);
\draw[orange] (2,1) -- (3,1);
\draw[orange] (0,2) -- (2,2);
\draw[orange] (0,3) -- (3,3);

\draw[orange] (0,0) -- (0,3);
\draw[orange] (1,0) -- (1,2);
\draw[orange] (2,0) -- (2,3);
\draw[orange] (3,1) -- (3,3);

%\black points
\foreach \x in {0, 2} 
\filldraw[fill=black!100,draw=black!80] (\x,0.5) circle (2.5pt)    node[anchor=north] {\small };
\foreach \x in {0, 2}
\filldraw[fill=black!100,draw=black!80] (\x,2.5) circle (2.5pt)    node[anchor=north] {\small };

\foreach \x in {1, 3}
\filldraw[fill=black!100,draw=black!80] (\x,1.5) circle (2.5pt)    node[anchor=north] {\small };
\end{tikzpicture}
\end{center}

Then, we obtain the following decorated snake graph tilings associated to perfect matchings depicted in Example~\ref{ex:snake_graph_tiling_matching}

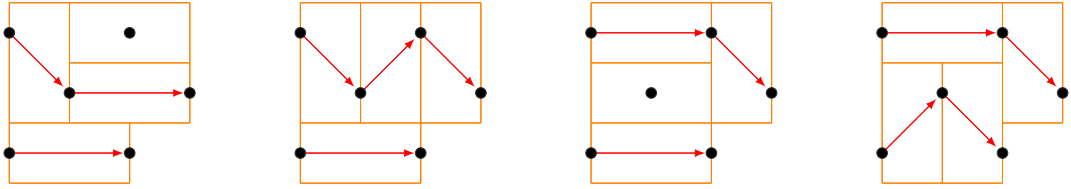
\begin{center}
\begin{tikzpicture}[scale=0.8]
\draw[orange] (0,0) -- (2,0);
\draw[orange] (0,1) -- (3,1);
\draw[orange] (1,2) -- (3,2);
\draw[orange] (0,3) -- (3,3);

\draw[orange] (0,0) -- (0,3);
\draw[orange] (1,1) -- (1,3);
\draw[orange] (2,0) -- (2,1);
\draw[orange] (3,1) -- (3,3);

\draw[orange] (0,0) -- (2,0);
\draw[orange] (0,1) -- (3,1);
\draw[orange] (1,2) -- (3,2);
\draw[orange] (0,3) -- (3,3);

\draw[orange] (0,0) -- (0,3);
\draw[orange] (1,1) -- (1,3);
\draw[orange] (2,0) -- (2,1);
\draw[orange] (3,1) -- (3,3);

%Arrows
\draw[red, fill, ->, >=latex] (0,0.5) -- (1.9,0.5);
\draw[red, fill, ->, >=latex] (1,1.5) -- (2.9,1.5);
\draw[red, fill, ->, >=latex] (0,2.5) -- (0.9,1.6);

%\black points
\foreach \x in {0, 2} 
\filldraw[fill=black!100,draw=black!80] (\x,0.5) circle (2.5pt)    node[anchor=north] {\small };
\foreach \x in {0, 2}
\filldraw[fill=black!100,draw=black!80] (\x,2.5) circle (2.5pt)    node[anchor=north] {\small };

\foreach \x in {1, 3}
\filldraw[fill=black!100,draw=black!80] (\x,1.5) circle (2.5pt)    node[anchor=north] {\small };

\end{tikzpicture} \hspace{1cm} \begin{tikzpicture}[scale=0.8]
\draw[orange] (0,0) -- (2,0);
\draw[orange] (0,1) -- (3,1);
%\draw[orange] (0,2) -- (3,2);
\draw[orange] (0,3) -- (3,3);

\draw[orange] (0,0) -- (0,3);
\draw[orange] (1,1) -- (1,3);
\draw[orange] (2,0) -- (2,3);
\draw[orange] (3,1) -- (3,3);

\draw[orange] (0,0) -- (2,0);
\draw[orange] (0,1) -- (3,1);
%\draw[orange] (0,2) -- (3,2);
\draw[orange] (0,3) -- (3,3);

\draw[orange] (0,0) -- (0,3);
\draw[orange] (1,1) -- (1,3);
\draw[orange] (2,0) -- (2,3);
\draw[orange] (3,1) -- (3,3);

%Arrows
\draw[red, fill, ->, >=latex] (0,0.5) -- (1.9,0.5);
\draw[red, fill, ->, >=latex] (0,2.5) -- (0.9,1.6);
\draw[red, fill, ->, >=latex] (2,2.5) -- (2.9,1.6);
\draw[red, fill, ->, >=latex] (1,1.5) -- (1.9,2.4);

%\black points
\foreach \x in {0, 2} 
\filldraw[fill=black!100,draw=black!80] (\x,0.5) circle (2.5pt)    node[anchor=north] {\small };
\foreach \x in {0, 2}
\filldraw[fill=black!100,draw=black!80] (\x,2.5) circle (2.5pt)    node[anchor=north] {\small };

\foreach \x in {1, 3}
\filldraw[fill=black!100,draw=black!80] (\x,1.5) circle (2.5pt)    node[anchor=north] {\small };

\end{tikzpicture} \hspace{1cm} \begin{tikzpicture}[scale=0.8]

\draw[orange] (0,0) -- (2,0);
\draw[orange] (0,1) -- (3,1);
\draw[orange] (0,2) -- (2,2);
\draw[orange] (0,3) -- (3,3);

\draw[orange] (0,0) -- (0,3);
%\draw[orange] (1,1) -- (1,3);
\draw[orange] (2,0) -- (2,3);
\draw[orange] (3,1) -- (3,3);

\draw[orange] (0,0) -- (2,0);
\draw[orange] (0,1) -- (3,1);
\draw[orange] (0,2) -- (2,2);
\draw[orange] (0,3) -- (3,3);

\draw[orange] (0,0) -- (0,3);
%\draw[orange] (1,1) -- (1,3);
\draw[orange] (2,0) -- (2,3);
\draw[orange] (3,1) -- (3,3);

%Arrows
\draw[red, fill, ->, >=latex] (0,0.5) -- (1.9,0.5);
\draw[red, fill, ->, >=latex] (0,2.5) -- (1.9,2.5);
\draw[red, fill, ->, >=latex] (2,2.5) -- (2.9,1.6);

%\black points
\foreach \x in {0, 2} 
\filldraw[fill=black!100,draw=black!80] (\x,0.5) circle (2.5pt)    node[anchor=north] {\small };
\foreach \x in {0, 2}
\filldraw[fill=black!100,draw=black!80] (\x,2.5) circle (2.5pt)    node[anchor=north] {\small };

\foreach \x in {1, 3}
\filldraw[fill=black!100,draw=black!80] (\x,1.5) circle (2.5pt)    node[anchor=north] {\small };
\end{tikzpicture} \hspace{1cm} \begin{tikzpicture}[scale=0.8]

\draw[orange] (0,0) -- (2,0);
\draw[orange] (2,1) -- (3,1);
\draw[orange] (0,2) -- (2,2);
\draw[orange] (0,3) -- (3,3);

\draw[orange] (0,0) -- (0,3);
\draw[orange] (1,0) -- (1,2);
\draw[orange] (2,0) -- (2,3);
\draw[orange] (3,1) -- (3,3);

\draw[orange] (0,0) -- (2,0);
\draw[orange] (2,1) -- (3,1);
\draw[orange] (0,2) -- (2,2);
\draw[orange] (0,3) -- (3,3);

\draw[orange] (0,0) -- (0,3);
\draw[orange] (1,0) -- (1,2);
\draw[orange] (2,0) -- (2,3);
\draw[orange] (3,1) -- (3,3);

%Arrows
\draw[red, fill, ->, >=latex] (0,2.5) -- (1.9,2.5);
\draw[red, fill, ->, >=latex] (2,2.5) -- (2.9,1.6);
\draw[red, fill, ->, >=latex] (0,0.5) -- (0.9,1.4);
\draw[red, fill, ->, >=latex] (1,1.5) -- (1.9,0.6);

%\black points
\foreach \x in {0, 2} 
\filldraw[fill=black!100,draw=black!80] (\x,0.5) circle (2.5pt)    node[anchor=north] {\small };
\foreach \x in {0, 2}
\filldraw[fill=black!100,draw=black!80] (\x,2.5) circle (2.5pt)    node[anchor=north] {\small };

\foreach \x in {1, 3}
\filldraw[fill=black!100,draw=black!80] (\x,1.5) circle (2.5pt)    node[anchor=north] {\small };

\end{tikzpicture}
\end{center}
and the triangular snake graph $\mathcal{T}_\mathcal{G}$ with its $2$-vertices, $s=(s_1,s_2)$ and $t=(t_1,t_2)$, is

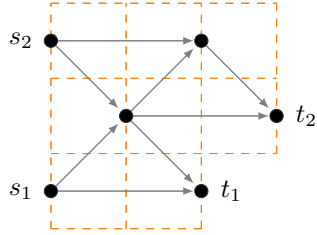
\begin{center}
\begin{tikzpicture}[scale=1]
\draw[orange, dashed] (0,0) -- (2,0);
\draw[orange, dashed] (0,1) -- (3,1);
\draw[orange, dashed] (0,2) -- (3,2);
\draw[orange, dashed] (0,3) -- (3,3);

\draw[orange, dashed] (0,0) -- (0,3);
\draw[orange, dashed] (1,0) -- (1,3);
\draw[orange, dashed] (2,0) -- (2,3);
\draw[orange, dashed] (3,1) -- (3,3);

\draw[orange, dashed] (0,0) -- (2,0);
\draw[orange, dashed] (0,1) -- (3,1);
\draw[orange, dashed] (0,2) -- (3,2);
\draw[orange, dashed] (0,3) -- (3,3);

\draw[orange, dashed] (0,0) -- (0,3);
\draw[orange, dashed] (1,0) -- (1,3);
\draw[orange, dashed] (2,0) -- (2,3);
\draw[orange, dashed] (3,1) -- (3,3);

%Arrows
\draw[gray, fill, ->, >=latex] (0,2.5) -- (1.9,2.5);
\draw[gray, fill, ->, >=latex] (2,2.5) -- (2.9,1.6);
\draw[gray, fill, ->, >=latex] (0,0.5) -- (0.9,1.4);
\draw[gray, fill, ->, >=latex] (1,1.5) -- (1.9,0.6);
\draw[gray, fill, ->, >=latex] (1,1.5) -- (2.9,1.5);
\draw[gray, fill, ->, >=latex] (0,0.5) -- (1.9,0.5);
\draw[gray, fill, ->, >=latex] (0,2.5) -- (0.9,1.6);
\draw[gray, fill, ->, >=latex] (1,1.5) -- (1.9,2.4);

%\black points
\foreach \x in {0, 2} 
\filldraw[fill=black!100,draw=black!80] (\x,0.5) circle (2.5pt)    node[anchor=north] {\small };
\foreach \x in {0, 2}
\filldraw[fill=black!100,draw=black!80] (\x,2.5) circle (2.5pt)    node[anchor=north] {\small };

\foreach \x in {1,3}
\filldraw[fill=black!100,draw=black!80] (\x,1.5) circle (2.5pt)    node[anchor=north] {\small };

\filldraw (-0.4,0.5) node {$s_1$};
\filldraw (-0.4,2.5) node {$s_2$};
\filldraw (2.4,0.5) node {$t_1$};
\filldraw (3.4,1.5) node {$t_2$};
\end{tikzpicture} 
\end{center}

We can observe that each decorated snake graph tiling contains two paths that do not intersect. Additionally, the paths start at vertex $s_i$ and end at vertex $t_i$, for some $i\in\{1,2\}$. 
\end{example}

\begin{lemma}\label{lemma:Perfect_matchings_to_routes}
Let $P\in \text{Match}(\mathcal{G})$. Then the decorated snake graph cover associated to $P$ represents a $k$-route from $s$ to $t$ in $\mathcal{T}_\mathcal{G}$.
\end{lemma}
\begin{proof}
Let $p=(v_{l+1}|\alpha_l,\alpha_{l-1},\dots,\alpha_1|v_1)$ be a maximal path in the decorated snake graph tiling associated to $P\in \text{Match}(\mathcal{G})$ such that $p$ starts at vertex $v_1$ and ends at vertex $v_{l+1}$. There exist $l$ decorated domino tiles with steps $(2,0)$, $(1,1)$ or $(1,-1)$ in such a way that every arrow $\alpha_i$ of $p$ corresponds to one of these decorated domino tiles. We claim that $v_1$ is a source and $v_{l+1}$ is a sink in $\mathcal{T}_\mathcal{G}$. Otherwise, if $v_1$ is not a source, then it would be on the right-hand side of a domino tile $D_{v_1}$, leading to two possible cases:

\begin{enumerate}
    \item $D_{v_1}$ is an empty domino. This case contradicts the requirement that consecutive horizontal black points should maintain a distance of $2$ units, not one.
    \item $D_{v_1}$ is a $(2,0)$, $(1,1)$, or $(1,-1)$ domino. This case breaches the requirement that $p$ is a maximal path starting at $v_1$ in the decorated snake graph tiling associated to $P$.
\end{enumerate}

Similarly, if $v_{l+1}$ is not a sink, it would be on the left-hand side of a domino tile $D_{v_{l+1}}$, and the same two cases apply as for $v_1$. Consequently, there must exist $i,j\in [k]$ such that $v_1=s_i$ and $v_{l+1}=t_j$. 
%(similarly, $v_{l+1}$ on the left-hand side of a domino tile $D_{v_{l+1}}$) (if $D_{v_1}$ or $D_{v_{l+1}}$ are empty dominoes), or breaching the requirement that $p$ is a path in the decorated snake graph because $D_{v_1}$ and $D_{v_{l+1}}$ cannot be $(2,0)$, $(1,1)$, or $(1,-1)$ dominoes or $p$ does not start or end at $v_1$ or $v_{l+1}$.

It is now necessary to prove that indeed $i=j$. Suppose $i\neq j$. Then, $p$ traverses an hourglass graph in one of the three ways shown on the left of Figure~\ref{fig:Non-perfect_matchings}.

\begin{figure}[ht]
\begin{center}
\begin{tikzpicture}[scale=0.8]
%Arrows

\draw[red, thick, fill, ->, >=latex] (0,0.5) -- (0.9,1.4);
\draw[gray, fill, ->, >=latex] (1,1.5) -- (1.9,0.6);
\draw[gray, fill, ->, >=latex] (0,0.5) -- (1.9,0.5);

\draw[red, thick, fill, ->, >=latex] (1,1.5) -- (1.9,2.4);
\draw[gray, fill, ->, >=latex] (0,2.5) -- (0.9,1.6);
\draw[gray, fill, ->, >=latex] (0,2.5) -- (1.9,2.5);

%:) :) :)

%\draw[dashed, gray] (-1.5,0) -- (0.9,1.4);
\draw[dashed, gray] (-1.5,0) -- (0,0.5);
%\draw[dashed, gray] (-1.5,0) -- (2,0.5);

%\draw[dashed, gray] (3.5,3) -- (0.9,1.4);
%\draw[dashed, gray] (3.5,3) -- (0,2.5);
\draw[dashed, gray] (3.5,3) -- (2,2.5);
\filldraw (-1.8,-0.1) node {$s_{k}$};
\filldraw (2.4,0.5) node {$t_{k}$};
\filldraw (-0.6,2.5) node {$s_{k+1}$};
\filldraw (3.9,3.1) node {$ t_{k+1}$};

%\black points
\foreach \x in {0, 2} 
\filldraw[fill=black!100,draw=black!80] (\x,0.5) circle (2.5pt)    node[anchor=north] {\small};

\foreach \x in {1}
\filldraw[fill=black!100,draw=black!80] (\x,1.5) circle (2.5pt)    node[anchor=north] {\small };

\foreach \x in {0, 2} 
\filldraw[fill=black!100,draw=black!80] (\x,2.5) circle (2.5pt)    node[anchor=north] {\small};

\end{tikzpicture} \hspace{1cm} \begin{tikzpicture}[scale=1]

\draw[red, ultra thick] (0,0) -- (0,1);
\draw[gray] (1,0) -- (1,1);
\draw[gray] (1,1) -- (0,1);
\draw[gray] (0,0) -- (1,0);

%\filldraw(0.5,0.5)     node[anchor=center] {$G_{2m-1}$};

\draw[gray] (0,1) -- (0,2);
\draw[red, ultra thick] (1,1) -- (1,2);
\draw[gray] (1,2) -- (0,2);
\draw[gray] (0,1) -- (1,1);

\draw[dashed, gray] (-1.5,-0.5) -- (0,0);
\draw[dashed, gray] (2.5,2.5) -- (1,2);

%\filldraw(0.5,1.5)     node[anchor=center] {$G_{2m}$};
\filldraw[fill=white!40,draw=black!80] (0,0) circle (2pt);
\filldraw[fill=white!40,draw=black!80] (1,0) circle (2pt);
\filldraw[fill=white!40,draw=black!80] (0,1) circle (2pt);
\filldraw[fill=white!40,draw=black!80] (1,1) circle (2pt);
\filldraw[fill=white!40,draw=black!80] (0,2) circle (2pt);
\filldraw[fill=white!40,draw=black!80] (1,2) circle (2pt);

\end{tikzpicture}
\end{center}

\begin{center}
\begin{tikzpicture}[scale=0.8]
%Arrows

\draw[gray, fill, ->, >=latex] (0,0.5) -- (0.9,1.4);
\draw[red, thick, fill, ->, >=latex] (1,1.5) -- (1.9,0.6);
\draw[gray, fill, ->, >=latex] (0,0.5) -- (1.9,0.5);

\draw[gray, fill, ->, >=latex] (1,1.5) -- (1.9,2.4);
\draw[red, thick, fill, ->, >=latex] (0,2.5) -- (0.9,1.6);
\draw[gray, fill, ->, >=latex] (0,2.5) -- (1.9,2.5);

%:) :) :)

%\draw[dashed, gray] (-1.5,0) -- (0.9,1.4);
\draw[dashed, gray] (-1.5,0) -- (0,0.5);
%\draw[dashed, gray] (-1.5,0) -- (2,0.5);

%\draw[dashed, gray] (3.5,3) -- (0.9,1.4);
%\draw[dashed, gray] (3.5,3) -- (0,2.5);
\draw[dashed, gray] (3.5,3) -- (2,2.5);
\filldraw (-1.8,-0.1) node {$s_{k}$};
\filldraw (2.4,0.5) node {$t_{k}$};
\filldraw (-0.6,2.5) node {$s_{k+1}$};
\filldraw (3.9,3.1) node {$ t_{k+1}$};

%\black points
\foreach \x in {0, 2} 
\filldraw[fill=black!100,draw=black!80] (\x,0.5) circle (2.5pt)    node[anchor=north] {\small};

\foreach \x in {1}
\filldraw[fill=black!100,draw=black!80] (\x,1.5) circle (2.5pt)    node[anchor=north] {\small };

\foreach \x in {0, 2} 
\filldraw[fill=black!100,draw=black!80] (\x,2.5) circle (2.5pt)    node[anchor=north] {\small};

\end{tikzpicture}\hspace{1cm} \begin{tikzpicture}[scale=1]

\draw[gray] (0,0) -- (0,1);
\draw[red, ultra thick] (1,0) -- (1,1);
\draw[gray] (1,1) -- (0,1);
\draw[gray] (0,0) -- (1,0);

%\filldraw(0.5,0.5)     node[anchor=center] {$G_{2m-1}$};

\draw[red, ultra thick] (0,1) -- (0,2);
\draw[gray] (1,1) -- (1,2);
\draw[gray] (1,2) -- (0,2);
\draw[gray] (0,1) -- (1,1);

%\filldraw(0.5,1.5)     node[anchor=center] {$G_{2m}$};

\draw[dashed, gray] (-1.5,-0.5) -- (0,0);
\draw[dashed, gray] (2.5,2.5) -- (1,2);

\filldraw[fill=white!40,draw=black!80] (0,0) circle (2pt);
\filldraw[fill=white!40,draw=black!80] (1,0) circle (2pt);
\filldraw[fill=white!40,draw=black!80] (0,1) circle (2pt);
\filldraw[fill=white!40,draw=black!80] (1,1) circle (2pt);
\filldraw[fill=white!40,draw=black!80] (0,2) circle (2pt);
\filldraw[fill=white!40,draw=black!80] (1,2) circle (2pt);

\end{tikzpicture}
\end{center}

\begin{center}
\begin{tikzpicture}[scale=0.8]

%Arrows

\draw[gray, fill, ->, >=latex] (0,0.5) -- (0.9,1.4);
\draw[gray, fill, ->, >=latex] (1,1.5) -- (1.9,0.6);
\draw[gray, fill, ->, >=latex] (0,0.5) -- (1.9,0.5);

\draw[red, thick, fill, ->, >=latex] (1,1.5) -- (1.9,2.4);
\draw[gray, fill, ->, >=latex] (0,2.5) -- (0.9,1.6);
\draw[gray, fill, ->, >=latex] (0,2.5) -- (1.9,2.5);

\draw[gray, fill, ->, >=latex] (-1,1.5) -- (-0.1,0.6);
\draw[red, thick, fill, ->, >=latex] (-1,1.5) -- (0.9,1.5);

%:) :) :)

%\draw[dashed, gray] (3.5,3) -- (0.9,1.4);
%\draw[dashed, gray] (3.5,3) -- (0,2.5);
\draw[dashed, gray] (3.5,3) -- (2,2.5);

\draw[dashed, gray] (-2.5,0.5) -- (-1,1.4);
%\draw[dashed, gray] (-3,0.5) -- (0,0.5);

\filldraw (-2.8,0.5) node {$s_{k}$};
\filldraw (2.4,0.5) node {$t_{k}$};
\filldraw (-0.6,2.5) node {$s_{k+1}$};
\filldraw (3.9,3.1) node {$ t_{k+1}$};

%\black points
\foreach \x in {0, 2} 
\filldraw[fill=black!100,draw=black!80] (\x,0.5) circle (2.5pt)    node[anchor=north] {\small};

\foreach \x in {-1, 1}
\filldraw[fill=black!100,draw=black!80] (\x,1.5) circle (2.5pt)    node[anchor=north] {\small };

\foreach \x in {0, 2} 
\filldraw[fill=black!100,draw=black!80] (\x,2.5) circle (2.5pt)    node[anchor=north] {\small};

\end{tikzpicture} \hspace{1cm} \begin{tikzpicture}[scale=1]

\draw[gray] (0,0) -- (0,1);
\draw[gray] (1,0) -- (1,1);
\draw[gray] (1,1) -- (0,1);
\draw[gray] (0,0) -- (1,0);

%\filldraw(0.5,0.5)     %node[anchor=center] {$G_{2m-1}$};

\draw[gray] (0,1) -- (0,2);
\draw[red, ultra thick] (1,1) -- (1,2);
\draw[gray] (1,2) -- (0,2);
\draw[gray] (0,1) -- (1,1);

%\filldraw(0.5,1.5)     %node[anchor=center] {$G_{2m}$};

\draw[red, ultra thick] (-1,1) -- (0,1);
\draw[gray] (-1,1) -- (-1,0);
\draw[gray] (-1,0) -- (0,0);

\draw[dashed, gray] (-2.5,-0.5) -- (-1,0);
\draw[dashed, gray] (2.5,2.5) -- (1,2);

\filldraw[fill=white!40,draw=black!80] (-1,0) circle (2pt);
\filldraw[fill=white!40,draw=black!80] (-1,1) circle (2pt);

\filldraw[fill=white!40,draw=black!80] (0,0) circle (2pt);
\filldraw[fill=white!40,draw=black!80] (1,0) circle (2pt);
\filldraw[fill=white!40,draw=black!80] (0,1) circle (2pt);
\filldraw[fill=white!40,draw=black!80] (1,1) circle (2pt);
\filldraw[fill=white!40,draw=black!80] (0,2) circle (2pt);
\filldraw[fill=white!40,draw=black!80] (1,2) circle (2pt);
\end{tikzpicture}
\end{center}
\caption{The 3 cases in which $p$ is a path starting at $s_i$ and ending at $t_j$, $i\neq j$.}
    \label{fig:Non-perfect_matchings}
\end{figure}
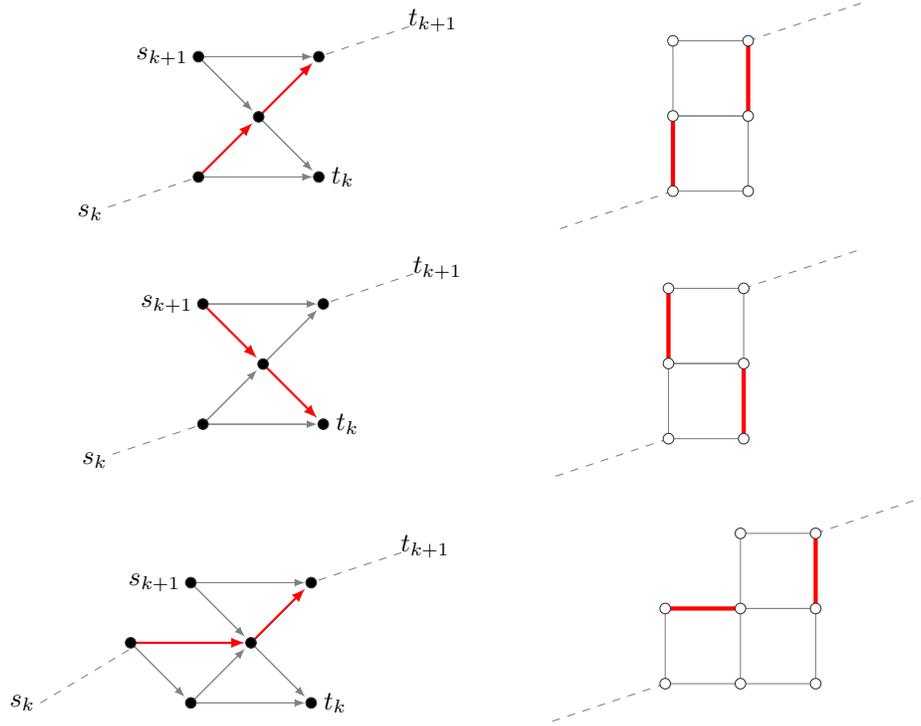
%If any of these cases happen, $P$ would have the edges shown in one of the three cases on the right of Figure~\ref{fig:Non-perfect_matchings}. By parity, we see that this is not possible because both to the right and to the left of $P$ there would be at least one unmatched vertex and this contradicts that $P$ is a perfect matching. None of these cases yields valid perfect matchings, thereby confirming that any path $p$ in the decorated snake graph tiling associated to a perfect matching $P$ starts at a source $s_i$ and ends at a sink $t_i$. 

If any of these cases were to occur, the perfect matching $P$ would have the edge configurations illustrated in one of the three cases on the right side of Figure~\ref{fig:Non-perfect_matchings}. However, by parity considerations, this situation is not possible, as both the right and left sides of $P$ would inevitably contain at least one unmatched vertex. Consequently, none of these cases results in a valid perfect matching, affirming that any path $p$ in the decorated snake graph tiling associated to a perfect matching $P$ starts at a source $s_i$ and ends at a sink $t_i$.

Moreover, consider two paths $p$ and $p^\prime$ where $p$ starts at $s_i$ and ends at $t_i$, while $p^\prime$ starts at $s_j$ and ends at $t_j$. If $s_i\neq s_j$ and $t_i\neq t_j$, their intersection at a vertex $v$ implies the existence of $\alpha$ and $\alpha^\prime$ on $p$ and $p^\prime$, respectively, such that $v$ is a common vertex of both arrows. This contradicts the requirement that the domino tiles containing $\alpha$ and $\alpha^\prime$ must be non-overlapping, as imposed by the snake graph tiling. 

Hence, we have proved that the paths in the snake graph tiling are paths exclusively from $s_i$ to $t_i$, for $i\in [k]$, and these paths are vertex-disjoint. Therefore, each perfect matching provides a unique $k$-route from $s$ to $t$ in $\mathcal{T}_\mathcal{G}$, thereby completing the proof. 
\end{proof}

\begin{remark}
It is evident that if we consider two distinct perfect matchings, their respective $k$-routes will differ. Specifically, if $e$ is an edge present in the perfect matching $P$ but absent in the perfect matching $P^\prime$, then the $k$-route $R_P$ will contain a decorated domino (either an empty, $(2,0)$, $(1,1)$ or a $(1,-1)$ domino) that is not present in $R_{P^\prime}$. In the cases where the domino in $R_P$ that is absent in $R_{P^\prime}$ is a $(2,0)$, $(1,-1)$, or $(1,1)$ domino, then the corresponding arrow will be absent in $R_{P^\prime}$. In the case where the empty edge formed by squares centered on the vertices $v_1$ and $v_2$ (where they are right-decorated and left-decorated squares, respectively) is present in $R_P$, we observe that in $R_{P^\prime}$ there exist two dominoes ($(2,0)$, $(1,1)$ or a $(1,-1)$) $D_1$ and $D_2$, in such a way that the square centered at $v_1$ is in $D_1$, and the square centered at $v_2$ is in $D_2$. Consequently, the arrows corresponding to $D_1$ and $D_2$ are present in $R_{P^\prime}$ but absent in $R_P$.
\end{remark}

\subsection{Edge contraction in snake graphs}\label{sec:edge_contraction}

Now, we are going to consider Remark~\ref{re:Basic_properties_T_G}. Again we consider a snake graph $\mathcal{G}$, which is built by the sequence of tiles $G_1, G_2,\dots, G_{d}$, with $d\geq 1$. If $d$ is an even number (resp. an odd number), then tiles $G_{2i-1}$, with $i\in \{1,2,\dots, \lfloor d/2\rfloor\}$ (resp. $i\in \{1,2,\dots, \lfloor d/2\rfloor+1\} $) will be identified with the subgraph
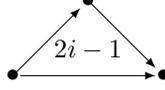
\begin{center}
\begin{tikzpicture}[domain=0:4, scale=1]
  
  \node (1) at (0,0) {$\bullet$};
  \node (1) at (1,1) {$\bullet$};
  \node (1) at (2,0) {$\bullet$};
  \draw[->, >=latex] (0.1,0) -- (1.9,0); 
  \node (1) at (1,0.35) {\small $2i-1$};
  \draw[->, >=latex] (0,0) -- (0.9,0.9); 
  \draw[->, >=latex] (1,1) -- (1.9,0.1); 
  
\end{tikzpicture}
\end{center}
and the tiles $G_{2i}$, with $i\in \{1,2,\dots, \lfloor d/2\rfloor\}$ will be identified with the subgraph

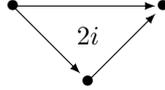
\begin{center}
\begin{tikzpicture}[domain=0:4, scale=1]

  \node (1) at (4,1.5) {$\bullet$};
  \node (1) at (3,2.5) {$\bullet$};
  \node (1) at (5,2.5) {$\bullet$};
  \node (1) at (4,2.1) {$2i$};
  \draw[->, >=latex] (4,1.5) -- (4.9,2.4); 
  \draw[->, >=latex] (3,2.5) -- (3.9,1.6);
  \draw[->, >=latex] (3,2.5) -- (4.9,2.5); 
\end{tikzpicture}
\end{center}

Then the acyclic-directed graph $\mathcal{T}_\mathcal{G}$ is formed by the following rules:
    
\begin{enumerate}
    \item If the tiles $G_j$ and $G_{j+1}$ of the snake graph $\mathcal{G}$ are joined by their east and west sides, then the respective triangles labeled by $j$ and $j+1$ are joined by the arrows corresponding to the steps $(1,1)$ and $(1,-1)$ as follows:
    
%Rule #1

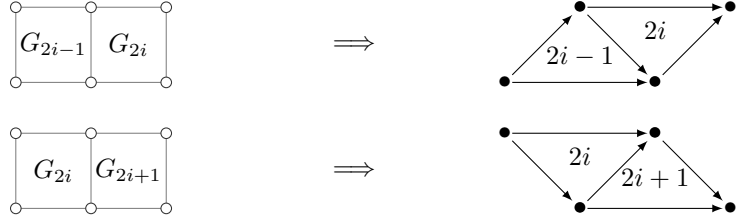
\begin{center}
\begin{tikzpicture}[scale=1.2]

\draw[gray] (0,1) -- (0,1);
\draw[gray] (0,0) -- (0,1);
\draw[gray] (1,1) -- (0,1);
\draw[gray] (0,0) -- (1,0);

\filldraw(0.5,0.5)     node[anchor=center] {$G_{2i-1}$};

\filldraw(4.5,0.5)     node[anchor=center] {$\Longrightarrow$};

\draw[gray] (1,0) -- (1,1);
\draw[gray] (2,0) -- (2,1);
\draw[gray] (2,1) -- (1,1);
\draw[gray] (1,0) -- (2,0);

\filldraw(1.5,0.5)     node[anchor=center] {$G_{2i}$};

  \node (1) at (7.5,1) {$\bullet$};
  \node (1) at (8.5,0) {$\bullet$};
  \node (1) at (9.5,1) {$\bullet$};
  \node (1) at (8.5,0.7) {$2i$};
  \draw[->, >=latex] (8.6,0.1) -- (9.4,0.9); 
  \draw[->, >=latex] (7.6,0.9) -- (8.4,0.1);
  \draw[->, >=latex] (7.6,1) -- (9.4,1);
  
  \node (1) at (6.5,0) {$\bullet$};
  \node (1) at (7.5,1) {$\bullet$};
  \node (1) at (8.5,0) {$\bullet$};
  \node (1) at (7.5,0.35) {\small $2i-1$};
  \draw[->, >=latex] (6.6,0) -- (8.4,0); 
  \draw[->, >=latex] (6.6,0.1) -- (7.4,0.9); %\draw[->, >=latex] (7.6,4.9) -- (8.4,4.1);

\filldraw[fill=white!40,draw=black!80] (0,0) circle (2pt);
\filldraw[fill=white!40,draw=black!80] (1,0) circle (2pt);
\filldraw[fill=white!40,draw=black!80] (2,0) circle (2pt);
\filldraw[fill=white!40,draw=black!80] (0,1) circle (2pt);
\filldraw[fill=white!40,draw=black!80] (1,1) circle (2pt);
\filldraw[fill=white!40,draw=black!80] (2,1) circle (2pt);
\end{tikzpicture}
\end{center} 
    
%Rule #2

\begin{center}
\begin{tikzpicture}[scale=1.2]

\draw[gray] (0,1) -- (0,1);
\draw[gray] (0,0) -- (0,1);
\draw[gray] (1,1) -- (0,1);
\draw[gray] (0,0) -- (1,0);

\filldraw(0.5,0.5)     node[anchor=center] {$G_{2i}$};

\filldraw(4.5,0.5)     node[anchor=center] {$\Longrightarrow$};

\draw[gray] (1,0) -- (1,1);
\draw[gray] (2,0) -- (2,1);
\draw[gray] (2,1) -- (1,1);
\draw[gray] (1,0) -- (2,0);

\filldraw(1.5,0.5)     node[anchor=center] {$G_{2i+1}$};

   \node (1) at (6.5,1) {$\bullet$};
  \node (1) at (7.5,0) {$\bullet$};
  \node (1) at (8.5,1) {$\bullet$};
  \node (1) at (7.5,0.7) {$2i$};
  \draw[->, >=latex] (7.6,0.1) -- (8.4,0.9); 
  \draw[->, >=latex] (6.6,0.9) -- (7.4,0.1);
  \draw[->, >=latex] (6.6,1) -- (8.4,1);
  
  \node (1) at (7.5,0) {$\bullet$};
  \node (1) at (8.5,1) {$\bullet$};
  \node (1) at (9.5,0) {$\bullet$};
  \node (1) at (8.5,0.4) {\small $2i+1$};
  \draw[->, >=latex] (7.6,0) -- (9.4,0); 
%  \draw[->, >=latex] (7.6,4.1) -- (8.4,4.9); 
  \draw[->, >=latex] (8.6,0.9) -- (9.4,0.1);

\filldraw[fill=white!40,draw=black!80] (0,0) circle (2pt);
\filldraw[fill=white!40,draw=black!80] (1,0) circle (2pt);
\filldraw[fill=white!40,draw=black!80] (2,0) circle (2pt);
\filldraw[fill=white!40,draw=black!80] (0,1) circle (2pt);
\filldraw[fill=white!40,draw=black!80] (1,1) circle (2pt);
\filldraw[fill=white!40,draw=black!80] (2,1) circle (2pt);  
\end{tikzpicture}
\end{center} 
 
    \item If the tiles $G_j$ and $G_{j+1}$ of the snake graph $\mathcal{G}$ are joined by their north and south sides, then the respective triangles labeled by $j$ and $j+1$ are joined by their north and south arrows (corresponding to steps $(1,0)$) or by their north and south vertices, as follows:

%Rule #3

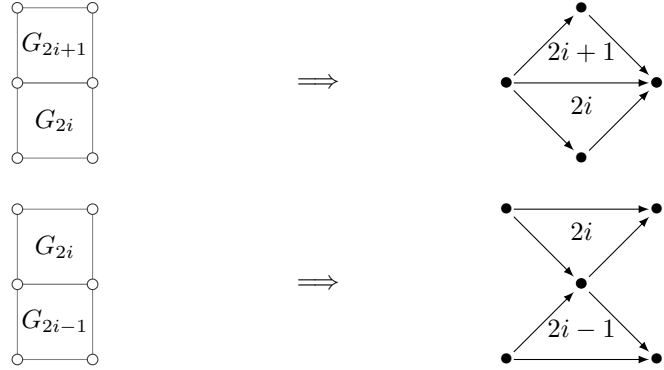
\begin{center}
\begin{tikzpicture}[scale=1.2]

\draw[gray] (0,0) -- (0,1);
\draw[gray] (1,0) -- (1,1);
\draw[gray] (1,1) -- (0,1);
\draw[gray] (0,0) -- (1,0);

\filldraw(0.5,0.5)     node[anchor=center] {$G_{2i}$};

\draw[gray] (0,1) -- (0,2);
\draw[gray] (1,1) -- (1,2);
\draw[gray] (1,2) -- (0,2);
\draw[gray] (0,1) -- (1,1);

\filldraw(0.5,1.5)     node[anchor=center] {$G_{2i+1}$};

\filldraw(4,1)     node[anchor=center] {$\Longrightarrow$};

  \node (1) at (6.5,1) {$\bullet$};
  \node (1) at (7.5,2) {$\bullet$};
  \node (1) at (8.5,1) {$\bullet$};
  \node (1) at (7.5,1.4) {\small $2i+1$};
  \draw[->, >=latex] (6.6,1) -- (8.4,1); 
  \draw[->, >=latex] (6.6,1.1) -- (7.4,1.9); 
  \draw[->, >=latex] (7.6,1.9) -- (8.4,1.1);
  
  \node (1) at (6.5,1) {$\bullet$};
  \node (1) at (7.5,0) {$\bullet$};
  \node (1) at (8.5,1) {$\bullet$};
  \node (1) at (7.5,0.7) {$2i$};
  \draw[->, >=latex] (7.6,0.1) -- (8.4,0.9); 
  \draw[->, >=latex] (6.6,0.9) -- (7.4,0.1);
  %\draw[->, >=latex] (6.6,1) -- (8.4,1); 

\filldraw[fill=white!40,draw=black!80] (0,0) circle (2pt);
\filldraw[fill=white!40,draw=black!80] (1,0) circle (2pt);
\filldraw[fill=white!40,draw=black!80] (0,1) circle (2pt);
\filldraw[fill=white!40,draw=black!80] (1,1) circle (2pt);
\filldraw[fill=white!40,draw=black!80] (0,2) circle (2pt);
\filldraw[fill=white!40,draw=black!80] (1,2) circle (2pt);
\end{tikzpicture}
\end{center}

%Rule #4

\begin{center}
\begin{tikzpicture}[scale=1.2]

\draw[gray] (0,0) -- (0,1);
\draw[gray] (1,0) -- (1,1);
\draw[gray] (1,1) -- (0,1);
\draw[gray] (0,0) -- (1,0);

\filldraw(0.5,0.5)     node[anchor=center] {$G_{2i-1}$};

\draw[gray] (0,1) -- (0,2);
\draw[gray] (1,1) -- (1,2);
\draw[gray] (1,2) -- (0,2);
\draw[gray] (0,1) -- (1,1);

\filldraw(0.5,1.5)     node[anchor=center] {$G_{2i}$};

\filldraw(4,1)     node[anchor=center] {$\Longrightarrow$};

  \node (1) at (6.5,0) {$\bullet$};
  \node (1) at (7.5,1) {$\bullet$};
  \node (1) at (8.5,0) {$\bullet$};
  \node (1) at (7.5,0.4) {\small $2i-1$};
  \draw[->, >=latex] (6.6,0) -- (8.4,0); 
  \draw[->, >=latex] (6.6,0.1) -- (7.4,0.9); 
  \draw[->, >=latex] (7.6,0.9) -- (8.4,0.1);
  
  \node (1) at (6.5,2) {$\bullet$};
  \node (1) at (7.5,1) {$\bullet$};
  \node (1) at (8.5,2) {$\bullet$};
  \node (1) at (7.5,1.7) {$2i$};
  \draw[->, >=latex] (7.6,1.1) -- (8.4,1.9); 
  \draw[->, >=latex] (6.6,1.9) -- (7.4,1.1);
  \draw[->, >=latex] (6.6,2) -- (8.4,2); 
  
\filldraw[fill=white!40,draw=black!80] (0,0) circle (2pt);
\filldraw[fill=white!40,draw=black!80] (1,0) circle (2pt);
\filldraw[fill=white!40,draw=black!80] (0,1) circle (2pt);
\filldraw[fill=white!40,draw=black!80] (1,1) circle (2pt);
\filldraw[fill=white!40,draw=black!80] (0,2) circle (2pt);
\filldraw[fill=white!40,draw=black!80] (1,2) circle (2pt);
\end{tikzpicture}
\end{center}
\end{enumerate}

We can see that the construction of $\mathcal{T}_\mathcal{G}$ is very similar to that of $\mathcal{G}$, except for the shape of the tiles (the square tiles are changed to triangular tiles). Therefore, making use of the operation in graph theory called edge contraction, and thinking about maintaining the information alluding to the cluster variables seen in \S~\ref{subsection:Perfect matchings and snake graphs}, we are going to make a different construction than the one introduced in \S~\ref{sec:Lattice_paths_and_perfect_matchings}.

\begin{definition}\label{def:Edge_contraction}
Let $\mathcal{G}$ a snake graph and let $G_i=(V_i,E_i)$ a tile of $\mathcal{G}$, where $V_i=\{v_p, v_q, v_r, v_s\}$ and $E_i=\{e_{pq}, e_{qr}, e_{rs}, e_{sp}\}$: 
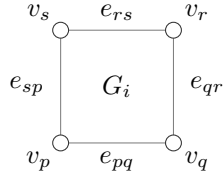
\begin{center}
\begin{tikzpicture}[y=.3cm, x=.3cm,font=\normalsize, scale=1]
\draw[gray] (0,0) -- (0,5);
\draw[gray] (5,0) -- (5,5);
\draw[gray] (5,5) -- (0,5);
\draw[gray] (0,0) -- (5,0);

\filldraw[fill=white!40,draw=black!80] (0,0) circle (3pt);
\filldraw[fill=white!40,draw=black!80] (0,5) circle (3pt);
\filldraw[fill=white!40,draw=black!80] (5,0) circle (3pt);
\filldraw[fill=white!40,draw=black!80] (5,5) circle (3pt);

\filldraw(0,-0.8)     node[anchor=east] {$v_{p}$};
\filldraw(7,-0.8)     node[anchor=east] {$v_{q}$};
\filldraw(7,5.8)     node[anchor=east] {$v_{r}$};
\filldraw(0,5.8)     node[anchor=east] {$v_{s}$};

\filldraw(2.5,2.5)     node[anchor=center] {$G_i$};

\filldraw(2.5,5.8)     node {$e_{rs}$};
\filldraw(2.5,-0.8)     node {$e_{pq}$};
\filldraw(6.5,2.5)     node {$e_{qr}$};
\filldraw(-1.5,2.5)     node {$e_{sp}$};

\end{tikzpicture}
\end{center}

Let $f_{rs}$ be a function that maps every vertex in $V_i-\{v_r,v_s\}$ to itself, and otherwise maps $v_r$ and $v_s$ to a new vertex $\textbf{e}_{rs}$. We define the \emph{north contraction} of $G_i$ as a new graph $f_{rs}(G_i)=(\hat{V}_i,\hat{E}_i)$, where 
\[
\hat{V}_i=(V_i-\{v_r,v_s\})\cup \{\textbf{e}_{rs}\} \quad \quad \text{and} \quad \quad \hat{E}_i=E_i-\{e_{rs}\},
\]
and for every $v\in V_i$, we have that $f_{rs}(v)\in \hat{V}_i$ is incident to an edge $e\in \hat{E}_i\subset E_i$ if and only if $e$ is incident to $v$ in $G_i$.

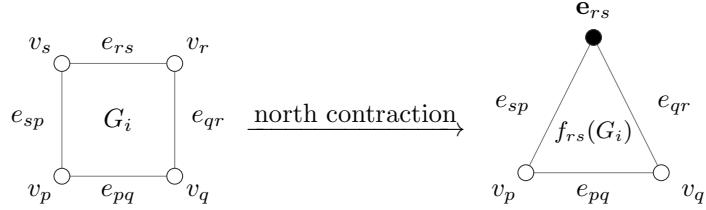
\begin{center}
\begin{tikzpicture}[y=.3cm, x=.3cm,font=\normalsize, scale=1]
\draw[gray] (0,0) -- (0,5);
\draw[gray] (5,0) -- (5,5);
\draw[gray] (5,5) -- (0,5);
\draw[gray] (0,0) -- (5,0);

\filldraw[fill=white!40,draw=black!80] (0,0) circle (3pt);
\filldraw[fill=white!40,draw=black!80] (0,5) circle (3pt);
\filldraw[fill=white!40,draw=black!80] (5,0) circle (3pt);
\filldraw[fill=white!40,draw=black!80] (5,5) circle (3pt);

\filldraw(0,-0.8)     node[anchor=east] {$v_{p}$};
\filldraw(7,-0.8)     node[anchor=east] {$v_{q}$};
\filldraw(7,5.8)     node[anchor=east] {$v_{r}$};
\filldraw(0,5.8)     node[anchor=east] {$v_{s}$};

\filldraw(2.5,2.5)     node[anchor=center] {$G_i$};

\filldraw(2.5,5.8)     node {$e_{rs}$};
\filldraw(2.5,-0.8)     node {$e_{pq}$};
\filldraw(6.5,2.5)     node {$e_{qr}$};
\filldraw(-1.5,2.5)     node {$e_{sp}$};

\filldraw(14.5,2.5)     node {$\underrightarrow{\text{\; north contraction \; }}$};
\end{tikzpicture} \begin{tikzpicture}[y=.3cm, x=.3cm,font=\normalsize, scale=1.2]
\draw[gray] (0,0) -- (2.5,5);
\draw[gray] (5,0) -- (2.5,5);
\draw[gray] (0,0) -- (5,0);

\filldraw[fill=white!40,draw=black!80] (0,0) circle (2.5pt);
\filldraw[fill=white!40,draw=black!80] (5,0) circle (2.5pt);
\filldraw[fill=black!100,draw=black!80] (2.5,5) circle (2.5pt);

\filldraw(0,-0.8)     node[anchor=east] {$v_{p}$};
\filldraw(7,-0.8)     node[anchor=east] {$v_{q}$};

\filldraw(2.5,1.5)     node[anchor=center] {\small $f_{rs}(G_i)$};

\filldraw(2.5,6)     node {$\textbf{e}_{rs}$};
\filldraw(2.5,-0.8)     node {$e_{pq}$};
\filldraw(5.5,2.5)     node {$e_{qr}$};
\filldraw(-0.5,2.5)     node {$e_{sp}$};

\end{tikzpicture}
\end{center}

Analogously, let $f_{pq}$ be a function that maps every vertex in $V_i-\{v_p,v_q\}$ to itself, and otherwise, maps $v_p$ and $v_q$ to a new vertex $\textbf{e}_{pq}$. We define the \emph{south contraction} of $G_i$ as a new graph $f_{pq}(G_i)=(\check{V}_i,\check{E}_i)$, where 
\[
\check{V}_i=(V_i-\{v_p,v_q\})\cup \{\textbf{e}_{pq}\} \quad \quad \text{and} \quad \quad \check{E_i}=E_i-\{e_{pq}\},
\]
and for every $v\in V_i$, we have that $f_{pq}(v)\in \check{V}_i$ is incident to an edge $e\in \check{E}_i\subset E_i$ if and only if $e$ is incident to $v$ in $G_i$.

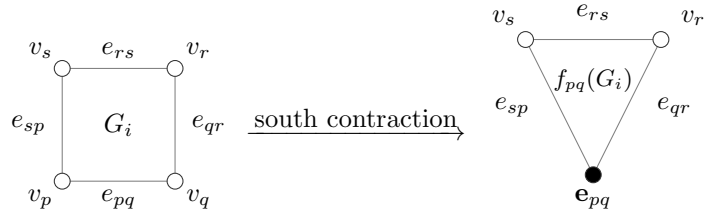
\begin{center}
\begin{tikzpicture}[y=.3cm, x=.3cm,font=\normalsize, scale=1]
\draw[gray] (0,0) -- (0,5);
\draw[gray] (5,0) -- (5,5);
\draw[gray] (5,5) -- (0,5);
\draw[gray] (0,0) -- (5,0);

\filldraw[fill=white!40,draw=black!80] (0,0) circle (3pt);
\filldraw[fill=white!40,draw=black!80] (0,5) circle (3pt);
\filldraw[fill=white!40,draw=black!80] (5,0) circle (3pt);
\filldraw[fill=white!40,draw=black!80] (5,5) circle (3pt);

\filldraw(0,-0.8)     node[anchor=east] {$v_{p}$};
\filldraw(7,-0.8)     node[anchor=east] {$v_{q}$};
\filldraw(7,5.8)     node[anchor=east] {$v_{r}$};
\filldraw(0,5.8)     node[anchor=east] {$v_{s}$};

\filldraw(2.5,2.5)     node[anchor=center] {$G_i$};

\filldraw(2.5,5.8)     node {$e_{rs}$};
\filldraw(2.5,-0.8)     node {$e_{pq}$};
\filldraw(6.5,2.5)     node {$e_{qr}$};
\filldraw(-1.5,2.5)     node {$e_{sp}$};

\filldraw(14.5,2.5)     node {$\underrightarrow{\text{\; south contraction \; 
 }}$};
\end{tikzpicture} \begin{tikzpicture}[y=.3cm, x=.3cm,font=\normalsize, scale=1.2]
\draw[gray] (0,5) -- (2.5,0);
\draw[gray] (5,5) -- (2.5,0);
\draw[gray] (0,5) -- (5,5);

\filldraw[fill=white!40,draw=black!80] (0,5) circle (2.5pt);
\filldraw[fill=white!40,draw=black!80] (5,5) circle (2.5pt);
\filldraw[fill=black!100,draw=black!80] (2.5,0) circle (2.5pt);

\filldraw(7,5.8)     node[anchor=east] {$v_{r}$};
\filldraw(0,5.8)     node[anchor=east] {$v_{s}$};

\filldraw(2.5,3.5)     node[anchor=center] {\small $f_{pq}(G_i)$};

\filldraw(2.5,6)     node {$e_{rs}$};
\filldraw(2.5,-0.9)     node {$\textbf{e}_{pq}$};
\filldraw(5.5,2.5)     node {$e_{qr}$};
\filldraw(-0.5,2.5)     node {$e_{sp}$};

\end{tikzpicture}
\end{center}
\end{definition}

If no confusion arises, any north contraction (resp. south contraction) will be called $f_N$ (resp. $f_S$).

\begin{definition}\label{def:contracted_snake_graph}
Let $\mathcal{G}$ be a snake graph with $d$ tiles and interior edges $e_1, e_2,\dots,e_{d-1}$. The \emph{contracted snake graph} associated to $\mathcal{G}$ is a connected planar graph consisting of a finite sequence of tiles $T_1, T_2, \dots, T_d$ in such a way that 
\begin{enumerate}
    \item $T_i=f_N(G_i)$ if $i$ is odd, and otherwise, $T_i=f_S(G_i)$ if $i$ is even.
    \item For each $i\in [d-1]$, $T_i$ and $T_{i+1}$ share exactly the edge $e_i$ or the vertex $\mathbf{e}_i$. 
\end{enumerate}
%when we apply a north contraction to the tiles $G_i$ of $\mathcal{G}$, where $i$ is odd, and apply south contraction to the tiles $G_j$ of $\mathcal{G}$, where $j$ is even. 

The \emph{oriented contracted snake graph} is obtained when we put orientations to the edges of the contracted snake graph with the condition that the arrows are oriented from left to right.
\end{definition}

\begin{remark}
In order to generate equivalent lattice paths in the oriented contracted snake graph, as discussed in \S~\ref{sec:Lattice_paths_and_perfect_matchings} for $\mathcal{T}_\mathcal{G}$ and characterized by steps from the set 
\[\{(2,0),(1,1),(1,-1)\},\] 
we can consider an isomorphic snake graph to $\mathcal{G}$ in Definition~\ref{def:contracted_snake_graph} with tiles $G_i$ as $1\times 2$ rectangles.  %\carolina{an approach involves treating the tiles of $\mathcal{G}$ as $1\times 2$-rectangles}. 
Consequently, the oriented contracted snake graph of $\mathcal{G}$ is canonically isomorphic to the triangular snake graph $\mathcal{T}_\mathcal{G}$. Therefore, we systematically identify these graphs. Hence we can think of $\mathcal{T}_\mathcal{G}$ either as a triangular snake graph or as an oriented contracted snake graph. The $k$-vertices $s=(s_1,\dots,s_k)$ and $t=(t_1,\dots,t_k)$ (as defined in Definition~\ref{def:k-vertices}) are also considered within this context.
%Therefore, we opt to employ the same notation, $\mathcal{T}_\mathcal{G}$, for the oriented contracted snake graph. The $k$-vertices $s=(s_1,\dots,s_k)$ and $t=(t_1,\dots,t_k)$ (as defined in Definition~\ref{def:k-vertices}) are also considered within this context.

%To obtain the same lattice paths that in \S\ref{sec:Lattice_paths_and_perfect_matchings} for $\mathcal{T}_\mathcal{G}$ (steps in the set $\{(2,0),(1,1),(1,-1)\}$), we can consider the tiles of $\mathcal{G}$ as $1\times 2$-rectangles. Then, the oriented contracted snake graph of $\mathcal{G}$ is essentially the same as the triangular snake graph of $\mathcal{G}$ (if the snake graph labels are ignored) so we will use the same notation of triangular snake graphs, $\mathcal{T}_\mathcal{G}$, for the oriented contracted snake graph.
\end{remark}
\begin{lemma}\label{lemma:characterization_contraction_edges}
    Let $\mathcal{T}_\mathcal{G}$ be the oriented contracted snake graph of $\mathcal{G}$. The sinks and sources are the only vertices that are not the result of contractions of edges of $\mathcal{G}$. 
\end{lemma}

\begin{proof}
    %It is immediate from the Definition~\ref{def:Edge_contraction}. If $\mathbf{e_{pq}}$ is a contraction of the edge $e_{pq}$ in a tile $G_i$ of $\mathcal{G}$, then the west and east edges of $G_i$ are incidents to $\mathbf{e_{pq}}$, where the west arrow into $\mathbf{e_{pq}}$ and the east arrow out to $\mathbf{e_{pq}}$. Then, the vertices obtained by the edge contraction cannot be sources or sinks.

    This directly follows from Definition~\ref{def:Edge_contraction}. If $\mathbf{e}_{pq}$ represents the contraction of the edge $e_{pq}$ within a tile $G_i$ of $\mathcal{G}$, then the west and east edges of $G_i$ are incidents to $\mathbf{e}_{pq}$, with the west arrow directed into $\mathbf{e}_{pq}$ and the east arrow directed out of $\mathbf{e}_{pq}$. Consequently, the vertices obtained through edge contractions cannot serve as sources or sinks. Moreover, any other vertex is obtained by an edge contraction as it is either the top or the bottom of a triangular tile.
\end{proof}

\begin{remark}\label{re:decontraction}
A \emph{decontraction} process can be applied to $\mathcal{T}_\mathcal{G}$ to reconstruct the snake graph $\mathcal{G}$ from $\mathcal{T}_\mathcal{G}$. According to Lemma~\ref{lemma:characterization_contraction_edges}, the sources, sinks, and their incident arrows remain unchanged. However, for each vertex $\mathbf{e}$ in $\mathcal{T}_\mathcal{G}$ that is neither a source nor a sink, two new vertices, $v_e$ and $v_e^\prime$, along with the edge $e$, are introduced. The arrows that end at $\mathbf{e}$ become incident to the vertex $v_e$, while those that start at $\mathbf{e}$ will be incident to the vertex $v_e^{\prime}$. Formally, this process is defined as follows: Let $\mathcal{T}_\mathcal{G}=(\overline{V},\overline{E})$ be the oriented contracted snake graph associated to $\mathcal{G}$. Denote by $S$ the set of sources and sinks in $\mathcal{T}_\mathcal{G}$. Let $h$ be a function that maps every vertex $\mathbf{e}$ in $V-S$ to a new arrow 
$\xymatrix{v_{\textbf{e}} \ar[r]^{e} & v_{\textbf{e}}^\prime}$. Define the sets:
\[
V^\prime=\bigcup_{\mathbf{e}\in \overline{V}-S} \{v_{\mathbf{e}},v_{\mathbf{e}}^\prime\} \quad \text{and} \quad E^\prime=\bigcup_{\mathbf{e}\in \overline{V}-S}\{h(\mathbf{e})\}=\bigcup_{\mathbf{e}\in \overline{V}-S}\{e\}.
\]
%Define the sets:
%\[
%V^\prime=\bigcup_{\mathbf{e}\in \overline{V}-S} \{v_{\mathbf{e}},v_{\mathbf{e}}^\prime\} \quad \text{and} \quad E^\prime=\bigcup_{\mathbf{e}\in \overline{V}-S}\{e\}.
%\] Then, $\overrightarrow{\mathcal{G}}=(V,E)$ is such that $V=(\overline{V}-S)\cup V^\prime$ and $E=\overline{E}\cup E^\prime$.

Then, $\overrightarrow{\mathcal{G}}=(V,E)$ is such that $V=(\overline{V}-S)\cup V^\prime$ and $E=\overline{E}\cup E^\prime$. For every arrow $\alpha \in \overline{E}$ such that $s(\alpha)=\mathbf{e}$, the corresponding arrow $\alpha \in E$ satisfies $s(\alpha)=v_{\mathbf{e}}$; otherwise, if $\alpha\in \overline{E}$ is such that $t(\alpha)=\mathbf{e}$, then $\alpha \in E$ satisfies $t(\alpha)=v_{\mathbf{e}}^\prime$. The underlying graph of $\overrightarrow{\mathcal{G}}$ (i.e., without considering the direction of the arrows) is the snake graph $\mathcal{G}$.
\end{remark}

\begin{theorem}\label{the:Perfect_matchings_and_Routes}
Let $\mathcal{G}$ be a snake graph. The set of perfect matchings $\text{Match}(\mathcal{G})$ of $\mathcal{G}$ is in bijection with the set of $k$-routes from $s$ to $t$ in $\mathcal{T}_\mathcal{G}$. 
\end{theorem}

\begin{proof}
According to Lemma~\ref{lemma:Perfect_matchings_to_routes}, every perfect matching $P\in \text{Match}(\mathcal{G})$ gives rise to a $k$-route from $s$ to $t$ in $\mathcal{T}_\mathcal{G}$. It is sufficient to prove that for every $k$-route $R$, we can obtain a perfect matching $P_R$ and that different $k$-routes give rise to different perfect matchings. We think of $\mathcal{T}_\mathcal{G}$ as the oriented contracted snake graph associated to $\mathcal{G}$ and consider the decontraction process and the notation described in Remark~\ref{re:decontraction}. 
Consider a $k$-route $R$ from $s$ to $t$ in $\mathcal{T}_\mathcal{G}$. We consider the following subsets of arrows and vertices in $\mathcal{T}_\mathcal{G}$:
\begin{enumerate}
    \item Let $T\subseteq \overline{E}$ be the set of all arrows $\alpha$ in $\mathcal{T}_\mathcal{G}$ such that $\alpha$ belongs to a path $p$ of $R$.
    \item Let $U\subseteq \overline{V}-S$ be the set of all non-source or non-sink vertices $\mathbf{e}$ in $\mathcal{T}_\mathcal{G}$ such that $\mathbf{e}$ does not belong to a path of $R$.
\end{enumerate}

We claim that the edges associated to $T$ and $U$ obtained after the decontraction process form a perfect matching $P_R$ of $\mathcal{G}$. First, we will show that all vertices in $\mathcal{G}$ belong to an edge of the decontraction process of $T\cup U$. This is evident because the vertices of $\mathcal{T}_\mathcal{G}$ are sources, sinks, or contractions of edges. In the first two cases, for every source $s_i$ and sink $t_i$ there exists a path in the $k$-route $R$ starting at $s_i$ and ending at $t_i$. Therefore, all sources and sinks belong to the edges obtained by the decontraction process of $T$. In the third case, every vertex $\mathbf{e}$ obtained by contraction of an edge belongs to the edges obtained by the decontraction process of $T$, provided $\mathbf{e}$ belongs to a path of $R$; or $\mathbf{e}$ belongs to the edges obtained by the decontraction process of $U$, in case $\mathbf{e}$ does not belong to a path of $R$.  

Now, we aim to show that for any two distinct edges $e$ and $e^\prime$ in the decontraction process of $T\cup U$, $e$ and $e^\prime$ do not share common vertices. For every path 
\[ 
p_i=
\xymatrix { s_i
\ar@[red][r]^{\alpha_1}& \mathbf{e_1} \ar@[red][r]^{\alpha_2}& \mathbf{e_2} \ar@[red][r]^{\alpha_3}& \cdots \ar@[red][r]^{\alpha_{l-1}}& \mathbf{e_{l-1}} \ar@[red][r]^{\alpha_l}& t_i
}
\]
in $R$, we obtain the following edges when we apply the decontraction process. 
\[
\xymatrix { s_i
\ar@{-}@[red][r]^{\alpha_1}& v_{\mathbf{e_1}} \ar@{-}[r]^{e_1} & v_{\mathbf{e_1}}^\prime \ar@{-}@[red][r]^{\alpha_2}& v_{\mathbf{e_2}} \ar@{-}[r]^{e_2} & v_{\mathbf{e_2}}^\prime \ar@{-}@[red][r]^{\alpha_3}& \cdots \ar@{-}@[red][r]^{\alpha_{l-1}}& v_{\mathbf{e_{l-1}}} \ar@{-}[r]^{e_{l-1}} & v_{\mathbf{e_{l-1}}}^\prime \ar@{-}@[red][r]^{\alpha_l}& t_i
}
\]
where the red edges $\alpha_1,\alpha_2,\dots,\alpha_l$ are edges in $P_R$, while the gray arrows $e_1,e_2,\dots, e_{l-1}$ are not. There is no other edge in $P_R$ that shares a vertex with the edges $\alpha_1,\dots,\alpha_l$. Otherwise, there would exist a path $p_j$ in $R$ that shares a vertex with $p_i$, it would contradict the fact that $R$ is a $k$-route. 

On the other hand, for every edge contraction vertex $\mathbf{e}$ that does not belong to a path of $R$

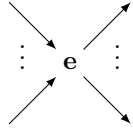
\begin{center}
\begin{tikzpicture}[domain=0:4, scale=0.9]

  \node (1) at (5,8) {$\mathbf{e}$};
  \draw[->, >=latex] (4.1,7.1) -- (4.8,7.8); 
  \draw[->, >=latex] (5.2,7.8) -- (5.9,7.1);
  
  \draw[->, >=latex] (5.2,8.2) -- (5.9,8.9); 
  \draw[->, >=latex] (4.1,8.9) -- (4.8,8.2);

  \node (1) at (5.7,8.2) {$\vdots$};
  \node (1) at (4.3,8.2) {$\vdots$};
  
\end{tikzpicture}
\end{center}
we obtain the following edges when we apply the decontraction process
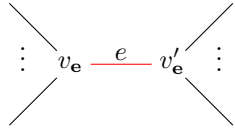
\begin{center}
\begin{tikzpicture}[domain=0:4, scale=0.9]

  \node (1) at (5,8) {$v_\mathbf{e}$};
  \node (1) at (6.5,8.05) {$v_\mathbf{e}^\prime$};
  \draw (4.1,7.1) -- (4.8,7.8); 
  \draw (6.7,7.8) -- (7.4,7.1);
  
  \draw (6.7,8.2) -- (7.4,8.9); 
  \draw (4.1,8.9) -- (4.8,8.2);

  \draw[red] (5.3,8) -- (6.2,8);
  \node (1) at (5.75,8.15) {$e$};

  \node (1) at (7.2,8.2) {$\vdots$};
  \node (1) at (4.3,8.2) {$\vdots$}; 
\end{tikzpicture}
\end{center}
where the red edge $e$ is an edge in $P_R$, and all incident edges to $e$ do not belong to $P_R$. Therefore, $P_R$ is a perfect matching of $\mathcal{G}$. 

Finally, let us consider two distinct $k$-routes, denoted as $R$ and $R^\prime$. We observe that their corresponding perfect matchings, namely $P_R$ and $P_{R^\prime}$, are distinct. This follows directly from the existence of an arrow $\alpha$ in a path of $R$ that is absent in any path of $R^\prime$. According to the definitions of $P_R$ and $P_{R^\prime}$, the edge corresponding to $\alpha$ is included in $P_R$ but not in $P_{R^\prime}$. Consequently, $P_R$ and $P_{R^\prime}$ differ.
\end{proof}

\begin{corollary}
The number of perfect matchings $m(\mathcal{G})$ of the snake graph $\mathcal{G}$ is equal to the determinant of the path matrix of $\mathcal{T}_\mathcal{G}$ between the $k$-vertices $s = (s_1,\dots, s_k )$ and $t = (t_1,\dots, t_k)$.
\end{corollary}

\begin{proof}
     The result follows immediately from Theorem~\ref{the:Perfect_matchings_and_Routes} and Lemma~\ref{lem:LGV}.
\end{proof}

This result establishes a strong connection between combinatorial structures in snake graphs and determinant formulas. In the next section, we explore how these path matrices are related to Hankel matrices, a fundamental class of matrices that encode classical number sequences such as Fibonacci and Catalan numbers. By leveraging these connections, we derive new determinant identities that reveal deeper algebraic properties of snake graphs.

\begin{example}\label{ex:contracted snake graph}
Consider the snake graph of Example~\ref{ex:SnakeGraphTilings}. Its contracted snake graph is depicted as follows.

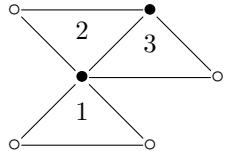
\begin{center}
\begin{tikzpicture}[domain=0:4, scale=0.9]

% dim M = 111  
  %\node (1) at (2,8) {$\mathcal{T}_{\mathcal{G}} \quad = $};
  
  \node (1) at (4,7) {$\circ$};
  \node (1) at (5,8) {$\bullet$};
  \node (1) at (6,7) {$\circ$};
  \node (1) at (5,7.5) {$1$};
  \draw (4.1,7) -- (5.9,7); 
  \draw (4.1,7.1) -- (4.9,7.9); 
  \draw (5.1,7.9) -- (5.9,7.1);
  
  \node (1) at (4,9) {$\circ$};
  \node (1) at (5,8) {$\bullet$};
  \node (1) at (6,9) {$\bullet$};
  \node (1) at (5,8.7) {$2$};
  \draw (5.1,8.1) -- (5.9,8.9); 
  \draw (4.1,8.9) -- (4.9,8.1);
  \draw (4.1,9) -- (5.9,9); 
  
  \node (1) at (5,8) {$\bullet$};
  \node (1) at (6,9) {$\bullet$};
  \node (1) at (7,8) {$\circ$};
  \node (1) at (6,8.5) {$3$};
  \draw (5.1,8) -- (6.9,8);
  \draw (6.1,8.9) -- (6.9,8.1);
  
\end{tikzpicture}
\end{center}
and its oriented contracted snake graph and associated path matrix are illustrated as follows:

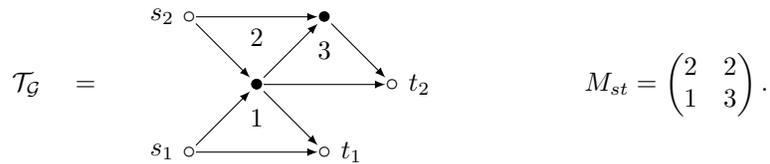
\begin{center}
\begin{tikzpicture}[domain=0:4, scale=0.9]

% dim M = 111  
  \node (1) at (2,8) {$\mathcal{T}_\mathcal{G} \quad = $};
  
    \node (1) at (11,8) {\quad $M_{st}=\begin{pmatrix}
    2 & 2 \\
    1 & 3
    \end{pmatrix}.$};

  \node (1) at (4,7) {$\circ$};
  \node (1) at (5,8) {$\bullet$};
  \node (1) at (6,7) {$\circ$};
  \node (1) at (5,7.5) {$1$};
  \draw[->, >=latex] (4.1,7) -- (5.9,7); 
  \draw[->, >=latex] (4.1,7.1) -- (4.9,7.9); 
  \draw[->, >=latex] (5.1,7.9) -- (5.9,7.1);
  
  \node (1) at (4,9) {$\circ$};
  \node (1) at (5,8) {$\bullet$};
  \node (1) at (6,9) {$\bullet$};
  \node (1) at (5,8.7) {$2$};
  \draw[->, >=latex] (5.1,8.1) -- (5.9,8.9); 
  \draw[->, >=latex] (4.1,8.9) -- (4.9,8.1);
  \draw[->, >=latex] (4.1,9) -- (5.9,9); 
  
    \node (1) at (5,8) {$\bullet$};
  \node (1) at (6,9) {$\bullet$};
  \node (1) at (7,8) {$\circ$};
  \node (1) at (6,8.5) {$3$};
  \draw[->, >=latex] (5.1,8) -- (6.9,8);
  \draw[->, >=latex] (6.1,8.9) -- (6.9,8.1);
  
   \node (1) at (3.6,7) {$s_1$};
   \node (1) at (3.6,9) {$s_2$};
   \node (1) at (6.4,7) {$t_1$};
   \node (1) at (7.4,8) {$t_2$};
\end{tikzpicture}
\end{center}

 The number of perfect matchings equals the determinant of $M_{st}$. There exist 4 perfect matchings in $\mathcal{G}$, each corresponding to distinct $2$-routes in $\mathcal{T}_\mathcal{G}$, which are:

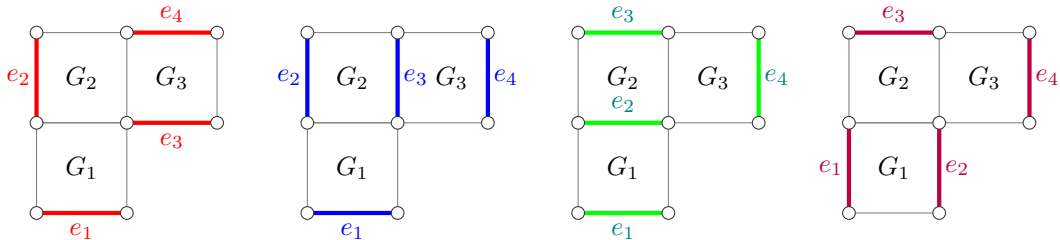
\begin{center}
\begin{tikzpicture}[scale=1.2]

% 1

\draw[gray] (0,0) -- (0,1);
\draw[gray] (1,0) -- (1,1);
\draw[gray] (1,1) -- (0,1);
\draw[red, ultra thick] (0,0) -- (1,0);

\filldraw(0.5,0.5)     node[anchor=center] {$G_1$};

\draw[red, ultra thick] (0,1) -- (0,2);
\draw[gray] (1,1) -- (1,2);
\draw[gray] (1,2) -- (0,2);
\draw[gray] (0,1) -- (1,1);

\filldraw(0.5,1.5)     node[anchor=center] {$G_{2}$};

\draw[gray] (2,1) -- (2,2);
\draw[red, ultra thick] (1,2) -- (2,2);
\draw[red, ultra thick] (2,1) -- (1,1);

\filldraw(1.5,1.5)     node[anchor=center] {$G_{3}$};

\filldraw[fill=white!40,draw=black!80] (0,0) circle (2pt);
\filldraw[fill=white!40,draw=black!80] (1,0) circle (2pt);
\filldraw[fill=white!40,draw=black!80] (0,1) circle (2pt);
\filldraw[fill=white!40,draw=black!80] (1,1) circle (2pt);
\filldraw[fill=white!40,draw=black!80] (2,1) circle (2pt);
\filldraw[fill=white!40,draw=black!80] (0,2) circle (2pt);
\filldraw[fill=white!40,draw=black!80] (1,2) circle (2pt);
\filldraw[fill=white!40,draw=black!80] (2,2) circle (2pt);

\filldraw(0.5,-0.2) node[anchor=center] {\textcolor{red}{$e_1$}};
\filldraw(-0.2,1.5) node[anchor=center] {\textcolor{red}{$e_2$}};
\filldraw(1.5,0.8) node[anchor=center] {\textcolor{red}{$e_3$}};
\filldraw(1.5,2.2) node[anchor=center] {\textcolor{red}{$e_4$}};

% y_3

\draw[gray] (3,0) -- (3,1);
\draw[gray] (4,0) -- (4,1);
\draw[gray] (4,1) -- (3,1);
\draw[blue, ultra thick] (3,0) -- (4,0);

\filldraw(3.5,0.5)     node[anchor=center] {$G_1$};

\draw[blue, ultra thick] (3,1) -- (3,2);
\draw[blue, ultra thick] (4,1) -- (4,2);
\draw[gray] (4,2) -- (3,2);
\draw[gray] (3,1) -- (4,1);

\filldraw(3.5,1.5)     node[anchor=center] {$G_{2}$};

\draw[blue, ultra thick] (5,1) -- (5,2);
\draw[gray] (4,2) -- (5,2);
\draw[gray] (5,1) -- (4,1);

\filldraw(4.6,1.5)     node[anchor=center] {$G_{3}$};

\filldraw[fill=white!40,draw=black!80] (3,0) circle (2pt);
\filldraw[fill=white!40,draw=black!80] (4,0) circle (2pt);
\filldraw[fill=white!40,draw=black!80] (3,1) circle (2pt);
\filldraw[fill=white!40,draw=black!80] (4,1) circle (2pt);
\filldraw[fill=white!40,draw=black!80] (5,1) circle (2pt);
\filldraw[fill=white!40,draw=black!80] (3,2) circle (2pt);
\filldraw[fill=white!40,draw=black!80] (4,2) circle (2pt);
\filldraw[fill=white!40,draw=black!80] (5,2) circle (2pt);

\filldraw(3.5,-0.2) node[anchor=center] {\textcolor{blue}{$e_1$}};
\filldraw(2.8,1.5) node[anchor=center] {\textcolor{blue}{$e_2$}};
\filldraw(4.2,1.5) node[anchor=center] {\textcolor{blue}{$e_3$}};
\filldraw(5.2,1.5) node[anchor=center] {\textcolor{blue}{$e_4$}};

% y_2y_3

\draw[gray] (6,0) -- (6,1);
\draw[gray] (7,0) -- (7,1);
\draw[green, ultra thick] (7,1) -- (6,1);
\draw[green, ultra thick] (6,0) -- (7,0);

\filldraw(6.5,0.5)     node[anchor=center] {$G_1$};

\draw[gray] (6,1) -- (6,2);
\draw[gray] (7,1) -- (7,2);
\draw[green, ultra thick] (7,2) -- (6,2);

\filldraw(6.5,1.5)     node[anchor=center] {$G_{2}$};

\draw[green, ultra thick] (8,1) -- (8,2);
\draw[gray] (7,2) -- (8,2);
\draw[gray] (8,1) -- (7,1);

\filldraw(7.5,1.5)     node[anchor=center] {$G_{3}$};

\filldraw[fill=white!40,draw=black!80] (6,0) circle (2pt);
\filldraw[fill=white!40,draw=black!80] (7,0) circle (2pt);
\filldraw[fill=white!40,draw=black!80] (6,1) circle (2pt);
\filldraw[fill=white!40,draw=black!80] (7,1) circle (2pt);
\filldraw[fill=white!40,draw=black!80] (8,1) circle (2pt);
\filldraw[fill=white!40,draw=black!80] (6,2) circle (2pt);
\filldraw[fill=white!40,draw=black!80] (7,2) circle (2pt);
\filldraw[fill=white!40,draw=black!80] (8,2) circle (2pt);

\filldraw(6.5,-0.2) node[anchor=center] {\textcolor{teal}{$e_1$}};
\filldraw(6.5,1.2) node[anchor=center] {\textcolor{teal}{$e_2$}};
\filldraw(6.5,2.2) node[anchor=center] {\textcolor{teal}{$e_3$}};
\filldraw(8.2,1.5) node[anchor=center] {\textcolor{teal}{$e_4$}};

% y_1y_2y_3

\draw[purple, ultra thick] (9,0) -- (9,1);
\draw[purple, ultra thick] (10,0) -- (10,1);
\draw[gray] (10,1) -- (9,1);
\draw[gray] (9,0) -- (10,0);

\filldraw(9.5,0.5)     node[anchor=center] {$G_1$};

\draw[gray] (9,1) -- (9,2);
\draw[gray] (10,1) -- (10,2);
\draw[purple, ultra thick] (10,2) -- (9,2);
\draw[gray] (9,1) -- (10,1);

\filldraw(9.5,1.5)     node[anchor=center] {$G_{2}$};

\draw[purple, ultra thick] (11,1) -- (11,2);
\draw[gray] (10,2) -- (11,2);
\draw[gray] (11,1) -- (10,1);

\filldraw(10.5,1.5)     node[anchor=center] {$G_{3}$};

\filldraw[fill=white!40,draw=black!80] (9,0) circle (2pt);
\filldraw[fill=white!40,draw=black!80] (10,0) circle (2pt);
\filldraw[fill=white!40,draw=black!80] (9,1) circle (2pt);
\filldraw[fill=white!40,draw=black!80] (10,1) circle (2pt);
\filldraw[fill=white!40,draw=black!80] (11,1) circle (2pt);
\filldraw[fill=white!40,draw=black!80] (9,2) circle (2pt);
\filldraw[fill=white!40,draw=black!80] (10,2) circle (2pt);
\filldraw[fill=white!40,draw=black!80] (11,2) circle (2pt);

\filldraw(8.8,0.5) node[anchor=center] {\textcolor{purple}{$e_1$}};
\filldraw(10.2,0.5) node[anchor=center] {\textcolor{purple}{$e_2$}};
\filldraw(9.5,2.2) node[anchor=center] {\textcolor{purple}{$e_3$}};
\filldraw(11.2,1.5) node[anchor=center] {\textcolor{purple}{$e_4$}};

\end{tikzpicture} 
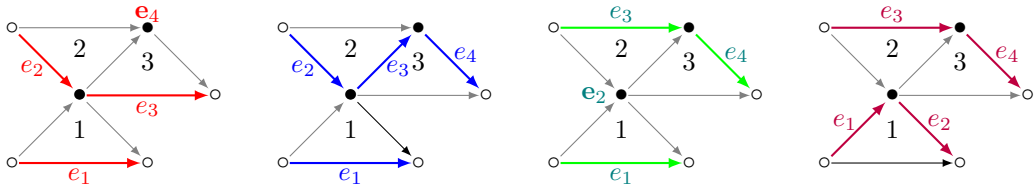
\par 
\vspace{1cm} \par
\begin{tikzpicture}[domain=0:4, scale=0.9]
% dim M = 111

  \node (1) at (4,7) {$\circ$};
  \node (1) at (5,8) {$\bullet$};
  \node (1) at (6,7) {$\circ$};
  \node (1) at (5,7.5) {$1$};
  \draw[->, >=latex, red, thick] (4.1,7) -- (5.9,7); 
  \draw[->, >=latex, gray] (4.1,7.1) -- (4.9,7.9); 
  \draw[->, >=latex, gray] (5.1,7.9) -- (5.9,7.1);
  
  \node (1) at (4,9) {$\circ$};
  \node (1) at (5,8) {$\bullet$};
  \node (1) at (6,9) {$\bullet$};
  \node (1) at (5,8.7) {$2$};
  \draw[->, >=latex, gray] (5.1,8.1) -- (5.9,8.9); 
  \draw[->, >=latex, red, thick] (4.1,8.9) -- (4.9,8.1);
  \draw[->, >=latex, gray] (4.1,9) -- (5.9,9); 
  
  \node (1) at (7,8) {$\circ$};
  \node (1) at (6,8.5) {$3$};
  \draw[->, >=latex, red, thick] (5.1,8) -- (6.9,8);
  \draw[->, >=latex, gray] (6.1,8.9) -- (6.9,8.1);

\filldraw(5,6.8) node[anchor=center] {\textcolor{red}{$e_1$}};
\filldraw(4.3,8.4) node[anchor=center] {\textcolor{red}{$e_2$}};
\filldraw(6,7.8) node[anchor=center] {\textcolor{red}{$e_3$}};
\filldraw(6,9.2) node[anchor=center] {\textcolor{red}{$\textbf{e}_4$}};

  % dim M = 111
  
  \node (1) at (8,7) {$\circ$};
  \node (1) at (9,8) {$\bullet$};
  \node (1) at (10,7) {$\circ$};
  \node (1) at (9,7.5) {$1$};
  \draw[->, >=latex, blue, thick] (8.1,7) -- (9.9,7); 
  \draw[->, >=latex, gray] (8.1,7.1) -- (8.9,7.9); 
  \draw[->, >=latex] (9.1,7.9) -- (9.9,7.1);
  
  \node (1) at (8,9) {$\circ$};
  \node (1) at (9,8) {$\bullet$};
  \node (1) at (10,9) {$\bullet$};
  \node (1) at (9,8.7) {$2$};
  \draw[->, >=latex, blue, thick] (9.1,8.1) -- (9.9,8.9); 
  \draw[->, >=latex, blue, thick] (8.1,8.9) -- (8.9,8.1);
  \draw[->, >=latex, gray] (8.1,9) -- (9.9,9); 
  
  \node (1) at (11,8) {$\circ$};
  \node (1) at (10,8.5) {$3$};
  \draw[->, >=latex, gray] (9.1,8) -- (10.9,8);
  \draw[->, >=latex, blue, thick] (10.1,8.9) -- (10.9,8.1);

\filldraw(9,6.8) node[anchor=center] {\textcolor{blue}{$e_1$}};
\filldraw(8.3,8.4) node[anchor=center] {\textcolor{blue}{$e_2$}};
\filldraw(9.7,8.4) node[anchor=center] {\textcolor{blue}{$e_3$}};
\filldraw(10.7,8.6) node[anchor=center] {\textcolor{blue}{$e_4$}};

  % dim M = 111
  
  \node (1) at (12,7) {$\circ$};
  \node (1) at (13,8) {$\bullet$};
  \node (1) at (14,7) {$\circ$};
  \node (1) at (13,7.5) {$1$};
  \draw[->, >=latex, green, thick] (12.1,7) -- (13.9,7); 
  \draw[->, >=latex, gray] (12.1,7.1) -- (12.9,7.9); 
  \draw[->, >=latex, gray] (13.1,7.9) -- (13.9,7.1);
  
  \node (1) at (12,9) {$\circ$};
  \node (1) at (14,9) {$\bullet$};
  \node (1) at (13,8.7) {$2$};
  \draw[->, >=latex, gray] (13.1,8.1) -- (13.9,8.9); 
  \draw[->, >=latex, gray] (12.1,8.9) -- (12.9,8.1);
  \draw[->, >=latex, green, thick] (12.1,9) -- (13.9,9); 
  
  \node (1) at (15,8) {$\circ$};
  \node (1) at (14,8.5) {$3$};
  \draw[->, >=latex, gray] (13.1,8) -- (14.9,8);
  \draw[->, >=latex, green, thick] (14.1,8.9) -- (14.9,8.1);

\filldraw(13,6.8) node[anchor=center] {\textcolor{teal}{$e_1$}};
\filldraw(12.6,8) node[anchor=center] {\textcolor{teal}{$\textbf{e}_2$}};
\filldraw(13,9.2) node[anchor=center] {\textcolor{teal}{$e_3$}};
\filldraw(14.7,8.6) node[anchor=center] {\textcolor{teal}{$e_4$}};

  % dim M = 111
  
  \node (1) at (16,7) {$\circ$};
  \node (1) at (17,8) {$\bullet$};
  \node (1) at (18,7) {$\circ$};
  \node (1) at (17,7.5) {$1$};
  \draw[->, >=latex] (16.1,7) -- (17.9,7); 
  \draw[->, >=latex, purple, thick] (16.1,7.1) -- (16.9,7.9); 
  \draw[->, >=latex, purple, thick] (17.1,7.9) -- (17.9,7.1);
  
  \node (1) at (16,9) {$\circ$};
  \node (1) at (18,9) {$\bullet$};
  \node (1) at (17,8.7) {$2$};
  \draw[->, >=latex, gray] (17.1,8.1) -- (17.9,8.9); 
  \draw[->, >=latex, gray] (16.1,8.9) -- (16.9,8.1);
  \draw[->, >=latex, purple, thick] (16.1,9) -- (17.9,9); 
  
  \node (1) at (19,8) {$\circ$};
  \node (1) at (18,8.5) {$3$};
  \draw[->, >=latex, gray] (17.1,8) -- (18.9,8);
  \draw[->, >=latex, purple, thick] (18.1,8.9) -- (18.9,8.1);

\filldraw(16.3,7.6) node[anchor=center] {\textcolor{purple}{$e_1$}};
\filldraw(17.7,7.6) node[anchor=center] {\textcolor{purple}{$e_2$}};
\filldraw(17,9.2) node[anchor=center] {\textcolor{purple}{$e_3$}};
\filldraw(18.7,8.6) node[anchor=center] {\textcolor{purple}{$e_4$}};
\end{tikzpicture}
\end{center}

We observe that when examining the contracted snake graph corresponding to a perfect matching, it aligns with the $k$-route immediately below it. Likewise, applying the decontraction process to a $k$-route yields the perfect matching that lies above it.
\end{example}

\begin{remark}\label{re:opposite edge contraction}
    Note that in Definition~\ref{def:contracted_snake_graph}, the assignment of northern and southern contractions to odd and even indices is arbitrary. One could equally well consider the opposite assignment, where northern contractions are associated with even indices and southern contractions with odd indices. Although the resulting contracted snake graphs are different with potentially varying numbers of sources and sinks, the number of routes remains invariant. Moreover, the bijection between the perfect matchings and the routes of the contracted snake graph obtained by this opposite assignment can still be established. Throughout this paper, we will consider both constructions, as the combinatorial properties of one may be more convenient than the other, depending on the desired formulas for counting paths or routes.
\end{remark}

\begin{example}
Consider the snake graph of Example~\ref{ex:SnakeGraphTilings}. Its contracted snake graph $\mathcal{T}_{\mathcal{G}}$ is depicted in Example~\ref{ex:contracted snake graph}. If we consider the opposite assignment, where northern contractions are associated with even indices and southern contractions with odd indices, we obtain the following triangular snake graph $\mathcal{T}_{\mathcal{G}}^{op}$ and corresponding routes

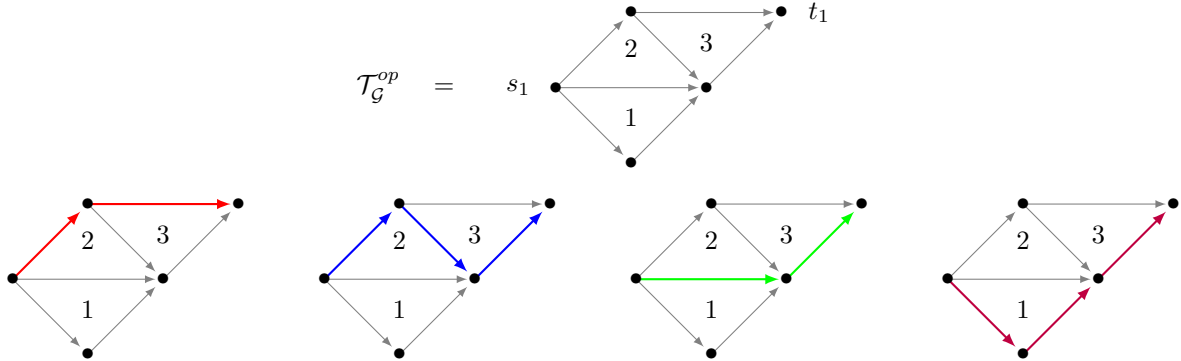
\begin{center}
\begin{tikzpicture}[domain=0:4, scale=2]

\draw[->, >=latex, gray] (1,1) -- (1.45,0.55);
\draw[->, >=latex, gray] (1,1) -- (1.95,1); 
\draw[->, >=latex, gray] (1.5,0.5) -- (1.95,0.95); 

\draw[->, >=latex, gray] (1,1) -- (1.45,1.45);
\draw[->, >=latex, gray] (1.5,1.5) -- (1.95,1.05);

\draw[->, >=latex, gray] (2,1) -- (2.45,1.45); 
\draw[->, >=latex, gray] (1.5,1.5) -- (2.45,1.5); 

%:::::::::::::::::::::::::::::::::::::::::::::::::::::::::::::::::::::::::::::::::

\node (1) at (1,1) {$\bullet$};
\node (1) at (1.5,1.5) {$\bullet$};
\node (1) at (2.5,1.5) {$\bullet$};
\node (1) at (1.5,0.5) {$\bullet$};
\node (1) at (2,1) {$\bullet$};
\node (1) at (0.75,1) {$s_1$};
\node (1) at (2.75,1.5) {$t_1$};

\filldraw(1.5,0.8)     node[anchor=center] {$1$};
\filldraw(1.5,1.26)     node[anchor=center] {$2$};
\filldraw(2,1.3)     node[anchor=center] {$3$};

\node (1) at (0,1) {$\mathcal{T}_{\mathcal{G}}^{op} \quad =$};

\end{tikzpicture}
\end{center}
%:):):):):):):):):):):):):):):):):):):):):):):):):):):):):):):):):):):):):):):):):)
\begin{center}
\begin{tikzpicture}[domain=0:4, scale=2]

\draw[->, >=latex, gray] (1,1) -- (1.45,0.55);
\draw[->, >=latex, gray] (1,1) -- (1.95,1); 
\draw[->, >=latex, gray] (1.5,0.5) -- (1.95,0.95); 

\draw[->, >=latex, thick, red] (1,1) -- (1.45,1.45);
\draw[->, >=latex, gray] (1.5,1.5) -- (1.95,1.05);

\draw[->, >=latex, gray] (2,1) -- (2.45,1.45); 
\draw[->, >=latex, thick, red] (1.5,1.5) -- (2.45,1.5); 

%:::::::::::::::::::::::::::::::::::::::::::::::::::::::::::::::::::::::::::::::::

\node (1) at (1,1) {$\bullet$};
\node (1) at (1.5,1.5) {$\bullet$};
\node (1) at (2.5,1.5) {$\bullet$};
\node (1) at (1.5,0.5) {$\bullet$};
\node (1) at (2,1) {$\bullet$};

\filldraw(1.5,0.8)     node[anchor=center] {$1$};
\filldraw(1.5,1.26)     node[anchor=center] {$2$};
\filldraw(2,1.3)     node[anchor=center] {$3$};

\end{tikzpicture} \hspace{0.5cm} \begin{tikzpicture}[domain=0:4, scale=2]

\draw[->, >=latex, gray] (1,1) -- (1.45,0.55);
\draw[->, >=latex, gray] (1,1) -- (1.95,1); 
\draw[->, >=latex, gray] (1.5,0.5) -- (1.95,0.95); 

\draw[->, >=latex, thick, blue] (1,1) -- (1.45,1.45);
\draw[->, >=latex, thick, blue] (1.5,1.5) -- (1.95,1.05);

\draw[->, >=latex, thick, blue] (2,1) -- (2.45,1.45); 
\draw[->, >=latex, gray] (1.5,1.5) -- (2.45,1.5); 

%:::::::::::::::::::::::::::::::::::::::::::::::::::::::::::::::::::::::::::::::::

\node (1) at (1,1) {$\bullet$};
\node (1) at (1.5,1.5) {$\bullet$};
\node (1) at (2.5,1.5) {$\bullet$};
\node (1) at (1.5,0.5) {$\bullet$};
\node (1) at (2,1) {$\bullet$};

\filldraw(1.5,0.8)     node[anchor=center] {$1$};
\filldraw(1.5,1.26)     node[anchor=center] {$2$};
\filldraw(2,1.3)     node[anchor=center] {$3$};

\end{tikzpicture} \hspace{0.5cm} \begin{tikzpicture}[domain=0:4, scale=2]

\draw[->, >=latex, gray] (1,1) -- (1.45,0.55);
\draw[->, >=latex, thick, green] (1,1) -- (1.95,1); 
\draw[->, >=latex, gray] (1.5,0.5) -- (1.95,0.95); 

\draw[->, >=latex, gray] (1,1) -- (1.45,1.45);
\draw[->, >=latex, gray] (1.5,1.5) -- (1.95,1.05);

\draw[->, >=latex, thick, green] (2,1) -- (2.45,1.45); 
\draw[->, >=latex, gray] (1.5,1.5) -- (2.45,1.5); 

%:::::::::::::::::::::::::::::::::::::::::::::::::::::::::::::::::::::::::::::::::

\node (1) at (1,1) {$\bullet$};
\node (1) at (1.5,1.5) {$\bullet$};
\node (1) at (2.5,1.5) {$\bullet$};
\node (1) at (1.5,0.5) {$\bullet$};
\node (1) at (2,1) {$\bullet$};

\filldraw(1.5,0.8)     node[anchor=center] {$1$};
\filldraw(1.5,1.26)     node[anchor=center] {$2$};
\filldraw(2,1.3)     node[anchor=center] {$3$};

\end{tikzpicture} \hspace{0.5cm} \begin{tikzpicture}[domain=0:4, scale=2]

\draw[->, >=latex, thick, purple] (1,1) -- (1.45,0.55);
\draw[->, >=latex, gray] (1,1) -- (1.95,1); 
\draw[->, >=latex, thick, purple] (1.5,0.5) -- (1.95,0.95); 

\draw[->, >=latex, gray] (1,1) -- (1.45,1.45);
\draw[->, >=latex, gray] (1.5,1.5) -- (1.95,1.05);

\draw[->, >=latex, thick, purple] (2,1) -- (2.45,1.45); 
\draw[->, >=latex, gray] (1.5,1.5) -- (2.45,1.5); 

%:::::::::::::::::::::::::::::::::::::::::::::::::::::::::::::::::::::::::::::::::

\node (1) at (1,1) {$\bullet$};
\node (1) at (1.5,1.5) {$\bullet$};
\node (1) at (2.5,1.5) {$\bullet$};
\node (1) at (1.5,0.5) {$\bullet$};
\node (1) at (2,1) {$\bullet$};

\filldraw(1.5,0.8)     node[anchor=center] {$1$};
\filldraw(1.5,1.26)     node[anchor=center] {$2$};
\filldraw(2,1.3)     node[anchor=center] {$3$};

\end{tikzpicture}

\end{center}

\end{example}

%:::::::::::::::::::::::::::::::::::::::::::::::::::::::::::::::::::::::::::::::::::::::::::::::::::::::::::::::::::::::::::::::::::::::::::::::::::::
\section{Fibonacci identities and Hankel matrices}\label{sec:Fibonacci identities and Hankel matrices}

\subsection{Identities in straight snake graphs}
\label{sec:ladder_graphs_hankel_and_general_formula}

In this section, our primary objective is to identify matrices exhibiting an "equivalence" to particular families of Hankel matrices based on lattice paths. This exploration intends to present instances where simpler acyclic graphs can be identified. These graphs provide a more straightforward approach for determining the entries of the path matrix and computing specific determinants. In particular, in this subsection, we work with a special class of snake graphs known as ladder graphs before addressing the general case. We fix the following notation:

\begin{definition} 
A \emph{ladder graph} $L_n$ is defined as a straight snake graph with $n$ tiles. 
\end{definition}

Following the construction from the previous section, in the next proposition, we consider a directed acyclic graph that contains as a subgraph the triangular snake graph associated with a ladder graph. Later, we will see that this subgraph preserves the same number of paths as the graph considered in the proposition derived from Aztec diamonds.

\begin{proposition}\label{prop:Catalan_Fibonacci} Let $C$ be the sequence of Catalan numbers $C_n$ and let $F_n$ be the $n$-th Fibonacci number. Then the following relationship holds:
\[
\det (H_k(C)+H_k^{\prime}(C))=F_{2k+1}, \]
where $H_{k}(C)$ and $H_{k}^{\prime}(C)$ are the Hankel matrices of the Catalan numbers.
\end{proposition}

\begin{proof}
The main idea of this proof is to identify a suitable graph that enables us to employ Lemma~\ref{lem:LGV}, thereby equating the requested determinant with the number of paths in the said graph. Notably, each component of the matrix counts the number of lattice paths in $\mathbb{Z}^2$ that start at $(0,0)$ and end at $(2n,0)$, and the number of lattice paths that start at $(0,0)$ and end at $(2n+2,0)$. Observe that in the acyclic-directed graph in Figure \ref{fig:Catalan acyclic-directed graph general}, there are precisely $C_{i+j-2 }+C_{i+j-1}$ paths from the vertex $s_i$ to the vertex $t_j$. 
\begin{figure}[ht]
\begin{center}
\begin{tikzpicture}[y=.3cm, x=.3cm,font=\normalsize, scale=0.5,rotate=-45]

\draw[gray,dashed] (29,14) -- (26,11);
\draw[gray,dashed] (-4,-19) -- (-1,-16);
\draw[->, >=latex] (-5,15) -- (0,15);

\draw[->, >=latex] (-5,10) -- (-0.6,14.6);
\draw[->, >=latex] (0,5) -- (4.6,9.6);
\draw[->, >=latex] (5,0) -- (9.6,4.6);
\draw[->, >=latex] (10,-5) -- (14.6,-0.6);

\foreach \y in {-15,-10,...,10}
 {\draw[gray,dashed] (-4,\y) -- (-1,\y);}
 
\foreach \y in {-10,-5,...,15,20}
 {\draw[->, >=latex] (-5,\y-10) -- (-5,\y-5);} 
 
\foreach \x in {0,5,...,25}
 {\draw[gray,dashed] (\x,14) -- (\x,11);} 
 %::::::::::::::::::::::::::::::::::::::::::::::::::::::::::::::::: 

  \foreach \x in {0,5,...,25}
 {         \draw[->, >=latex] (\x,15) -- (\x+5,15);    
         } 
         
  \foreach \x in {0,5,...,20}
 {         \draw[->, >=latex] (\x,10) -- (\x+5,10);
           \draw[->, >=latex] (\x,5) -- (\x,10);
      
         } 
         
  \foreach \x in {0,5,...,15}
 {         \draw[->, >=latex] (\x,5) -- (\x+5,5);
           \draw[->, >=latex] (\x,0) -- (\x,5);
      
         } 
           
  \foreach \x in {0,5,10}
 {         \draw[->, >=latex] (\x,0) -- (\x+5,0);
           \draw[->, >=latex] (\x,-5) -- (\x,0);
      
         } 
         
  \foreach \x in {0,5}
 {         \draw[->, >=latex] (\x,-5) -- (\x+5,-5);
           \draw[->, >=latex] (\x,-10) -- (\x,-5);
      
         }          

 \draw[->, >=latex] (0,-10) -- (5,-10);
 \draw[->, >=latex] (0,-15) -- (0,-10);

\node (4) at (32,15) {$t_n$};
\node (5) at (27,10) {$t_3$};
\node (6) at (22,5) {$t_2$};
\node (7) at (17,0) {$t_1$};
\node (8) at (10,-7) {$s_1$};
\node (9) at (5,-12) {$s_2$};
\node (10) at (0,-17) {$s_3$};
\node (11) at (-5,-22) {$s_n$};
\end{tikzpicture}
\end{center}
    \caption{Catalan acyclic-directed graph associated to $n$-routes.}
\label{fig:Catalan acyclic-directed graph general}
\end{figure}
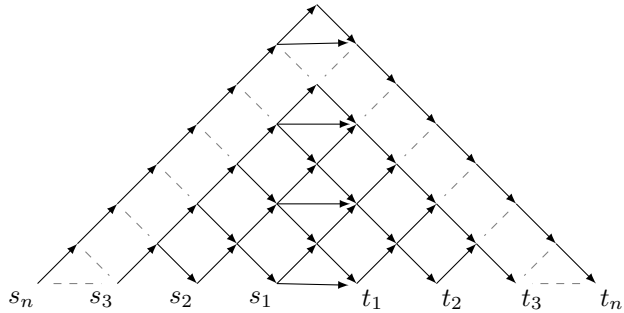
Therefore, the requested determinant counts the $n$-routes from $s$ to $t$ in the aforementioned graph. These paths correspond to the Aztec diamonds defined in \ref{subsection:Aztec diamonds and Schröder paths}, which, in turn, are linked to perfect matchings. Consequently, the result follows immediately.
\end{proof}

\begin{example} \label{Ex:L4_Aztec_Diamond}
For the five perfect matchings of $L_3$ we obtain the following domino tiling of $Az(2)$:
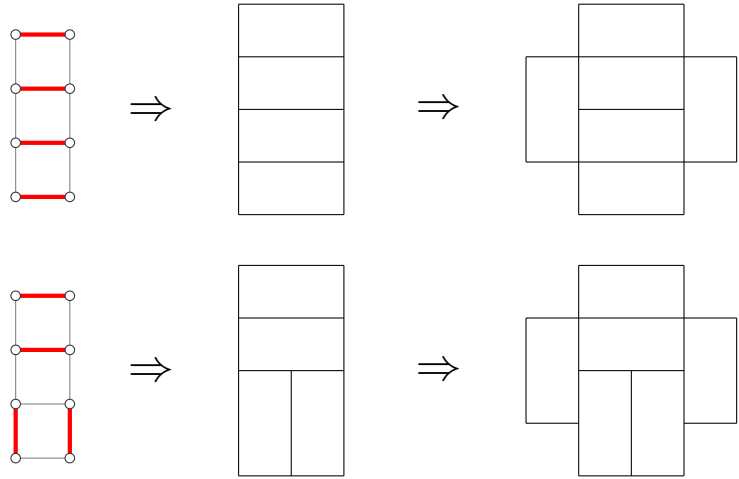
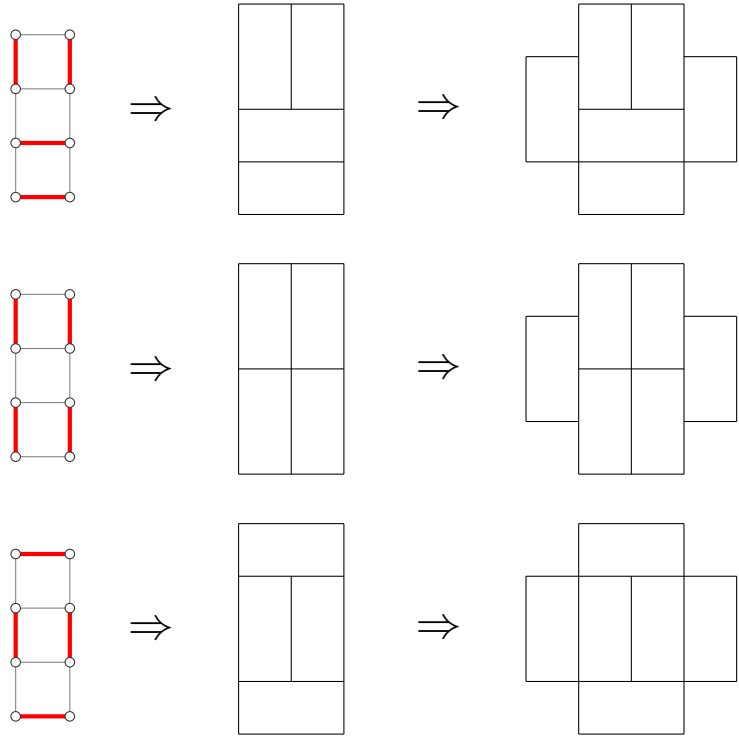
\begin{center}
\begin{tikzpicture}[y=.3cm, x=.3cm,font=\normalsize,scale=0.6]
\foreach \i in {0,1} \draw [gray] (4*\i,0) -- (4*\i,4);
\foreach \i in {0,1} \draw [gray] (4*\i,4) -- (4*\i,8);
\foreach \i in {0,1} \draw [gray] (4*\i,8) -- (4*\i,12);
\foreach \i in {0,...,3} \draw [red, ultra thick] (0,4*\i) -- (4,4*\i);
\foreach \i in {0,...,3} \filldraw[fill=white!40,draw=black!80] (0,4*\i) circle (3pt)    node[anchor=north] {\small };
\foreach \i in {0,...,3} \filldraw[fill=white!40,draw=black!80] (4,4*\i) circle (3pt)    node[anchor=north] {\small };

\filldraw[fill=white!40,draw=white!80] (10,5) circle (3pt)    node[anchor=south] {\huge $\Rightarrow$};
\end{tikzpicture} \hspace{0.5cm}
\begin{tikzpicture}[domain=0:4, scale=0.7]
  \draw (3,0) -- (5,0);
  \draw (3,2) -- (3,-2);
  \draw (5,-2) -- (5,2);
  \draw (3,2) -- (5,2);
  \draw (3,-2) -- (5,-2);
  \draw (3,-1) -- (5,-1);
  \draw (3,1) -- (5,1);
  
\filldraw[fill=white!40,draw=white!80] (6.8,-0.35) circle (3pt)    node[anchor=south] {\huge $\Rightarrow$};
\end{tikzpicture} \hspace{0.5cm} \begin{tikzpicture}[domain=0:4, scale=0.7]
  \draw (6,-1) -- (6,1);
  \draw (2,-1) -- (2,1);
  \draw (2,1) -- (3,1);
  \draw (3,0) -- (5,0);
  \draw (5,1) -- (6,1);
  \draw (6,-1) -- (5,-1);
  \draw (3,-1) -- (2,-1);
  \draw (3,2) -- (3,-2);
  \draw (5,-2) -- (5,2);
  \draw (3,2) -- (5,2);
  \draw (3,-2) -- (5,-2);
  \draw (3,-1) -- (5,-1);
  \draw (3,1) -- (5,1);
\end{tikzpicture}
\end{center}
\vspace{0.25cm}
\begin{center}
\begin{tikzpicture}[y=.3cm, x=.3cm,font=\normalsize,scale=0.6]
\foreach \i in {0,1} \draw [gray] (4*\i,0) -- (4*\i,4);
\foreach \i in {0,1} \draw [gray] (4*\i,4) -- (4*\i,8);
\foreach \i in {0,1} \draw [gray] (4*\i,8) -- (4*\i,12);
\foreach \i in {0,...,3} \draw [gray] (0,4*\i) -- (4,4*\i);
\foreach \i in {0,1} \draw [red, ultra thick] (4*\i,0) -- (4*\i,4);
\foreach \i in {2,3} \draw [red, ultra thick] (0,4*\i) -- (4,4*\i);
\foreach \i in {0,...,3} \filldraw[fill=white!40,draw=black!80] (0,4*\i) circle (3pt)    node[anchor=north] {\small };
\foreach \i in {0,...,3} \filldraw[fill=white!40,draw=black!80] (4,4*\i) circle (3pt)    node[anchor=north] {\small };

\filldraw[fill=white!40,draw=white!80] (10,5) circle (3pt)    node[anchor=south] {\huge $\Rightarrow$};
\end{tikzpicture} \hspace{0.5cm}
\begin{tikzpicture}[domain=0:4,scale=0.7]
   \draw (3,0) -- (5,0);
  \draw (3,2) -- (3,-2);
  \draw (5,-2) -- (5,2);
  \draw (3,2) -- (5,2);
  \draw (3,-2) -- (5,-2);
  \draw (4,-2) -- (4,0);
  \draw (3,1) -- (5,1);

\filldraw[fill=white!40,draw=white!80] (6.8,-0.35) circle (3pt)    node[anchor=south] {\huge $\Rightarrow$};
\end{tikzpicture} \hspace{0.5cm} \begin{tikzpicture}[domain=0:4, scale=0.7]
  \draw (6,-1) -- (6,1);
  \draw (2,-1) -- (2,1);
  \draw (2,1) -- (3,1);
  \draw (3,0) -- (5,0);
  \draw (5,1) -- (6,1);
  \draw (6,-1) -- (5,-1);
  \draw (3,-1) -- (2,-1);
  \draw (3,2) -- (3,-2);
  \draw (5,-2) -- (5,2);
  \draw (3,2) -- (5,2);
  \draw (3,-2) -- (5,-2);
  \draw (4,-2) -- (4,0);
  \draw (3,1) -- (5,1);
\end{tikzpicture}
\end{center}
\vspace{0.25cm}
\begin{center}
\begin{tikzpicture}[y=.3cm, x=.3cm,font=\normalsize,scale=0.6]
\foreach \i in {0,1} \draw [gray] (4*\i,0) -- (4*\i,4);
\foreach \i in {0,1} \draw [gray] (4*\i,4) -- (4*\i,8);
\foreach \i in {0,1} \draw [gray] (4*\i,8) -- (4*\i,12);
\foreach \i in {0,...,3} \draw [gray] (0,4*\i) -- (4,4*\i);
\foreach \i in {0,1} \draw [red, ultra thick] (4*\i,8) -- (4*\i,12);
\foreach \i in {0,1} \draw [red, ultra thick] (0,4*\i) -- (4,4*\i);
\foreach \i in {0,...,3} \filldraw[fill=white!40,draw=black!80] (0,4*\i) circle (3pt)    node[anchor=north] {\small };
\foreach \i in {0,...,3} \filldraw[fill=white!40,draw=black!80] (4,4*\i) circle (3pt)    node[anchor=north] {\small };

\filldraw[fill=white!40,draw=white!80] (10,5) circle (3pt)    node[anchor=south] {\huge $\Rightarrow$};
\end{tikzpicture} \hspace{0.5cm}
\begin{tikzpicture}[domain=0:4,scale=0.7]
  \draw (3,0) -- (5,0);
  \draw (3,2) -- (3,-2);
  \draw (5,-2) -- (5,2);
  \draw (3,2) -- (5,2);
  \draw (3,-2) -- (5,-2);
  \draw (4,0) -- (4,2);
  \draw (3,-1) -- (5,-1);

\filldraw[fill=white!40,draw=white!80] (6.8,-0.35) circle (3pt)    node[anchor=south] {\huge $\Rightarrow$};
\end{tikzpicture} \hspace{0.5cm} \begin{tikzpicture}[domain=0:4, scale=0.7]
  \draw (6,-1) -- (6,1);
  \draw (2,-1) -- (2,1);
  \draw (2,1) -- (3,1);
  \draw (3,0) -- (5,0);
  \draw (5,1) -- (6,1);
  \draw (6,-1) -- (5,-1);
  \draw (3,-1) -- (2,-1);
  \draw (3,2) -- (3,-2);
  \draw (5,-2) -- (5,2);
  \draw (3,2) -- (5,2);
  \draw (3,-2) -- (5,-2);
  \draw (4,0) -- (4,2);
  \draw (3,-1) -- (5,-1);
\end{tikzpicture}
\end{center}
\vspace{0.25cm}
\begin{center}
\begin{tikzpicture}[y=.3cm, x=.3cm,font=\normalsize,scale=0.6]
\foreach \i in {0,1} \draw [gray] (4*\i,0) -- (4*\i,4);
\foreach \i in {0,1} \draw [gray] (4*\i,4) -- (4*\i,8);
\foreach \i in {0,1} \draw [gray] (4*\i,8) -- (4*\i,12);
\foreach \i in {0,...,3} \draw [gray] (0,4*\i) -- (4,4*\i);
\foreach \i in {0,1} \draw [red, ultra thick] (4*\i,0) -- (4*\i,4);
\foreach \i in {0,1} \draw [red, ultra thick] (4*\i,8) -- (4*\i,12);
\foreach \i in {0,...,3} \filldraw[fill=white!40,draw=black!80] (0,4*\i) circle (3pt)    node[anchor=north] {\small };
\foreach \i in {0,...,3} \filldraw[fill=white!40,draw=black!80] (4,4*\i) circle (3pt)    node[anchor=north] {\small };

\filldraw[fill=white!40,draw=white!80] (10,5) circle (3pt)    node[anchor=south] {\huge $\Rightarrow$};
\end{tikzpicture} \hspace{0.5cm}
\begin{tikzpicture}[domain=0:4,scale=0.7]
  \draw (3,0) -- (5,0);
  \draw (3,2) -- (3,-2);
  \draw (5,-2) -- (5,2);
  \draw (3,2) -- (5,2);
  \draw (3,-2) -- (5,-2);
  \draw (4,-2) -- (4,2);
  \filldraw[fill=white!40,draw=white!80] (6.8,-0.35) circle (3pt)    node[anchor=south] {\huge $\Rightarrow$};
\end{tikzpicture} \hspace{0.5cm} \begin{tikzpicture}[domain=0:4, scale=0.7]
  \draw (6,-1) -- (6,1);
  \draw (2,-1) -- (2,1);
  \draw (2,1) -- (3,1);
  \draw (3,0) -- (5,0);
  \draw (5,1) -- (6,1);
  \draw (6,-1) -- (5,-1);
  \draw (3,-1) -- (2,-1);
  \draw (3,2) -- (3,-2);
  \draw (5,-2) -- (5,2);
  \draw (3,2) -- (5,2);
  \draw (3,-2) -- (5,-2);
  \draw (4,-2) -- (4,2);
\end{tikzpicture}
\end{center}
\vspace{0.25cm}
\begin{center}
\begin{tikzpicture}[y=.3cm, x=.3cm,font=\normalsize,scale=0.6]
\foreach \i in {0,1} \draw [gray] (4*\i,0) -- (4*\i,4);
\foreach \i in {0,1} \draw [gray] (4*\i,4) -- (4*\i,8);
\foreach \i in {0,1} \draw [gray] (4*\i,8) -- (4*\i,12);
\foreach \i in {0,...,3} \draw [gray] (0,4*\i) -- (4,4*\i);
\foreach \i in {0,1} \draw [red, ultra thick] (4*\i,4) -- (4*\i,8);
\foreach \i in {0,3} \draw [red, ultra thick] (0,4*\i) -- (4,4*\i);
\foreach \i in {0,...,3} \filldraw[fill=white!40,draw=black!80] (0,4*\i) circle (3pt)    node[anchor=north] {\small };
\foreach \i in {0,...,3} \filldraw[fill=white!40,draw=black!80] (4,4*\i) circle (3pt)    node[anchor=north] {\small };

\filldraw[fill=white!40,draw=white!80] (10,5) circle (3pt)    node[anchor=south] {\huge $\Rightarrow$};
\end{tikzpicture} \hspace{0.5cm}
\begin{tikzpicture}[domain=0:4,scale=0.7]
  \draw (3,1) -- (5,1);
  \draw (5,-1) -- (3,-1);
  \draw (3,2) -- (3,-2);
  \draw (5,-2) -- (5,2);
  \draw (3,2) -- (5,2);
  \draw (3,-2) -- (5,-2);
  \draw (4,-1) -- (4,1);

  \filldraw[fill=white!40,draw=white!80] (6.8,-0.35) circle (3pt)    node[anchor=south] {\huge $\Rightarrow$};
\end{tikzpicture} \hspace{0.5cm}
\begin{tikzpicture}[domain=0:4, scale=0.7]
  \draw (6,-1) -- (6,1);
  \draw (2,-1) -- (2,1);
  \draw (2,1) -- (6,1);
  \draw (6,-1) -- (2,-1);
  \draw (3,2) -- (3,-2);
  \draw (5,-2) -- (5,2);
  \draw (3,2) -- (5,2);
  \draw (3,-2) -- (5,-2);
  \draw (4,-1) -- (4,1);
\end{tikzpicture}
\end{center}

Therefore, we have the following routes for each perfect matching.
%::::::::::::::::::::::::::::::::::::::::::::::::::::::::::::::
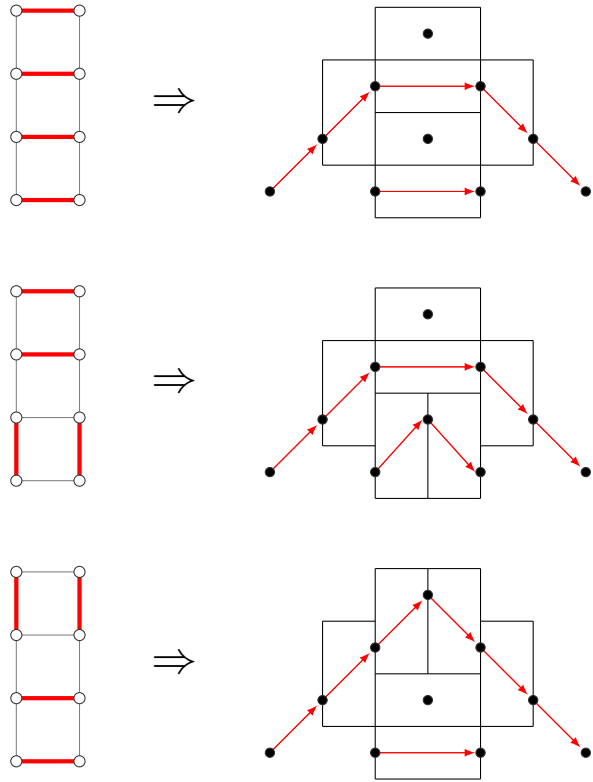
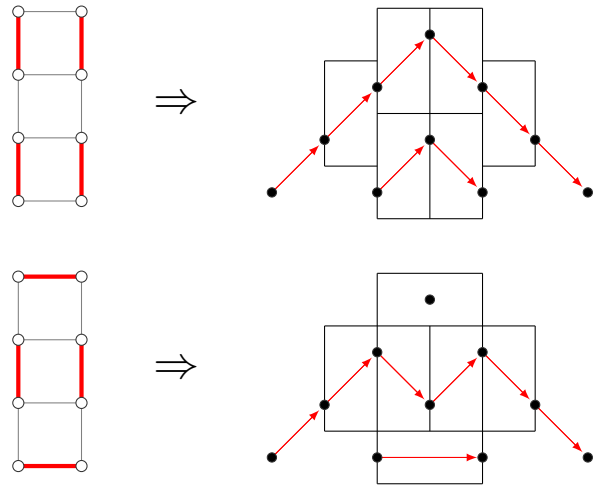
\begin{center}
\begin{tikzpicture}[y=.3cm, x=.3cm,font=\normalsize, scale=0.7]
\foreach \i in {0,1} \draw [gray] (4*\i,0) -- (4*\i,4);
\foreach \i in {0,1} \draw [gray] (4*\i,4) -- (4*\i,8);
\foreach \i in {0,1} \draw [gray] (4*\i,8) -- (4*\i,12);
\foreach \i in {0,...,3} \draw [gray] (0,4*\i) -- (4,4*\i);
\foreach \i in {0,...,3} \draw [red, ultra thick] (0,4*\i) -- (4,4*\i);
\foreach \i in {0,...,3} \filldraw[fill=white!40,draw=black!80] (0,4*\i) circle (3pt)    node[anchor=north] {\small };
\foreach \i in {0,...,3} \filldraw[fill=white!40,draw=black!80] (4,4*\i) circle (3pt)    node[anchor=north] {\small };
\filldraw[fill=white!40,draw=white!80] (10,5) circle (3pt)    node[anchor=south] {\huge $\Rightarrow$};
\end{tikzpicture} \hspace{0.5cm}
\begin{tikzpicture}[domain=0:4, scale=0.7]
  \draw (6,-1) -- (6,1);
  \draw (2,-1) -- (2,1);
  \draw (2,1) -- (3,1);
  \draw (3,0) -- (5,0);
  \draw (5,1) -- (6,1);
  \draw (6,-1) -- (5,-1);
  \draw (3,-1) -- (2,-1);
  \draw (3,2) -- (3,-2);
  \draw (5,-2) -- (5,2);
  \draw (3,2) -- (5,2);
  \draw (3,-2) -- (5,-2);
  \draw (3,-1) -- (5,-1);
  \draw (3,1) -- (5,1);

  \draw[fill,->, >=latex][red] (3,0.5) -- (4.9,0.5); 
  \draw[fill,->, >=latex][red] (3,-1.5) -- (4.9,-1.5); 
  
  \draw[fill,->, >=latex][red] (2,-0.5) -- (2.9,0.4); 
  \draw[fill,->, >=latex][red] (5,0.5) -- (5.9,-0.4);

  \draw[fill,->, >=latex][red] (1,-1.5) -- (1.9,-0.6); 
  \draw[fill,->, >=latex][red] (6,-0.5) -- (6.9,-1.4);

  \filldraw[fill=black!100,draw=black!80] (2,-0.5) circle (2.5pt)    node[anchor=north] {\small };
  \filldraw[fill=black!100,draw=black!80] (1,-1.5) circle (2.5pt)    node[anchor=north] {\small };
  \filldraw[fill=black!100,draw=black!80] (6,-0.5) circle (2.5pt)    node[anchor=north] {\small };
  \filldraw[fill=black!100,draw=black!80] (7,-1.5) circle (2.5pt)    node[anchor=north] {\small };
  
  \filldraw[fill=black!100,draw=black!80] (4,1.5) circle (2.5pt)    node[anchor=north] {\small };
  \filldraw[fill=black!100,draw=black!80] (3,0.5) circle (2.5pt)    node[anchor=north] {\small };
  \filldraw[fill=black!100,draw=black!80] (5,0.5) circle (2.5pt)    node[anchor=north] {\small };
  \filldraw[fill=black!100,draw=black!80] (4,-0.5) circle (2.5pt)    node[anchor=north] {\small };
  \filldraw[fill=black!100,draw=black!80] (3,-1.5) circle (2.5pt)    node[anchor=north] {\small };
  \filldraw[fill=black!100,draw=black!80] (5,-1.5) circle (2.5pt)    node[anchor=north] {\small };
\end{tikzpicture}
\end{center}
\vspace{0.5cm}
\begin{center}
\begin{tikzpicture}[y=.3cm, x=.3cm,font=\normalsize, scale=0.7]
\foreach \i in {0,1} \draw [gray] (4*\i,0) -- (4*\i,4);
\foreach \i in {0,1} \draw [gray] (4*\i,4) -- (4*\i,8);
\foreach \i in {0,1} \draw [gray] (4*\i,8) -- (4*\i,12);
\foreach \i in {0,...,3} \draw [gray] (0,4*\i) -- (4,4*\i);
\foreach \i in {0,1} \draw [red, ultra thick] (4*\i,0) -- (4*\i,4);
\foreach \i in {2,3} \draw [red, ultra thick] (0,4*\i) -- (4,4*\i);
\foreach \i in {0,...,3} \filldraw[fill=white!40,draw=black!80] (0,4*\i) circle (3pt)    node[anchor=north] {\small };
\foreach \i in {0,...,3} \filldraw[fill=white!40,draw=black!80] (4,4*\i) circle (3pt)    node[anchor=north] {\small };

\filldraw[fill=white!40,draw=white!80] (10,5) circle (3pt)    node[anchor=south] {\huge $\Rightarrow$};
\end{tikzpicture} \hspace{0.5cm}
\begin{tikzpicture}[domain=0:4, scale=0.7]
  \draw (6,-1) -- (6,1);
  \draw (2,-1) -- (2,1);
  \draw (2,1) -- (3,1);
  \draw (3,0) -- (5,0);
  \draw (5,1) -- (6,1);
  \draw (6,-1) -- (5,-1);
  \draw (3,-1) -- (2,-1);
  \draw (3,2) -- (3,-2);
  \draw (5,-2) -- (5,2);
  \draw (3,2) -- (5,2);
  \draw (3,-2) -- (5,-2);
  \draw (4,-2) -- (4,0);
  \draw (3,1) -- (5,1);

  \draw[fill,->, >=latex][red] (3,0.5) -- (4.9,0.5); 
  \draw[fill,->, >=latex][red] (4,-0.5) -- (4.9,-1.5); 
  \draw[fill,->, >=latex][red] (3,-1.5) -- (3.9,-0.5);  
  \draw[fill,->, >=latex][red] (2,-0.5) -- (2.9,0.4); 
  \draw[fill,->, >=latex][red] (5,0.5) -- (5.9,-0.4);

  \draw[fill,->, >=latex][red] (1,-1.5) -- (1.9,-0.6); 
  \draw[fill,->, >=latex][red] (6,-0.5) -- (6.9,-1.4);

  \filldraw[fill=black!100,draw=black!80] (2,-0.5) circle (2.5pt)    node[anchor=north] {\small };
  \filldraw[fill=black!100,draw=black!80] (1,-1.5) circle (2.5pt)    node[anchor=north] {\small };
  \filldraw[fill=black!100,draw=black!80] (6,-0.5) circle (2.5pt)    node[anchor=north] {\small };
  \filldraw[fill=black!100,draw=black!80] (7,-1.5) circle (2.5pt)    node[anchor=north] {\small };
  
%Black points
  \filldraw[fill=black!100,draw=black!80] (4,1.5) circle (2.5pt)    node[anchor=north] {\small };
  \filldraw[fill=black!100,draw=black!80] (3,0.5) circle (2.5pt)    node[anchor=north] {\small };
  \filldraw[fill=black!100,draw=black!80] (5,0.5) circle (2.5pt)    node[anchor=north] {\small };
  \filldraw[fill=black!100,draw=black!80] (4,-0.5) circle (2.5pt)    node[anchor=north] {\small };
  \filldraw[fill=black!100,draw=black!80] (3,-1.5) circle (2.5pt)    node[anchor=north] {\small };
  \filldraw[fill=black!100,draw=black!80] (5,-1.5) circle (2.5pt)    node[anchor=north] {\small };
\end{tikzpicture}
\end{center}
\vspace{0.5cm}
\begin{center}
\begin{tikzpicture}[y=.3cm, x=.3cm,font=\normalsize, scale=0.7]
\foreach \i in {0,1} \draw [gray] (4*\i,0) -- (4*\i,4);
\foreach \i in {0,1} \draw [gray] (4*\i,4) -- (4*\i,8);
\foreach \i in {0,1} \draw [gray] (4*\i,8) -- (4*\i,12);
\foreach \i in {0,...,3} \draw [gray] (0,4*\i) -- (4,4*\i);
\foreach \i in {0,1} \draw [red, ultra thick] (4*\i,8) -- (4*\i,12);
\foreach \i in {0,1} \draw [red, ultra thick] (0,4*\i) -- (4,4*\i);
\foreach \i in {0,...,3} \filldraw[fill=white!40,draw=black!80] (0,4*\i) circle (3pt)    node[anchor=north] {\small };
\foreach \i in {0,...,3} \filldraw[fill=white!40,draw=black!80] (4,4*\i) circle (3pt)    node[anchor=north] {\small };

\filldraw[fill=white!40,draw=white!80] (10,5) circle (3pt)    node[anchor=south] {\huge $\Rightarrow$};
\end{tikzpicture} \hspace{0.5cm}
\begin{tikzpicture}[domain=0:4, scale=0.7]
  \draw (6,-1) -- (6,1);
  \draw (2,-1) -- (2,1);
  \draw (2,1) -- (3,1);
  \draw (3,0) -- (5,0);
  \draw (5,1) -- (6,1);
  \draw (6,-1) -- (5,-1);
  \draw (3,-1) -- (2,-1);
  \draw (3,2) -- (3,-2);
  \draw (5,-2) -- (5,2);
  \draw (3,2) -- (5,2);
  \draw (3,-2) -- (5,-2);
  \draw (4,0) -- (4,2);
  \draw (3,-1) -- (5,-1);
  
  \draw[fill,->, >=latex][red] (4,1.5) -- (4.9,0.6); 
  \draw[fill,->, >=latex][red] (3,0.5) -- (3.9,1.4);
  \draw[fill,->, >=latex][red] (3,-1.5) -- (4.9,-1.5);  
  \draw[fill,->, >=latex][red] (2,-0.5) -- (2.9,0.4); 
  \draw[fill,->, >=latex][red] (5,0.5) -- (5.9,-0.4);

  \draw[fill,->, >=latex][red] (1,-1.5) -- (1.9,-0.6); 
  \draw[fill,->, >=latex][red] (6,-0.5) -- (6.9,-1.4);

  \filldraw[fill=black!100,draw=black!80] (2,-0.5) circle (2.5pt)    node[anchor=north] {\small };
  \filldraw[fill=black!100,draw=black!80] (1,-1.5) circle (2.5pt)    node[anchor=north] {\small };
  \filldraw[fill=black!100,draw=black!80] (6,-0.5) circle (2.5pt)    node[anchor=north] {\small };
  \filldraw[fill=black!100,draw=black!80] (7,-1.5) circle (2.5pt)    node[anchor=north] {\small };
    
%Black points
  \filldraw[fill=black!100,draw=black!80] (4,1.5) circle (2.5pt)    node[anchor=north] {\small };
  \filldraw[fill=black!100,draw=black!80] (3,0.5) circle (2.5pt)    node[anchor=north] {\small };
  \filldraw[fill=black!100,draw=black!80] (5,0.5) circle (2.5pt)    node[anchor=north] {\small };
  \filldraw[fill=black!100,draw=black!80] (4,-0.5) circle (2.5pt)    node[anchor=north] {\small };
  \filldraw[fill=black!100,draw=black!80] (3,-1.5) circle (2.5pt)    node[anchor=north] {\small };
  \filldraw[fill=black!100,draw=black!80] (5,-1.5) circle (2.5pt)    node[anchor=north] {\small };
\end{tikzpicture}
\end{center}
\vspace{0.5cm}
\begin{center}
\begin{tikzpicture}[y=.3cm, x=.3cm,font=\normalsize, scale=0.7]
\foreach \i in {0,1} \draw [gray] (4*\i,0) -- (4*\i,4);
\foreach \i in {0,1} \draw [gray] (4*\i,4) -- (4*\i,8);
\foreach \i in {0,1} \draw [gray] (4*\i,8) -- (4*\i,12);
\foreach \i in {0,...,3} \draw [gray] (0,4*\i) -- (4,4*\i);
\foreach \i in {0,1} \draw [red, ultra thick] (4*\i,0) -- (4*\i,4);
\foreach \i in {0,1} \draw [red, ultra thick] (4*\i,8) -- (4*\i,12);
\foreach \i in {0,...,3} \filldraw[fill=white!40,draw=black!80] (0,4*\i) circle (3pt)    node[anchor=north] {\small };
\foreach \i in {0,...,3} \filldraw[fill=white!40,draw=black!80] (4,4*\i) circle (3pt)    node[anchor=north] {\small };

\filldraw[fill=white!40,draw=white!80] (10,5) circle (3pt)    node[anchor=south] {\huge $\Rightarrow$};
\end{tikzpicture} \hspace{0.5cm}
\begin{tikzpicture}[domain=0:4, scale=0.7]
  \draw (6,-1) -- (6,1);
  \draw (2,-1) -- (2,1);
  \draw (2,1) -- (3,1);
  \draw (3,0) -- (5,0);
  \draw (5,1) -- (6,1);
  \draw (6,-1) -- (5,-1);
  \draw (3,-1) -- (2,-1);
  \draw (3,2) -- (3,-2);
  \draw (5,-2) -- (5,2);
  \draw (3,2) -- (5,2);
  \draw (3,-2) -- (5,-2);
  \draw (4,-2) -- (4,2);

  \draw[fill,->, >=latex][red] (4,1.5) -- (4.9,0.6); 
  \draw[fill,->, >=latex][red] (3,0.5) -- (3.9,1.4);
  \draw[fill,->, >=latex][red] (4,-0.5) -- (4.9,-1.4); 
  \draw[fill,->, >=latex][red] (3,-1.5) -- (3.9,-0.6);  
  \draw[fill,->, >=latex][red] (2,-0.5) -- (2.9,0.4); 
  \draw[fill,->, >=latex][red] (5,0.5) -- (5.9,-0.4);

  \draw[fill,->, >=latex][red] (1,-1.5) -- (1.9,-0.6); 
  \draw[fill,->, >=latex][red] (6,-0.5) -- (6.9,-1.4);

  \filldraw[fill=black!100,draw=black!80] (2,-0.5) circle (2.5pt)    node[anchor=north] {\small };
  \filldraw[fill=black!100,draw=black!80] (1,-1.5) circle (2.5pt)    node[anchor=north] {\small };
  \filldraw[fill=black!100,draw=black!80] (6,-0.5) circle (2.5pt)    node[anchor=north] {\small };
  \filldraw[fill=black!100,draw=black!80] (7,-1.5) circle (2.5pt)    node[anchor=north] {\small };

%Black points
  \filldraw[fill=black!100,draw=black!80] (4,1.5) circle (2.5pt)    node[anchor=north] {\small };
  \filldraw[fill=black!100,draw=black!80] (3,0.5) circle (2.5pt)    node[anchor=north] {\small };
  \filldraw[fill=black!100,draw=black!80] (5,0.5) circle (2.5pt)    node[anchor=north] {\small };
  \filldraw[fill=black!100,draw=black!80] (4,-0.5) circle (2.5pt)    node[anchor=north] {\small };
  \filldraw[fill=black!100,draw=black!80] (3,-1.5) circle (2.5pt)    node[anchor=north] {\small };
  \filldraw[fill=black!100,draw=black!80] (5,-1.5) circle (2.5pt)    node[anchor=north] {\small };
\end{tikzpicture}
\end{center}
\vspace{0.5cm}
\begin{center}
\begin{tikzpicture}[y=.3cm, x=.3cm,font=\normalsize, scale=0.7]
\foreach \i in {0,1} \draw [gray] (4*\i,0) -- (4*\i,4);
\foreach \i in {0,1} \draw [gray] (4*\i,4) -- (4*\i,8);
\foreach \i in {0,1} \draw [gray] (4*\i,8) -- (4*\i,12);
\foreach \i in {0,...,3} \draw [gray] (0,4*\i) -- (4,4*\i);
\foreach \i in {0,1} \draw [red, ultra thick] (4*\i,4) -- (4*\i,8);
\foreach \i in {0,3} \draw [red, ultra thick] (0,4*\i) -- (4,4*\i);
\foreach \i in {0,...,3} \filldraw[fill=white!40,draw=black!80] (0,4*\i) circle (3pt)    node[anchor=north] {\small };
\foreach \i in {0,...,3} \filldraw[fill=white!40,draw=black!80] (4,4*\i) circle (3pt)    node[anchor=north] {\small };

\filldraw[fill=white!40,draw=white!80] (10,5) circle (3pt)    node[anchor=south] {\huge $\Rightarrow$};
\end{tikzpicture} \hspace{0.5cm}
\begin{tikzpicture}[domain=0:4, scale=0.7]
  \draw (6,-1) -- (6,1);
  \draw (2,-1) -- (2,1);
  \draw (2,1) -- (6,1);
  \draw (6,-1) -- (2,-1);
  \draw (3,2) -- (3,-2);
  \draw (5,-2) -- (5,2);
  \draw (3,2) -- (5,2);
  \draw (3,-2) -- (5,-2);
  \draw (4,-1) -- (4,1);

  \draw[fill,->, >=latex][red] (4,-0.5) -- (4.9,0.4); 
  \draw[fill,->, >=latex][red] (3,0.5) -- (3.9,-0.4);
  \draw[fill,->, >=latex][red] (3,-1.5) -- (4.9,-1.5);  

  \draw[fill,->, >=latex][red] (2,-0.5) -- (2.9,0.4); 
  \draw[fill,->, >=latex][red] (5,0.5) -- (5.9,-0.4);

  \draw[fill,->, >=latex][red] (1,-1.5) -- (1.9,-0.6); 
  \draw[fill,->, >=latex][red] (6,-0.5) -- (6.9,-1.4);

  \filldraw[fill=black!100,draw=black!80] (2,-0.5) circle (2.5pt)    node[anchor=north] {\small };
  \filldraw[fill=black!100,draw=black!80] (1,-1.5) circle (2.5pt)    node[anchor=north] {\small };
  \filldraw[fill=black!100,draw=black!80] (6,-0.5) circle (2.5pt)    node[anchor=north] {\small };
  \filldraw[fill=black!100,draw=black!80] (7,-1.5) circle (2.5pt)    node[anchor=north] {\small };

%Black points
  \filldraw[fill=black!100,draw=black!80] (4,1.5) circle (2.5pt)    node[anchor=north] {\small };
  \filldraw[fill=black!100,draw=black!80] (3,0.5) circle (2.5pt)    node[anchor=north] {\small };
  \filldraw[fill=black!100,draw=black!80] (5,0.5) circle (2.5pt)    node[anchor=north] {\small };
  \filldraw[fill=black!100,draw=black!80] (4,-0.5) circle (2.5pt)    node[anchor=north] {\small };
  \filldraw[fill=black!100,draw=black!80] (3,-1.5) circle (2.5pt)    node[anchor=north] {\small };
  \filldraw[fill=black!100,draw=black!80] (5,-1.5) circle (2.5pt)    node[anchor=north] {\small };
\end{tikzpicture}
\end{center}

Therefore, the acyclic-directed graph in Figure \ref{fig:Acyclic-directed graph constructed by superimposing the routes} is constructed by superimposing every one of the previously obtained routes.
\begin{figure}[ht]
\begin{center}
\begin{tikzpicture}[domain=0:4, scale=0.7]
  
  \draw[->, >=latex,red] (4,-0.5) -- (4.9,0.4); 
  \draw[->, >=latex,red] (3,0.5) -- (3.9,-0.4);
  \draw[->, >=latex,red] (3,-1.5) -- (4.9,-1.5);  
  \draw[->, >=latex,red] (3,0.5) -- (4.9,0.5); 
 
  \draw[->, >=latex,red] (1,-1.5) -- (1.9,-0.6); 
  \draw[->, >=latex,red] (2,-0.5) -- (2.9,0.4); 
  
  \draw[->, >=latex,red] (3,-1.5) -- (3.9,-0.6); 
  \draw[->, >=latex,red] (4,-0.5) -- (4.9,-1.4); 
  
  \draw[->, >=latex,red] (5,0.5) -- (5.9,-0.4);
  \draw[->, >=latex,red] (6,-0.5) -- (6.9,-1.4);
  
  \draw[->, >=latex,red] (3,0.5) -- (3.9,1.4);
  \draw[->, >=latex,red] (4,1.5) -- (4.9,0.6);
  
\node (1) at (3,-1.9) {$s_1$};
\node (2) at (1,-1.9) {$s_2$};
\node (3) at (5,-1.9) {$t_1$};
\node (4) at (7,-1.9) {$t_2$};

  \filldraw[fill=black!100,draw=black!80] (2,-0.5) circle (2.5pt)    node[anchor=north] {\small };
  \filldraw[fill=black!100,draw=black!80] (1,-1.5) circle (2.5pt)    node[anchor=north] {\small };
  \filldraw[fill=black!100,draw=black!80] (6,-0.5) circle (2.5pt)    node[anchor=north] {\small };
  \filldraw[fill=black!100,draw=black!80] (7,-1.5) circle (2.5pt)    node[anchor=north] {\small };
  \filldraw[fill=black!100,draw=black!80] (4,1.5) circle (2.5pt)    node[anchor=north] {\small };
  \filldraw[fill=black!100,draw=black!80] (3,0.5) circle (2.5pt)    node[anchor=north] {\small };
  \filldraw[fill=black!100,draw=black!80] (5,0.5) circle (2.5pt)    node[anchor=north] {\small };
  \filldraw[fill=black!100,draw=black!80] (4,-0.5) circle (2.5pt)    node[anchor=north] {\small };
  \filldraw[fill=black!100,draw=black!80] (3,-1.5) circle (2.5pt)    node[anchor=north] {\small };
  \filldraw[fill=black!100,draw=black!80] (5,-1.5) circle (2.5pt)    node[anchor=north] {\small };
\end{tikzpicture}
\end{center}
    \caption{Acyclic-directed graph constructed by superimposing $2$-routes.}
\label{fig:Acyclic-directed graph constructed by superimposing the routes}
\end{figure}
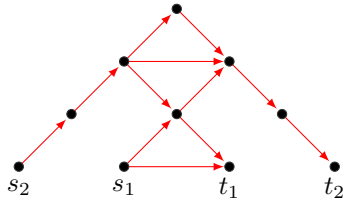
The acyclic-directed graph in Figure \ref{fig:Acyclic-directed graph constructed by superimposing the routes} allows us to add arrows as long as they do not affect the number of $2$-routes. For instance, we can add two arrows as shown in Figure \ref{fig:Catalan acyclic-directed graph} (one ending at vertex $s_1$ and another starting at vertex $t_1$). This is because $2$-routes are a collection of two vertex-disjoint paths. Since the path from $s_1$ to $t_1$ shares vertices with the new arrows, the path from $s_2$ to $t_2$ cannot intersect it.
\begin{figure}[ht]
\begin{center}
\begin{tikzpicture}[domain=0:4, scale=0.7]
  
  \draw[->, >=latex,red] (4,-0.5) -- (4.9,0.4); 
  \draw[->, >=latex,red] (3,0.5) -- (3.9,-0.4);
  \draw[->, >=latex,blue] (3,-1.5) -- (4.9,-1.5);  
  \draw[->, >=latex,blue] (3,0.5) -- (4.9,0.5); 
  \draw[->, >=latex,red] (5,-1.5) -- (5.9,-0.6);
  \draw[->, >=latex,red] (2,-0.5) -- (2.9,-1.4);
 
  \draw[->, >=latex,red] (1,-1.5) -- (1.9,-0.6); 
  \draw[->, >=latex,red] (2,-0.5) -- (2.9,0.4); 
  
  \draw[->, >=latex,red] (3,-1.5) -- (3.9,-0.6); 
  \draw[->, >=latex,red] (4,-0.5) -- (4.9,-1.4); 
  
  \draw[->, >=latex,red] (5,0.5) -- (5.9,-0.4);
  \draw[->, >=latex,red] (6,-0.5) -- (6.9,-1.4);
  
  \draw[->, >=latex,red] (3,0.5) -- (3.9,1.4);
  \draw[->, >=latex,red] (4,1.5) -- (4.9,0.6);
  
\node (1) at (3,-1.9) {$s_1$};
\node (2) at (1,-1.9) {$s_2$};
\node (3) at (5,-1.9) {$t_1$};
\node (4) at (7,-1.9) {$t_2$};

  \filldraw[fill=black!100,draw=black!80] (2,-0.5) circle (2.5pt)    node[anchor=north] {\small };
  \filldraw[fill=black!100,draw=black!80] (1,-1.5) circle (2.5pt)    node[anchor=north] {\small };
  \filldraw[fill=black!100,draw=black!80] (6,-0.5) circle (2.5pt)    node[anchor=north] {\small };
  \filldraw[fill=black!100,draw=black!80] (7,-1.5) circle (2.5pt)    node[anchor=north] {\small };
  \filldraw[fill=black!100,draw=black!80] (4,1.5) circle (2.5pt)    node[anchor=north] {\small };
  \filldraw[fill=black!100,draw=black!80] (3,0.5) circle (2.5pt)    node[anchor=north] {\small };
  \filldraw[fill=black!100,draw=black!80] (5,0.5) circle (2.5pt)    node[anchor=north] {\small };
  \filldraw[fill=black!100,draw=black!80] (4,-0.5) circle (2.5pt)    node[anchor=north] {\small };
  \filldraw[fill=black!100,draw=black!80] (3,-1.5) circle (2.5pt)    node[anchor=north] {\small };
  \filldraw[fill=black!100,draw=black!80] (5,-1.5) circle (2.5pt)    node[anchor=north] {\small };
\end{tikzpicture}
\end{center}
    \caption{Catalan acyclic-directed graph associated to $2$-routes.}
\label{fig:Catalan acyclic-directed graph}
\end{figure}
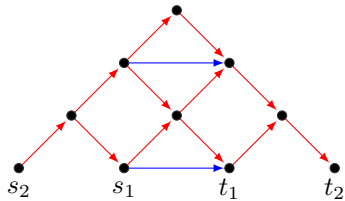

To find the entries of the path matrix, the color of the horizontal arrows was changed to blue in Figure \ref{fig:Catalan acyclic-directed graph}. If we want to see the paths from vertex $s_i$ to vertex $t_j$, then we can see that the number of paths that use a blue edge is $C_{i+j-2}$, and the number of paths that do not use those blue edges is $C_{i+j-1}$. Therefore, the path matrix is
\[ M_2 =
\left( \begin{array}{cc}
\textcolor{blue}{C_0}+ \textcolor{red}{C_1} & \textcolor{blue}{C_1}+\textcolor{red}{C_2} \\ 
\textcolor{blue}{C_1}+\textcolor{red}{C_2} & \textcolor{blue}{C_2}+\textcolor{red}{C_3} \end{array} \right)=
\left( \begin{array}{cc}
 2 & 3 \\ 
 3 & 7 \end{array} \right) \]
where it can easily be verified that $\det(M_2)=5$.
\end{example}

\begin{proposition} \label{prop:Ladder_Fibonacci}
 Let $M_{st}=(m_{ij})_{1\leq i,j \leq k}$ be the path matrix associated to the triangular snake graph of the vertical ladder graph $L_{2k-1}$, $k\geq 1$, and let $F_n$ be the $n$-th Fibonacci number. Then the following relationship holds
\[ 
\det M_{st} = F_{2k+1},
\]
where
   \begin{equation*}
     m_{ij} = \left\{
	       \begin{array}{ll}
		 F_3=2      & \mathrm{if\ } i=j=1; \\
		 F_4=3      & \mathrm{if\ } i=j\in\{2,\dots,k\} ; \\
		 F_2=1      & \mathrm{if\ } i=j+1 \quad \mathrm{or} \quad i=j-1 ; \\
		 F_0=0 & \mathrm{otherwise}.
	       \end{array}
	     \right.
   \end{equation*}
\end{proposition}

\begin{proof}
Consider the vertical ladder graph $L_{2k-1}$. We obtain the following triangular snake graph $\mathcal{T}_{L_{2k-1}}$
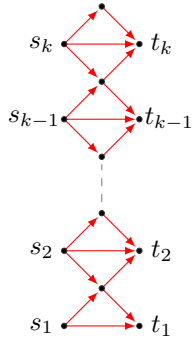
\begin{center}
\begin{tikzpicture}[domain=0:4, scale=0.6]

  \draw[->, >=latex,red] (4,7) -- (4.9,7.9); 
  \draw[->, >=latex,red] (3,8) -- (4.9,8); 
  \draw[->, >=latex,red] (3,8) -- (3.9,7.1); 
  \draw[->, >=latex,red] (3,8) -- (3.9,8.9);
  \draw[->, >=latex,red] (4,9) -- (4.9,8.1);
  
  \draw[->, >=latex,red] (4,5) -- (4.9,5.9); 
  \draw[->, >=latex,red] (3,6) -- (4.9,6); 
  \draw[->, >=latex,red] (3,6) -- (3.9,5.1); 
  \draw[->, >=latex,red] (3,6) -- (3.9,6.9);
  \draw[->, >=latex,red] (4,7) -- (4.9,6.1);
  
  \draw[gray,dashed] (4,3.8) -- (4,4.8);
  
  \draw[->, >=latex,red] (4,1.5) -- (4.9,2.4); 
  \draw[->, >=latex,red] (3,2.5) -- (3.9,1.6);
  \draw[->, >=latex,red] (3,0.5) -- (4.9,0.5);  
  \draw[->, >=latex,red] (3,2.5) -- (4.9,2.5); 
  
  \draw[->, >=latex,red] (3,0.5) -- (3.9,1.4); 
  \draw[->, >=latex,red] (4,1.5) -- (4.9,0.6); 
  
  \draw[->, >=latex,red] (3,2.5) -- (3.9,3.4);
  \draw[->, >=latex,red] (4,3.5) -- (4.9,2.6);
  
\node (1) at (2.4,0.5) {$s_1$};
\node (2) at (2.4,2.5) {$s_2$};
\node (3) at (5.6,0.5) {$t_1$};
\node (4) at (5.6,2.5) {$t_2$};
\node (5) at (2.2,6) {$s_{k-1}$};
\node (6) at (2.4,8) {$s_k$};
\node (7) at (5.8,6) {$t_{k-1}$};
\node (8) at (5.6,8) {$t_k$};

  \filldraw[fill=black!100,draw=black!80] (4,3.5) circle (2pt)    node[anchor=north] {\small };
  \filldraw[fill=black!100,draw=black!80] (3,2.5) circle (2pt)    node[anchor=north] {\small };
  \filldraw[fill=black!100,draw=black!80] (5,2.5) circle (2pt)    node[anchor=north] {\small };
  \filldraw[fill=black!100,draw=black!80] (4,1.5) circle (2pt)    node[anchor=north] {\small };
  \filldraw[fill=black!100,draw=black!80] (3,0.5) circle (2pt)    node[anchor=north] {\small };
  \filldraw[fill=black!100,draw=black!80] (5,0.5) circle (2pt)    node[anchor=north] {\small };

  \filldraw[fill=black!100,draw=black!80] (4,9) circle (2pt)    node[anchor=north] {\small };
  \filldraw[fill=black!100,draw=black!80] (3,8) circle (2pt)    node[anchor=north] {\small };
  \filldraw[fill=black!100,draw=black!80] (5,8) circle (2pt)    node[anchor=north] {\small };
  \filldraw[fill=black!100,draw=black!80] (4,7) circle (2pt)    node[anchor=north] {\small };
  \filldraw[fill=black!100,draw=black!80] (3,6) circle (2pt)    node[anchor=north] {\small };
  \filldraw[fill=black!100,draw=black!80] (5,6) circle (2pt)    node[anchor=north] {\small };
  \filldraw[fill=black!100,draw=black!80] (4,5) circle (2pt)    node[anchor=north] {\small };
\end{tikzpicture}
\end{center}

From a direct inspection we can see that the number $m_{ii}$ of paths $p_i$ from $s_i$ to $t_i$ is equal to 3 if $i\in \{2,3,\dots, k\}$ and equal to 2 if $i=1$. In addition, the number $m_{ij}$ of paths from $s_i$ to $t_j$ is equal to 1 if $|i-j|=1$ and equal to 0 if $|i-j|>1$. Then, using Lemma~\ref{lem:LGV} we conclude the proposition. 
\end{proof}

\begin{example}
Consider the ladder graph $L_3$ of the Example~\ref{Ex:L4_Aztec_Diamond}. The triangular snake graph can be represented as follows:
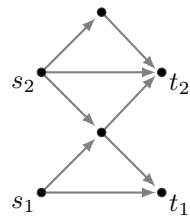
\begin{center}
\begin{tikzpicture}[domain=0:4, scale=0.8]
  
  \draw[->, >=latex][gray, thick] (4,-0.5) -- (4.9,0.45); 
  \draw[->, >=latex][gray, thick] (3,0.5) -- (3.9,-0.4);
  \draw[->, >=latex][gray, thick] (3,-1.5) -- (4.9,-1.5);  
  \draw[->, >=latex][gray, thick] (3,0.5) -- (4.9,0.5); 
  
  \draw[->, >=latex][gray, thick] (3,-1.5) -- (3.9,-0.6); 
  \draw[->, >=latex][gray, thick] (4,-0.5) -- (4.9,-1.4); 
  
  \draw[->, >=latex][gray, thick] (3,0.5) -- (3.9,1.4);
  \draw[->, >=latex][gray, thick] (4,1.5) -- (4.9,0.55);
  
\node (1) at (2.7,-1.7) {$s_1$};
\node (2) at (2.7,0.3) {$s_2$};
\node (3) at (5.3,-1.7) {$t_1$};
\node (4) at (5.3,0.3) {$t_2$};

  \filldraw[fill=black!100,draw=black!80] (4,1.5) circle (2pt)    node[anchor=north] {\small };
  \filldraw[fill=black!100,draw=black!80] (3,0.5) circle (2pt)    node[anchor=north] {\small };
  \filldraw[fill=black!100,draw=black!80] (5,0.5) circle (2pt)    node[anchor=north] {\small };
  \filldraw[fill=black!100,draw=black!80] (4,-0.5) circle (2pt)    node[anchor=north] {\small };
  \filldraw[fill=black!100,draw=black!80] (3,-1.5) circle (2pt)    node[anchor=north] {\small };
  \filldraw[fill=black!100,draw=black!80] (5,-1.5) circle (2pt)    node[anchor=north] {\small };
\end{tikzpicture}
\end{center}

We obtain the following perfect matchings and their corresponding $2$-routes
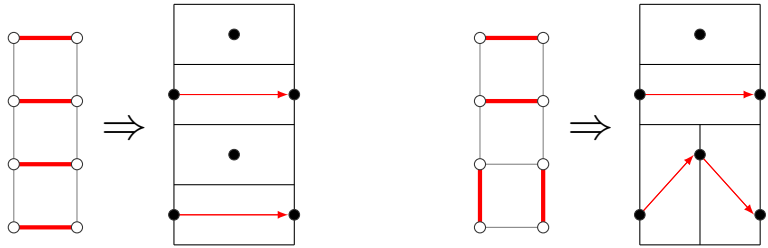
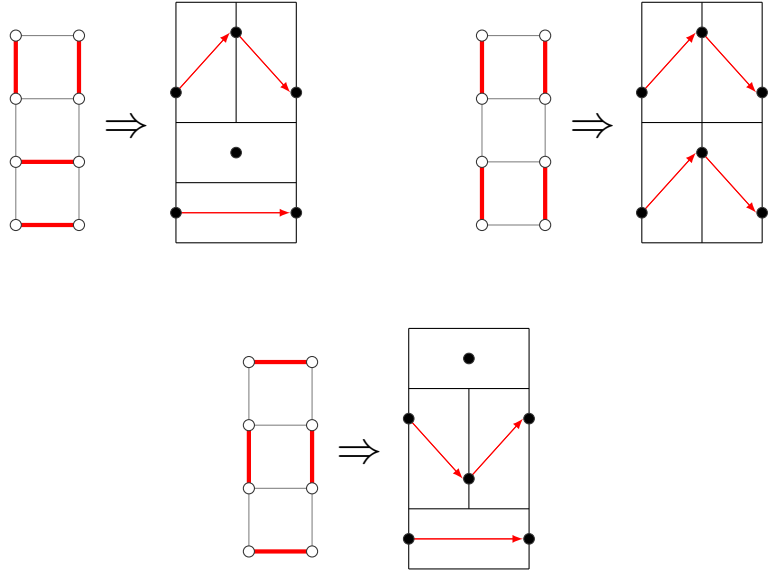
\begin{center}
\begin{tikzpicture}[y=.3cm, x=.3cm,font=\normalsize,scale=0.7]
\foreach \i in {0,1} \draw [gray] (4*\i,0) -- (4*\i,4);
\foreach \i in {0,1} \draw [gray] (4*\i,4) -- (4*\i,8);
\foreach \i in {0,1} \draw [gray] (4*\i,8) -- (4*\i,12);
\foreach \i in {0,...,3} \draw [gray] (0,4*\i) -- (4,4*\i);
\foreach \i in {0,...,3} \draw [red, ultra thick] (0,4*\i) -- (4,4*\i);
\foreach \i in {0,...,3} \filldraw[fill=white!40,draw=black!80] (0,4*\i) circle (3pt)    node[anchor=north] {\small };
\foreach \i in {0,...,3} \filldraw[fill=white!40,draw=black!80] (4,4*\i) circle (3pt)    node[anchor=north] {\small };
\filldraw[fill=white!40,draw=white!80] (7,5) circle (3pt)    node[anchor=south] {\huge $\Rightarrow$};
\end{tikzpicture}
\begin{tikzpicture}[domain=0:4,scale=0.8]

  \draw (3,2) -- (5,2);
  \draw (3,4) -- (3,0);
  \draw (5,0) -- (5,4);
  \draw (3,4) -- (5,4);
  \draw (3,0) -- (5,0);
  \draw (3,1) -- (5,1);
  \draw (3,3) -- (5,3);

  \draw[fill,->, >=latex][red] (3,2.5) -- (4.9,2.5); 
  \draw[fill,->, >=latex][red] (3,0.5) -- (4.9,0.5); 
  
  \filldraw[fill=black!100,draw=black!80] (3,2.5) circle (2.5pt)    node[anchor=north] {\small };
  \filldraw[fill=black!100,draw=black!80] (5,2.5) circle (2.5pt)    node[anchor=north] {\small };
  \filldraw[fill=black!100,draw=black!80] (3,0.5) circle (2.5pt)    node[anchor=north] {\small };
  \filldraw[fill=black!100,draw=black!80] (5,0.5) circle (2.5pt)    node[anchor=north] {\small };
  \filldraw[fill=black!100,draw=black!80] (4,1.5) circle (2.5pt)    node[anchor=north] {\small };
  \filldraw[fill=black!100,draw=black!80] (4,3.5) circle (2.5pt)    node[anchor=north] {\small };  
\end{tikzpicture}
\hspace{2cm}
\begin{tikzpicture}[y=.3cm, x=.3cm,font=\normalsize,scale=0.7]
\foreach \i in {0,1} \draw [gray] (4*\i,0) -- (4*\i,4);
\foreach \i in {0,1} \draw [gray] (4*\i,4) -- (4*\i,8);
\foreach \i in {0,1} \draw [gray] (4*\i,8) -- (4*\i,12);
\foreach \i in {0,...,3} \draw [gray] (0,4*\i) -- (4,4*\i);
\foreach \i in {0,1} \draw [red, ultra thick] (4*\i,0) -- (4*\i,4);
\foreach \i in {2,3} \draw [red, ultra thick] (0,4*\i) -- (4,4*\i);
\foreach \i in {0,...,3} \filldraw[fill=white!40,draw=black!80] (0,4*\i) circle (3pt)    node[anchor=north] {\small };
\foreach \i in {0,...,3} \filldraw[fill=white!40,draw=black!80] (4,4*\i) circle (3pt)    node[anchor=north] {\small };

\filldraw[fill=white!40,draw=white!80] (7,5) circle (3pt)    node[anchor=south] {\huge $\Rightarrow$};
\end{tikzpicture}
\begin{tikzpicture}[domain=0:4,scale=0.8]
  \draw (3,2) -- (5,2);
  \draw (3,4) -- (3,0);
  \draw (5,0) -- (5,4);
  \draw (3,4) -- (5,4);
  \draw (3,0) -- (5,0);
  \draw (4,0) -- (4,2);
  \draw (3,3) -- (5,3);

  \draw[fill,->, >=latex][red] (3,2.5) -- (4.9,2.5); 
  \draw[fill,->, >=latex][red] (4,1.5) -- (4.9,0.5); 
  \draw[fill,->, >=latex][red] (3,0.5) -- (3.9,1.5);

  \filldraw[fill=black!100,draw=black!80] (4,3.5) circle (2.5pt)    node[anchor=north] {\small };
  \filldraw[fill=black!100,draw=black!80] (3,2.5) circle (2.5pt)    node[anchor=north] {\small };
  \filldraw[fill=black!100,draw=black!80] (5,2.5) circle (2.5pt)    node[anchor=north] {\small };
  \filldraw[fill=black!100,draw=black!80] (4,1.5) circle (2.5pt)    node[anchor=north] {\small };
  \filldraw[fill=black!100,draw=black!80] (3,0.5) circle (2.5pt)    node[anchor=north] {\small };
  \filldraw[fill=black!100,draw=black!80] (5,0.5) circle (2.5pt)    node[anchor=north] {\small };
\end{tikzpicture}
\end{center}
\vspace{0.5cm}
\begin{center}
\begin{tikzpicture}[y=.3cm, x=.3cm,font=\normalsize,scale=0.7]
\foreach \i in {0,1} \draw [gray] (4*\i,0) -- (4*\i,4);
\foreach \i in {0,1} \draw [gray] (4*\i,4) -- (4*\i,8);
\foreach \i in {0,1} \draw [gray] (4*\i,8) -- (4*\i,12);
\foreach \i in {0,...,3} \draw [gray] (0,4*\i) -- (4,4*\i);
\foreach \i in {0,1} \draw [red, ultra thick] (4*\i,8) -- (4*\i,12);
\foreach \i in {0,1} \draw [red, ultra thick] (0,4*\i) -- (4,4*\i);
\foreach \i in {0,...,3} \filldraw[fill=white!40,draw=black!80] (0,4*\i) circle (3pt)    node[anchor=north] {\small };
\foreach \i in {0,...,3} \filldraw[fill=white!40,draw=black!80] (4,4*\i) circle (3pt)    node[anchor=north] {\small };

\filldraw[fill=white!40,draw=white!80] (7,5) circle (3pt)    node[anchor=south] {\huge $\Rightarrow$};
\end{tikzpicture}
\begin{tikzpicture}[domain=0:4,scale=0.8]
  \draw (3,0) -- (5,0);
  \draw (3,2) -- (3,-2);
  \draw (5,-2) -- (5,2);
  \draw (3,2) -- (5,2);
  \draw (3,-2) -- (5,-2);
  \draw (4,0) -- (4,2);
  \draw (3,-1) -- (5,-1);
  
  \draw[fill,->, >=latex][red] (4,1.5) -- (4.9,0.5); 
  \draw[fill,->, >=latex][red] (3,0.5) -- (3.9,1.5);
  \draw[fill,->, >=latex][red] (3,-1.5) -- (4.9,-1.5); 

  \filldraw[fill=black!100,draw=black!80] (4,1.5) circle (2.5pt)    node[anchor=north] {\small };
  \filldraw[fill=black!100,draw=black!80] (3,0.5) circle (2.5pt)    node[anchor=north] {\small };
  \filldraw[fill=black!100,draw=black!80] (5,0.5) circle (2.5pt)    node[anchor=north] {\small };
  \filldraw[fill=black!100,draw=black!80] (4,-0.5) circle (2.5pt)    node[anchor=north] {\small };
  \filldraw[fill=black!100,draw=black!80] (3,-1.5) circle (2.5pt)    node[anchor=north] {\small };
  \filldraw[fill=black!100,draw=black!80] (5,-1.5) circle (2.5pt)    node[anchor=north] {\small };
\end{tikzpicture}
\hspace{2cm}
\begin{tikzpicture}[y=.3cm, x=.3cm,font=\normalsize,scale=0.7]
\foreach \i in {0,1} \draw [gray] (4*\i,0) -- (4*\i,4);
\foreach \i in {0,1} \draw [gray] (4*\i,4) -- (4*\i,8);
\foreach \i in {0,1} \draw [gray] (4*\i,8) -- (4*\i,12);
\foreach \i in {0,...,3} \draw [gray] (0,4*\i) -- (4,4*\i);
\foreach \i in {0,1} \draw [red, ultra thick] (4*\i,0) -- (4*\i,4);
\foreach \i in {0,1} \draw [red, ultra thick] (4*\i,8) -- (4*\i,12);
\foreach \i in {0,...,3} \filldraw[fill=white!40,draw=black!80] (0,4*\i) circle (3pt)    node[anchor=north] {\small };
\foreach \i in {0,...,3} \filldraw[fill=white!40,draw=black!80] (4,4*\i) circle (3pt)    node[anchor=north] {\small };

\filldraw[fill=white!40,draw=white!80] (7,5) circle (3pt)    node[anchor=south] {\huge $\Rightarrow$};
\end{tikzpicture}
\begin{tikzpicture}[domain=0:4,scale=0.8]
  \draw (3,0) -- (5,0);
  \draw (3,2) -- (3,-2);
  \draw (5,-2) -- (5,2);
  \draw (3,2) -- (5,2);
  \draw (3,-2) -- (5,-2);
  \draw (4,-2) -- (4,2);

  \draw[fill,->, >=latex][red] (4,1.5) -- (4.9,0.5); 
  \draw[fill,->, >=latex][red] (3,0.5) -- (3.9,1.5);
  \draw[fill,->, >=latex][red] (4,-0.5) -- (4.9,-1.5); 
  \draw[fill,->, >=latex][red] (3,-1.5) -- (3.9,-0.5);

  \filldraw[fill=black!100,draw=black!80] (4,1.5) circle (2.5pt)    node[anchor=north] {\small };
  \filldraw[fill=black!100,draw=black!80] (3,0.5) circle (2.5pt)    node[anchor=north] {\small };
  \filldraw[fill=black!100,draw=black!80] (5,0.5) circle (2.5pt)    node[anchor=north] {\small };
  \filldraw[fill=black!100,draw=black!80] (4,-0.5) circle (2.5pt)    node[anchor=north] {\small };
  \filldraw[fill=black!100,draw=black!80] (3,-1.5) circle (2.5pt)    node[anchor=north] {\small };
  \filldraw[fill=black!100,draw=black!80] (5,-1.5) circle (2.5pt)    node[anchor=north] {\small };
\end{tikzpicture}
\end{center}
\vspace{0.5cm}
\begin{center}
\begin{tikzpicture}[y=.3cm, x=.3cm,font=\normalsize,scale=0.7]
\foreach \i in {0,1} \draw [gray] (4*\i,0) -- (4*\i,4);
\foreach \i in {0,1} \draw [gray] (4*\i,4) -- (4*\i,8);
\foreach \i in {0,1} \draw [gray] (4*\i,8) -- (4*\i,12);
\foreach \i in {0,...,3} \draw [gray] (0,4*\i) -- (4,4*\i);
\foreach \i in {0,1} \draw [red, ultra thick] (4*\i,4) -- (4*\i,8);
\foreach \i in {0,3} \draw [red, ultra thick] (0,4*\i) -- (4,4*\i);
\foreach \i in {0,...,3} \filldraw[fill=white!40,draw=black!80] (0,4*\i) circle (3pt)    node[anchor=north] {\small };
\foreach \i in {0,...,3} \filldraw[fill=white!40,draw=black!80] (4,4*\i) circle (3pt)    node[anchor=north] {\small };

\filldraw[fill=white!40,draw=white!80] (7,5) circle (3pt)    node[anchor=south] {\huge $\Rightarrow$};
\end{tikzpicture}
\begin{tikzpicture}[domain=0:4,scale=0.8]
  \draw (3,1) -- (5,1);
  \draw (5,-1) -- (3,-1);
  \draw (3,2) -- (3,-2);
  \draw (5,-2) -- (5,2);
  \draw (3,2) -- (5,2);
  \draw (3,-2) -- (5,-2);
  \draw (4,-1) -- (4,1);

  \draw[fill,->, >=latex][red] (4,-0.5) -- (4.9,0.5); 
  \draw[fill,->, >=latex][red] (3,0.5) -- (3.9,-0.5);
  \draw[fill,->, >=latex][red] (3,-1.5) -- (4.9,-1.5);  

  \filldraw[fill=black!100,draw=black!80] (4,1.5) circle (2.5pt)    node[anchor=north] {\small };
  \filldraw[fill=black!100,draw=black!80] (3,0.5) circle (2.5pt)    node[anchor=north] {\small };
  \filldraw[fill=black!100,draw=black!80] (5,0.5) circle (2.5pt)    node[anchor=north] {\small };
  \filldraw[fill=black!100,draw=black!80] (4,-0.5) circle (2.5pt)    node[anchor=north] {\small };
  \filldraw[fill=black!100,draw=black!80] (3,-1.5) circle (2.5pt)    node[anchor=north] {\small };
  \filldraw[fill=black!100,draw=black!80] (5,-1.5) circle (2.5pt)    node[anchor=north] {\small };
\end{tikzpicture}
\end{center}

Then comparing Proposition~\ref{prop:Catalan_Fibonacci} with Proposition~\ref{prop:Ladder_Fibonacci} we obtain
\[ \left| \begin{array}{cc}
 C_0+C_1 & C_1+C_2 \\ 
 C_1+C_2 & C_2+C_3 \end{array} \right| = \left| \begin{array}{cc}
 2 & 3 \\ 
 3 & 7 \end{array} \right| = F_5 =
\left| \begin{array}{cc}
 2 & 1 \\ 
 1 & 3 \end{array}\right| = \left| \begin{array}{cc}
 F_3 & F_2 \\ 
 F_2 & F_4 \end{array} \right|. \]
  
\end{example}

\begin{remark}
Based on the earlier finding, we can discern that the Hankel matrix discussed in Proposition~\ref{prop:Catalan_Fibonacci} shares an identical determinant with the matrix presented in Proposition~\ref{prop:Ladder_Fibonacci}. Although both matrices exhibit similarities in terms of determinants, it is noteworthy to highlight the contrast in their components. The former matrix encompasses a total of $k^2$ non-zero entries, while the latter matrix boasts a more compact structure with only $3k-2$ non-zero entries. 

This discrepancy in the number of non-zero entries underscores a notable difference in the structures of these matrices. Exploring such differences might provide valuable insights into the underlying structures and relationships in Hankel matrices and contribute to the broader understanding of their significance in mathematical theory and applications.
\end{remark}
%::::::::::::::::::::::::::::::::::::::::::::::::::::::::::::::::::::::::::::::::::::::::::::::::::::::::::::::::::::::::::::::::::::::::::::::::::

 \begin{proposition}\label{prop:Hankel matrix of Catalan numbers} Let $M_{st}=(m_{ij})_{1\leq i,j \leq k}$ be the path matrix associated to the triangular snake graph of the vertical ladder graph $L_{2k-2}$, where $k\geq 1$. Let $C$ be the sequence of Catalan numbers $C_n$ and let $F_n$ be the $n$-th Fibonacci number. Then the following relationship holds:
\[
\det M_{st}=F_{2k}=\det \left( H_{k}(C)+H_{k}^{\prime}(C)-E_{k,k} \right),
\]
where $E_{i,j}$ is the matrix whose $(i,j)$-th entry is $1$ and all other entries are zero, $H_{k}(C)$ and $H_{k}^{\prime}(C)$ are the Hankel matrices of the Catalan numbers and
   \begin{equation*}
     m_{ij} = \left\{
	       \begin{array}{ll}
		 F_3=2      & \mathrm{if\ } i=j=1  \quad\mathrm{or} \quad i=j=k; \\
		 F_4=3      & \mathrm{if\ } i=j\in\{2,\dots,k-1\} ; \\
		 F_2=1      & \mathrm{if\ } i=j+1 \quad \mathrm{or} \quad i=j-1 ; \\
		 F_0=0 & \mathrm{otherwise}.
	       \end{array}
	     \right.
   \end{equation*}
\end{proposition}
\begin{proof}
This follows by the same method as in the proofs of Propositions~\ref{prop:Catalan_Fibonacci} and \ref{prop:Ladder_Fibonacci}, considering the ladder graph $L_{2k-2}$. The following is the triangular snake graph associated to $L_{2k-2}$.
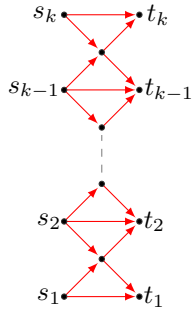
\begin{center}
\begin{tikzpicture}[domain=0:4, scale=0.6]

  \draw[->, >=latex,red] (4,7) -- (4.9,7.9); 
  \draw[->, >=latex,red] (3,8) -- (4.9,8); 
  \draw[->, >=latex,red] (3,8) -- (3.9,7.1); 
  
  \draw[->, >=latex,red] (4,5) -- (4.9,5.9); 
  \draw[->, >=latex,red] (3,6) -- (4.9,6); 
  \draw[->, >=latex,red] (3,6) -- (3.9,5.1); 
  \draw[->, >=latex,red] (3,6) -- (3.9,6.9);
  \draw[->, >=latex,red] (4,7) -- (4.9,6.1);
  
  \draw[gray,dashed] (4,3.8) -- (4,4.8);
  
  \draw[->, >=latex,red] (4,1.5) -- (4.9,2.4); 
  \draw[->, >=latex,red] (3,2.5) -- (3.9,1.6);
  \draw[->, >=latex,red] (3,0.5) -- (4.9,0.5);  
  \draw[->, >=latex,red] (3,2.5) -- (4.9,2.5); 
  
  \draw[->, >=latex,red] (3,0.5) -- (3.9,1.4); 
  \draw[->, >=latex,red] (4,1.5) -- (4.9,0.6); 
  
  \draw[->, >=latex,red] (3,2.5) -- (3.9,3.4);
  \draw[->, >=latex,red] (4,3.5) -- (4.9,2.6);
  
\node (1) at (2.6,0.5) {$s_1$};
\node (2) at (2.6,2.5) {$s_2$};
\node (3) at (5.4,0.5) {$t_1$};
\node (4) at (5.4,2.5) {$t_2$};
\node (5) at (2.2,6) {$s_{k-1}$};
\node (6) at (2.5,8) {$s_{k}$};
\node (7) at (5.8,6) {$t_{k-1}$};
\node (8) at (5.5,8) {$t_{k}$};

  \filldraw[fill=black!100,draw=black!80] (4,3.5) circle (2pt)    node[anchor=north] {\small };
  \filldraw[fill=black!100,draw=black!80] (3,2.5) circle (2pt)    node[anchor=north] {\small };
  \filldraw[fill=black!100,draw=black!80] (5,2.5) circle (2pt)    node[anchor=north] {\small };
  \filldraw[fill=black!100,draw=black!80] (4,1.5) circle (2pt)    node[anchor=north] {\small };
  \filldraw[fill=black!100,draw=black!80] (3,0.5) circle (2pt)    node[anchor=north] {\small };
  \filldraw[fill=black!100,draw=black!80] (5,0.5) circle (2pt)    node[anchor=north] {\small };

  \filldraw[fill=black!100,draw=black!80] (3,8) circle (2pt)    node[anchor=north] {\small };
  \filldraw[fill=black!100,draw=black!80] (5,8) circle (2pt)    node[anchor=north] {\small };
  \filldraw[fill=black!100,draw=black!80] (4,7) circle (2pt)    node[anchor=north] {\small };
  \filldraw[fill=black!100,draw=black!80] (3,6) circle (2pt)    node[anchor=north] {\small };
  \filldraw[fill=black!100,draw=black!80] (5,6) circle (2pt)    node[anchor=north] {\small };
  \filldraw[fill=black!100,draw=black!80] (4,5) circle (2pt)    node[anchor=north] {\small };
\end{tikzpicture}
\end{center}

An easy computation shows that the path matrix of the previous triangular snake graph is the desired one.
\end{proof}

\subsection{Some identities in not straight snake graphs}\label{sec:Some identities in not straight snake graphs}
Building on our previous results regarding Fibonacci numbers and ladder graphs, this chapter extends our analysis to more general snake graphs. We introduce a new notation to define a snake graph in terms of the length of each maximal straight subgraph within this graph, allowing us to study their properties and relationships with Fibonacci numbers.

\begin{definition}\label{def:chain}
    Given a snake graph $\mathcal{G}$, a \textit{chain} $\mathcal{G}_{ij}$ is a subgraph of $\mathcal{G}$ that forms a maximal ladder graph. That is, there are no other ladder graphs within $\mathcal{G}$ that contain $\mathcal{G}_{ij}$, and $\mathcal{G}_{ij}$ consists of the tiles $G_k$, where $i\leq k \leq j$. The \textit{length} $l(\mathcal{G}_{ij})$ of a chain $\mathcal{G}_{ij}$ is defined as the number of tiles it contains, \ie, $l(\mathcal{G}_{ij})=j-i+1$.
\end{definition}
Based on Definition~\ref{def:chain}, we can describe a snake graph using its horizontal and vertical chains. We denoted by $\mathcal{G}_h(l_1,l_2,\dots,l_k)$ the snake graph in which $l_i$ represents the length of the horizontal chains when $i$ is odd, and the length of the vertical chains when $i$ is even. Conversely, $\mathcal{G}_v(l_1,l_2,\dots,l_k)$ denotes the snake graph where $l_i$ represents the length of the vertical chains if $i$ is odd, and the length of the horizontal chains if  $i$ is even. The concept of a chain can be analogously defined for triangular snake graphs.

\begin{remark}
The only snake graph that contains a chain of length $1$ is the one formed by a single tile, denoted as \(\mathcal{G} = G_1 = \mathcal{G}_h(1) = \mathcal{G}_v(1)\). In all other cases, the chains have a length of at least $2$. However, in some computations (see Lemma~\ref{lemma:sum-prod-Fibonacci}), initial chains of length $0$ or $1$ may appear. By convention, we define \(\mathcal{G}_h(0, l_2, \dots, l_k) = \mathcal{G}_v(l_2 - 1, l_3, \dots, l_k)\) and \(\mathcal{G}_h(1, l_2, \dots, l_k) = \mathcal{G}_v(l_2, l_3, \dots, l_k)\). Similarly, we set \(\mathcal{G}_v(0, l_2, \dots, l_k) = \mathcal{G}_h(l_2 - 1, l_3, \dots, l_k)\) and \(\mathcal{G}_v(1, l_2, \dots, l_k) = \mathcal{G}_h(l_2, l_3, \dots, l_k)\).
\end{remark}
\begin{example} The snake graph $\mathcal{G}_h(2,2,4,3,2,2)$ is illustrated below:
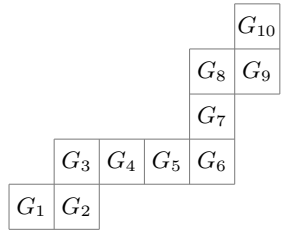
\begin{center}
\begin{tikzpicture}[y=0.3cm, x=0.3cm,font=\normalsize, scale=0.5]

\draw [gray] (25,0) -- (35,0);
\draw [gray] (30,10) -- (35,10);
\draw [gray] (25,5) -- (35,5);
\draw [gray] (35,0) -- (35,10);
\draw [gray] (30,-10) -- (30,10);
\draw [gray] (25,-10) -- (25,5);
\draw [gray] (10,-5) -- (30,-5);
\draw [gray] (5,-10) -- (30,-10);
\draw [gray] (20,-5) -- (20,-10);
\draw [gray] (15,-5) -- (15,-15);
\draw [gray] (10,-5) -- (10,-15);
\draw [gray] (5,-10) -- (5,-15);
\draw [gray] (5,-15) -- (15,-15);

    \filldraw (7.5,-12.5) node[anchor=center] {\small $G_1$};
    \filldraw (12.5,-12.5) node[anchor=center] {\small $G_2$};
    \filldraw (12.5,-7.5) node[anchor=center] {\small $G_3$};   
    \filldraw (17.5,-7.5) node[anchor=center] {\small $G_4$};
    \filldraw (22.5,-7.5) node[anchor=center] {\small $G_5$};
    \filldraw (27.5,-7.5) node[anchor=center] {\small $G_6$};
    \filldraw (27.5,-2.5) node[anchor=center] {\small $G_7$};
    \filldraw (27.5,2.5) node[anchor=center] {\small $G_8$};
    \filldraw (32.5,2.5) node[anchor=center] {\small $G_9$};
    \filldraw (32.5,7.5) node[anchor=center] {\small $G_{10}$};
\end{tikzpicture}
\end{center}
This snake graph consists of $3$ horizontal chains:
\begin{align*}
\mathcal{G}_{1,2}&=\{G_1,G_2\}, \\ \mathcal{G}_{3,6}&=\{G_3,G_4,G_5,G_6\}, \\ \mathcal{G}_{8,9}&=\{G_8,G_9\},
\end{align*}
and $3$ vertical chains:
\begin{align*}
    \mathcal{G}_{2,3}&=\{G_2,G_3\}, \\ \mathcal{G}_{6,8}&=\{G_6,G_7,G_8\}, \\
    \mathcal{G}_{9,10}&=\{G_9,G_{10}\}.
\end{align*}
\end{example}

This chain description allows for recursive combinatorial counting using Fibonacci numbers, as illustrated in the following lemma.

\begin{lemma}\label{lemma:sum-prod-Fibonacci}
Let $\mathcal{G}$ be a snake graph represented by $\mathcal{G}_h(l_1,l_2,\dots,l_k)$, where $l_i\geq 2$ for all $i\in \{1,\dots,k\}$. The number of perfect matchings $m(\mathcal{G})$ of $\mathcal{G}$ is given by the recurrence relation:
\[
m(\mathcal{G})=F_{l_1}\ m(\mathcal{G}_v(l_2-1,l_3,\dots,l_k))+F_{l_1+1}\ m(\mathcal{G}_v(l_2-2,l_3,\dots,l_k)).
\]
\end{lemma}
\begin{proof}
We classify all perfect matchings of the snake graph $\mathcal{G} = \mathcal{G}_h(l_1,l_2,\dots,l_k)$ by considering the edge incident to the lower right vertex of the tile $G_{l_1}$. This edge can either be a vertical or an horizontal edge, as illustrated in the following figure:
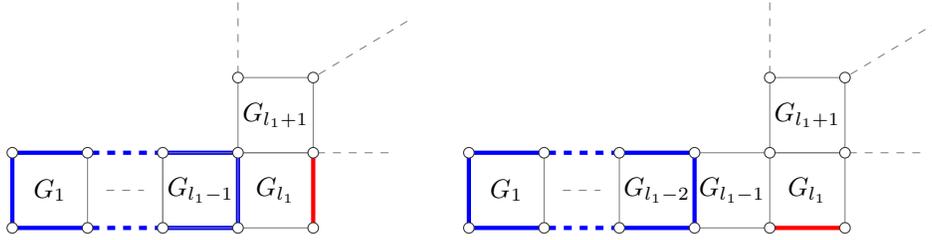
\begin{center}
    \begin{tikzpicture}[scale=1]

\filldraw(-2.5,0.5) node[anchor=center] {$G_{1}$};

\draw[gray] (-1,0) -- (-1,1);
\draw[gray] (-2,0) -- (-2,1);
\draw[blue, ultra thick] (0,0) -- (0,1);
\draw[blue, ultra thick] (0,1) -- (-1,1);
\draw[blue, ultra thick] (0,0) -- (-1,0);

\draw[blue, ultra thick, dashed] (-1,1) -- (-2,1);
\draw[blue, ultra thick,dashed] (-1,0) -- (-2,0);

\draw[blue, ultra thick] (-2,1) -- (-3,1);
\draw[blue, ultra thick] (-2,0) -- (-3,0);

\draw[blue, ultra thick] (-3,0) -- (-3,1);

\draw[gray] (0,0) -- (0,1);
\draw[red, ultra thick] (1,0) -- (1,1);
\draw[gray] (1,1) -- (0,1);
\draw[gray] (0,0) -- (1,0);

%\filldraw(0.5,0.5)     %node[anchor=center] {$G_{2m-1}$};

\draw[gray] (0,1) -- (0,2);
\draw[gray] (1,1) -- (1,2);
\draw[gray] (1,2) -- (0,2);
\draw[gray] (0,1) -- (1,1);

%\filldraw(0.5,1.5)     %node[anchor=center] {$G_{2m}$};

\draw[gray] (-1,1) -- (0,1);
\draw[gray] (-1,1) -- (-1,0);
\draw[gray] (-1,0) -- (0,0);

\filldraw(-0.5,0.5) node[anchor=center] {$G_{l_1-1}$};
\filldraw(0.5,0.5) node[anchor=center] {$G_{l_1}$};
\filldraw(0.5,1.5) node[anchor=center] {$G_{l_1+1}$};

%\draw[dashed, gray] (-2.75,0.5) -- (-2.25,0.5);
\draw[dashed, gray] (2,1) -- (1,1);
\draw[dashed, gray] (0,3) -- (0,2);

\draw[dashed, gray] (-1.75,0.5) -- (-1.25,0.5);
\draw[dashed, gray] (2.25,2.75) -- (1,2);

\filldraw[fill=white!40,draw=black!80] (-3,0) circle (2pt);
\filldraw[fill=white!40,draw=black!80] (-3,1) circle (2pt);

\filldraw[fill=white!40,draw=black!80] (-2,0) circle (2pt);
\filldraw[fill=white!40,draw=black!80] (-2,1) circle (2pt);

\filldraw[fill=white!40,draw=black!80] (-1,0) circle (2pt);
\filldraw[fill=white!40,draw=black!80] (-1,1) circle (2pt);

\filldraw[fill=white!40,draw=black!80] (0,0) circle (2pt);
\filldraw[fill=white!40,draw=black!80] (1,0) circle (2pt);
\filldraw[fill=white!40,draw=black!80] (0,1) circle (2pt);
\filldraw[fill=white!40,draw=black!80] (1,1) circle (2pt);
\filldraw[fill=white!40,draw=black!80] (0,2) circle (2pt);
\filldraw[fill=white!40,draw=black!80] (1,2) circle (2pt);
\end{tikzpicture} \hspace{0.5 cm} \begin{tikzpicture}[scale=1]

\filldraw(-1.5,0.5) node[anchor=center] {$G_{l_1-2}$};
\filldraw(-3.5,0.5) node[anchor=center] {$G_{1}$};

\draw[gray] (-2,0) -- (-2,1);
\draw[gray] (-3,0) -- (-3,1);
\draw[blue, ultra thick] (-1,0) -- (-1,1);
\draw[blue, ultra thick] (-1,1) -- (-2,1);
\draw[blue, ultra thick] (-1,0) -- (-2,0);

\draw[blue, ultra thick, dashed] (-2,1) -- (-3,1);
\draw[blue, ultra thick,dashed] (-2,0) -- (-3,0);

\draw[blue, ultra thick] (-3,1) -- (-4,1);
\draw[blue, ultra thick] (-3,0) -- (-4,0);

\draw[blue, ultra thick] (-4,0) -- (-4,1);

\filldraw(-0.5,0.5) node[anchor=center] {$G_{l_1-1}$};

\draw[gray] (0,0) -- (0,1);
\draw[gray] (1,0) -- (1,1);
\draw[gray] (1,1) -- (0,1);
\draw[red, ultra thick] (0,0) -- (1,0);

\filldraw(0.5,0.5) node[anchor=center] {$G_{l_1}$};

\draw[gray] (0,1) -- (0,2);
\draw[gray] (1,1) -- (1,2);
\draw[gray] (1,2) -- (0,2);
\draw[gray] (0,1) -- (1,1);

\filldraw(0.5,1.5) node[anchor=center] {$G_{l_1+1}$};

\draw[gray] (-1,1) -- (0,1);
%\draw[gray] (-1,1) -- (-1,0);
\draw[gray] (-1,0) -- (0,0);

\draw[dashed, gray] (-2.75,0.5) -- (-2.25,0.5);
\draw[dashed, gray] (2,1) -- (1,1);
\draw[dashed, gray] (0,3) -- (0,2);

%\draw[dashed, gray] (-1.75,0.5) -- (-1.25,0.5);
\draw[dashed, gray] (2.25,2.75) -- (1,2);

\filldraw[fill=white!40,draw=black!80] (-4,0) circle (2pt);
\filldraw[fill=white!40,draw=black!80] (-4,1) circle (2pt);

\filldraw[fill=white!40,draw=black!80] (-3,0) circle (2pt);
\filldraw[fill=white!40,draw=black!80] (-3,1) circle (2pt);

\filldraw[fill=white!40,draw=black!80] (-2,0) circle (2pt);
\filldraw[fill=white!40,draw=black!80] (-2,1) circle (2pt);

\filldraw[fill=white!40,draw=black!80] (-1,0) circle (2pt);
\filldraw[fill=white!40,draw=black!80] (-1,1) circle (2pt);

\filldraw[fill=white!40,draw=black!80] (0,0) circle (2pt);
\filldraw[fill=white!40,draw=black!80] (1,0) circle (2pt);
\filldraw[fill=white!40,draw=black!80] (0,1) circle (2pt);
\filldraw[fill=white!40,draw=black!80] (1,1) circle (2pt);
\filldraw[fill=white!40,draw=black!80] (0,2) circle (2pt);
\filldraw[fill=white!40,draw=black!80] (1,2) circle (2pt);
\end{tikzpicture}
\end{center}

If a perfect matching contains the vertical edge, then the edges on the left part of the snake graph are chosen in $m(L_{l_1-1}) = F_{l_1+1}$ ways, and the edges of the remaining part of the snake graph are chosen in $m(\mathcal{G}_v(l_2-2,l_3,\dots,l_k))$ ways. Analogously, for a perfect matching that contains the horizontal edge, the edges on the left part of the snake graph are chosen in $m(L_{l_1-2}) = F_{l_1}$ ways, and the edges of the remaining part of the snake graph are chosen in $m(\mathcal{G}_v(l_2-1,l_3,\dots,l_k))$ ways. Therefore, the total number of perfect matchings $m(\mathcal{G})$ is given by the recurrence:
    \[
    m(\mathcal{G}) = F_{l_1}\ m(\mathcal{G}_v(l_2-1,l_3,\dots,l_k)) + F_{l_1+1}\ m(\mathcal{G}_v(l_2-2,l_3,\dots,l_k)),
    \]
    as desired.
\end{proof}

\begin{remark}
As a preliminary observation, we note that a recurrence for $m(\mathcal{G}_v(l_1,\dots,l_k))$ can be derived in a manner analogous to the one presented in Lemma~\ref{lemma:sum-prod-Fibonacci}. This is because $\mathcal{G}_v(l_1,l_2,\dots,l_k)$ is a reflection of $\mathcal{G}_h(l_1,l_2,\dots,l_k)$, implying that both graphs are structurally identical, just mirrored. Consequently, we have \[m(\mathcal{G}_v(l_1,l_2,\dots,l_k)) = m(\mathcal{G}_h(l_1,l_2,\dots,l_k)).\] Additionally, for any snake graph, the number of perfect matchings can be expressed as a sum of products of Fibonacci numbers. In \S~\ref{sec:ladder_graphs_hankel_and_general_formula}, we observed that the entries of the path matrix for the triangular snake graph associated with a ladder graph are Fibonacci numbers. This leads us to inquire whether it is possible to express the entries of the path matrices for general snake graphs in terms of Fibonacci-like expressions.
\end{remark}

\begin{proposition}\label{prop:General Fibonacci path matrix} Let $M_{st}=(m_{ij})_{1\leq i,j \leq k}$ be the path matrix associated to the following triangular snake graph:
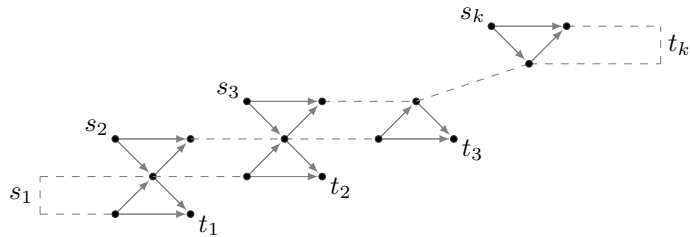
\begin{center}
\begin{tikzpicture}[scale=0.6]

   \node (1) at (-2.5,1) {$s_1$};
   \node (1) at (-0.5,2.8) {$s_2$};
   \node (1) at (3,3.8) {$s_3$};
   \node (1) at (9.5,5.8) {$s_k$};
   \node (1) at (2.5,0.2) {$t_1$};
   \node (1) at (6,1.2) {$t_2$};
   \node (1) at (9.5,2.2) {$t_3$};
   \node (1) at (15,5) {$t_k$};

%Arrows l1-l2

\draw[gray, fill, ->, >=latex] (0,0.5) -- (0.9,1.4);
%\draw[red, ultra thick, fill, ->, >=latex] (1,1.5) -- (1.9,0.6);
\draw[gray, fill, ->, >=latex] (1,1.5) -- (1.9,0.6);
\draw[gray, fill, ->, >=latex] (0,0.5) -- (1.9,0.5);

\draw[gray, fill, ->, >=latex] (1,1.5) -- (1.9,2.4);
\draw[gray, fill, ->, >=latex] (0,2.5) -- (0.9,1.6);
\draw[gray, fill, ->, >=latex] (0,2.5) -- (1.9,2.5);

%\draw[gray, fill, ->, >=latex] (-1,1.5) -- (-0.1,0.6);
%\draw[gray, fill, ->, >=latex] (-1,1.5) -- (0.9,1.5);

%:) :) :)

%\draw[dashed, gray] (1,2.75) -- (1,3.5);

\draw[dashed, gray] (-2,1.5) -- (1,1.5);
\draw[dashed, gray] (-2,0.5) -- (0,0.5);
\draw[dashed, gray] (-2,0.5) -- (-2,1.5);

%\black points
\foreach \x in {0, 2} 
\filldraw[fill=black!100,draw=black!80] (\x,0.5) circle (2.5pt)    node[anchor=north] {\small};

\foreach \x in {1}
\filldraw[fill=black!100,draw=black!80] (\x,1.5) circle (2.5pt)    node[anchor=north] {\small };

\foreach \x in {0, 2} 
\filldraw[fill=black!100,draw=black!80] (\x,2.5) circle (2.5pt)    node[anchor=north] {\small};

%Arrows l2-l3

\draw[gray, fill, ->, >=latex] (3.5,1.5) -- (4.4,2.4);
\draw[gray, fill, ->, >=latex] (4.5,2.5) -- (5.4,1.6);
\draw[gray, fill, ->, >=latex] (3.5,1.5) -- (5.4,1.5);

\draw[gray, fill, ->, >=latex] (4.5,2.5) -- (5.4,3.4);
\draw[gray, fill, ->, >=latex] (3.5,3.5) -- (4.4,2.6);
\draw[gray, fill, ->, >=latex] (3.5,3.5) -- (5.4,3.5);

%\draw[gray, fill, ->, >=latex] (4,2.5) -- (4.9,1.6);
%\draw[gray, fill, ->, >=latex] (4,2.5) -- (5.9,2.5);

\draw[dashed, gray] (2,2.5) -- (4.5,2.5);
\draw[dashed, gray] (1,1.5) -- (3.5,1.5);

\draw[dashed, gray] (5.5,3.5) -- (8,3.5);
\draw[dashed, gray] (4.5,2.5) -- (7,2.5);

%\black points
\foreach \x in {3.5, 5.5} 
\filldraw[fill=black!100,draw=black!80] (\x,1.5) circle (2.5pt)    node[anchor=north] {\small};

\foreach \x in {4.5}
\filldraw[fill=black!100,draw=black!80] (\x,2.5) circle (2.5pt)    node[anchor=north] {\small };

\foreach \x in {3.5, 5.5} 
\filldraw[fill=black!100,draw=black!80] (\x,3.5) circle (2.5pt)    node[anchor=north] {\small};

\draw[gray, fill, ->, >=latex] (7,2.5) -- (7.9,3.4);
\draw[gray, fill, ->, >=latex] (8,3.5) -- (8.9,2.6);
\draw[gray, fill, ->, >=latex] (7,2.5) -- (8.9,2.5);

%\black points
\foreach \x in {7, 9} 
\filldraw[fill=black!100,draw=black!80] (\x,2.5) circle (2.5pt)    node[anchor=north] {\small};

\foreach \x in {8}
\filldraw[fill=black!100,draw=black!80] (\x,3.5) circle (2.5pt)    node[anchor=north] {\small };

%etc
\draw[dashed, gray] (8,3.5) -- (11,4.5);

%Arrows lk-lk+1
\draw[gray, fill, ->, >=latex] (11,4.5) -- (11.9,5.4);
\draw[gray, fill, ->, >=latex] (10,5.5) -- (10.9,4.6);
\draw[gray, fill, ->, >=latex] (10,5.5) -- (11.9,5.5);

%\draw[dashed, gray] (9,2.5) -- (12,2.5);
%\draw[dashed, gray] (8,1.5) -- (11,1.5);

\draw[dashed, gray] (12,5.5) -- (14.5,5.5);
\draw[dashed, gray] (11,4.5) -- (14.5,4.5);
\draw[dashed, gray] (14.5,4.5) -- (14.5,5.5);

%\black points
%\foreach \x in {11, 13} 
%\filldraw[fill=black!100,draw=black!80] (\x,1.5) circle (2.5pt)    node[anchor=north] {\small};

\foreach \x in {11}
\filldraw[fill=black!100,draw=black!80] (\x,4.5) circle (2.5pt)    node[anchor=north] {\small };

\foreach \x in {10, 12} 
\filldraw[fill=black!100,draw=black!80] (\x,5.5) circle (2.5pt)    node[anchor=north] {\small};

%\draw[gray, fill, ->, >=latex] (13.5,4.5) -- (14.4,5.4);
%\draw[gray, fill, ->, >=latex] (14.5,5.5) -- (15.4,4.6);
%\draw[gray, fill, ->, >=latex] (13.5,4.5) -- (15.4,4.5);

%\black points
%\foreach \x in {13.5, 15.5} 
%\filldraw[fill=black!100,draw=black!80] (\x,4.5) circle (2.5pt)    node[anchor=north] {\small};

%\foreach \x in {14.5}
%\filldraw[fill=black!100,draw=black!80] (\x,5.5) circle (2.5pt)    node[anchor=north] {\small };
\end{tikzpicture}
\end{center}
derived from the snake graph $\mathcal{G}_h(l_1,2,l_2,2,\dots,2,l_k)$ such that edge contraction of the $k-1$ vertical chains results in the hourglass graphs $\hourglass_i$, $i\in\{1,\dots,k-1\}$. Then, the elements of $M_{st}$ are given by:
   \begin{equation*}
     m_{ij} = \left\{\begin{array}{ll}
     \displaystyle{F_{l_i+1}F_{l_j+1}\prod_{r=i+1}^{j-1}F_{l_{r}}} & \mathrm{if\ } i<j; \\
     F_{l_{i}+2} & \mathrm{if\ } i=j; \\
     F_2 & \mathrm{if\ } i=j+1; \\
     F_0 & \mathrm{otherwise}.
	              \end{array} \right.
   \end{equation*}
\end{proposition}
\begin{proof}
    Since the vertical chains of the snake graph $\mathcal{G}_h(l_1,2,l_2,2,\dots,2,l_k)$ are all of length two, and their edge contraction results in an hourglass graph, the number of paths from vertex \( s_i \) to vertex \( t_i \), for \( i \in \{1, \dots, k\} \), is equal to the Fibonacci number \( F_{l_i+2} \), as the graph corresponds to a triangular snake graph associated to a ladder graph $L_{l_i}$. On the other hand, there is only \( F_2 = 1 \) path from vertex \( s_i \) to vertex \( t_j \) if \( i - j = 1 \), and \( F_0 = 0 \) if \( i - j > 1 \). 

    Finally, to address the remaining case, we note that a path from \( s_i \) to \( s_j \) for \( i < j \) must pass through the necks of the intermediate hourglass graphs $\hourglass_i, \dots, \hourglass_{j-1}$, so does not pass through the arrows incident to the source or sink vertices of these graphs. Specifically, there are \( F_{l_i+1} \) paths from \( s_i \) to the neck \( n_i \) of \( \hourglass_i \), \( F_{l_j+1} \) paths from \( n_{j-1} \) to \( t_j \), and \( F_{l_r} \) paths from \( n_{r-1} \) to \( n_r \) for each \( r \in \{i+1, \dots, j-1\} \). Therefore, the total number of paths is \( m_{ij} = \displaystyle{F_{l_i+1}F_{l_j+1}\prod_{r=i+1}^{j-1}F_{l_r}} \), as desired.
\end{proof}

\begin{remark} 
In Proposition~\ref{prop:General Fibonacci path matrix}, the associated triangular snake graph can be obtained by applying either the contraction defined in Definition~\ref{def:contracted_snake_graph} or the opposite contraction considered in Remark~\ref{re:opposite edge contraction}. For general snake graphs where not all edge contractions of vertical chains result in hourglass graphs, it is still possible to express path matrices in terms of Fibonacci numbers. However, these cases involve applying both the original and opposite contractions simultaneously. Consequently, some entries in the path matrices may be determinants of smaller path matrices, which themselves can be described using Fibonacci numbers.
\end{remark}

In \S~\ref{sec:ladder_graphs_hankel_and_general_formula}, we examined path matrices whose entries were expressed in terms of Fibonacci or Catalan numbers, with determinants also yielding Fibonacci numbers. We now explore an example where the determinant corresponds to a different well-known numerical sequence.

\begin{proposition}
Let $P_n$ be the $n$-th Pell number defined recursively by $P_{n}= 2P_{n-1} + P_{n-2}$ for $P_0 = 0$, $P_1 = 1$ and $n\geq 2$. Let $M_k=(m_{ij})$ and $M_k'=(m'_{ij})$ be the $k\times k$ matrices defined as follows:
   \begin{equation*}
     m_{ij} = \left\{\begin{array}{ll}
     F_{5} & \mathrm{if\ } i=j=1; \\
     F_{5}+F_{3} & \mathrm{if\ } i=j\in \{2,\dots,k\}; \\
     F_{4}F_{3}& \mathrm{if\ } i=1 \mathrm{\ and\ } j>1; \\
     (F_{4}+F_2)F_{3}& \mathrm{if\ } i\neq 1 \mathrm{\ and\ } i<j; \\
     F_2 & \mathrm{if\ } i=j+1; \\
     F_0 & \mathrm{otherwise}.
	              \end{array} \right. 
   \end{equation*} and    \begin{equation*}
     m'_{ij} = \left\{\begin{array}{ll}
     F_{5} & \mathrm{if\ } i=j=1; \\
     F_{5}+F_{3} & \mathrm{if\ } i=j\in \{2,\dots,k-1\}; \\
     F_{4}F_{3}& \mathrm{if\ } i=1 \mathrm{\ and\ } 1<j<k; \\
     (F_{4}+F_2)F_{3}& \mathrm{if\ } i\neq 1, \mathrm{\ } j\neq k \mathrm{\ and\ } i<j; \\
     F_{4}+F_2& \mathrm{if\ } 1<i<k \mathrm{\ and\ } j=k; \\
     F_{4} & \mathrm{if\ } i\in\{1,k\} \mathrm{\ and\ } j=k; \\
     F_2 & \mathrm{if\ } i=j+1; \\
     F_0 & \mathrm{otherwise}.
	              \end{array} \right. 
   \end{equation*}
   Then, 
   \[
   \det M_k = P_{2k+1} \quad \text{and} \quad \det M_k' = P_{2k}.
   \]
\end{proposition}

\begin{proof}
Consider the snake graph $\mathcal{G}$ associated to the continued fraction $[\underbrace{2,2,\cdots,2,2}_{(2k)-times}]$. This snake graph consists of $2k-1$ chains, each of length $3$; that is, $\mathcal{G}=\mathcal{G}_h(\underbrace{3,3,\cdots,3,3}_{(2k-1)-times})$. From a direct inspection of Figure~\ref{fig:Pell snake graph}, we can observe that the matrix $M_k$ corresponds to the path matrix associated to the triangular snake graph $\mathcal{T}_{\mathcal{G}}$. 
%\vspace{-0.5cm}
\begin{figure}[ht]
\begin{center}
\begin{tikzpicture}[scale=1.4]   
\draw[gray] (2.5,1) -- (2.5,1.5);
\draw[gray] (3,1) -- (3,1.5);
\draw[gray] (0.5,0) -- (2,0);
\draw[gray] (0.5,0.5) -- (2,0.5);
\draw[gray] (1.5,1.5) -- (3,1.5);
\draw[gray] (1.5,1) -- (3,1);
\draw[gray] (2,0) -- (2,1.5);
\draw[gray] (1.5,0) -- (1.5,1.5);
\draw[gray] (1,0) -- (1,0.5);
\draw[gray] (0.5,0) -- (0.5,0.5);

\draw[dashed, gray] (3,1.5) -- (3.5,2);

\end{tikzpicture} \hspace{1cm} \begin{tikzpicture}[domain=0:4, scale=1.4]

\draw[->, >=latex, gray] (0,0) -- (0.9,0); 
\draw[->, >=latex, gray] (0,0) -- (0.4,0.4); 
\draw[->, >=latex, gray] (0.5,0.5) -- (0.9,0.1);

\draw[->, >=latex, gray] (0.5,0.5) -- (1.4,0.5); 
\draw[->, >=latex, gray] (1,0) -- (1.4,0.4); 

\draw[->, >=latex, gray] (1,0) -- (1.9,0); 
\draw[->, >=latex, gray] (1.5,0.5) -- (1.9,0.1);

\draw[->, >=latex, gray] (1,1) -- (1.4,0.6);
\draw[->, >=latex, gray] (1,1) -- (1.9,1); 
\draw[->, >=latex, gray] (1.5,0.5) -- (1.9,0.9); 

\draw[->, >=latex, gray] (1,1) -- (1.4,1.4);
\draw[->, >=latex, gray] (1.5,1.5) -- (1.9,1.1);

\draw[->, >=latex, gray] (2,1) -- (2.4,1.4); 
\draw[->, >=latex, gray] (1.5,1.5) -- (2.4,1.5); 

\draw[->, >=latex, gray] (2,1) -- (2.9,1);
\draw[->, >=latex, gray] (2.5,1.5) -- (2.9,1.1);

\draw[dashed, gray] (2.5,1.5) -- (3,2);

%:::::::::::::::::::::::::::::::::::::::::::::::::::::::::::::::::::::::::::::::::
\node (1) at (0,0) {$\bullet$};
\node (1) at (2,0) {$\bullet$};
\node (1) at (1,1) {$\bullet$};
\node (1) at (1.5,1.5) {$\bullet$};
\node (1) at (2.5,1.5) {$\bullet$};
\node (1) at (1,0) {$\bullet$};
\node (1) at (0.5,0.5) {$\bullet$};
\node (1) at (1.5,0.5) {$\bullet$};
\node (1) at (2,1) {$\bullet$};
\node (1) at (3,1) {$\bullet$};
\node (1) at (-0.25,0) {$s_1$};
\node (1) at (2.25,0) {$t_1$};
\node (1) at (0.75,1) {$s_2$};
\node (1) at (3.25,1) {$t_2$};

\filldraw(0.5,0.24)     node[anchor=center] {$1$};
\filldraw(1,0.26)     node[anchor=center] {$2$};
\filldraw(1.5,0.24)     node[anchor=center] {$3$};
\filldraw(1.5,0.76)     node[anchor=center] {$4$};
\filldraw(1.5,1.26)     node[anchor=center] {$5$};
\filldraw(2,1.26)     node[anchor=center] {$6$};
\filldraw(2.5,1.26)     node[anchor=center] {$7$};

\end{tikzpicture}
\end{center}
\caption{Pell Snake graph and its corresponding triangular snake graph.}
\label{fig:Pell snake graph}
\end{figure}
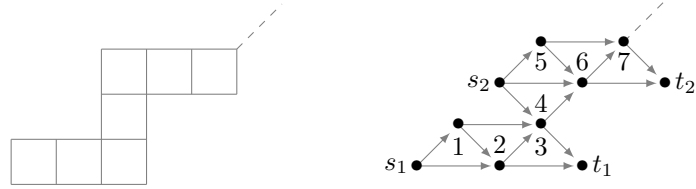
Furthermore, by directly examining the routes in 
$\mathcal{T}_{\mathcal{G}}$, we observe that the paths from 
$s_1$ to 
$t_1$ fall into one of the following three cases:
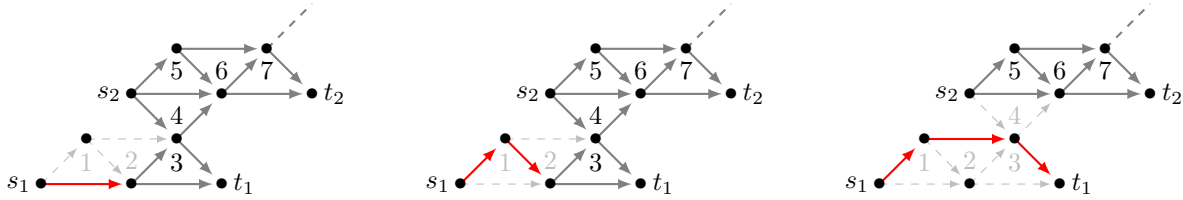
\begin{center}
\begin{tikzpicture}[domain=0:4, scale=1.4]

\draw[->, >=latex, red, thick] (0,0) -- (0.9,0); 
\draw[->, >=latex, lightgray, dashed] (0,0) -- (0.4,0.4); 
\draw[->, >=latex, lightgray, dashed] (0.5,0.5) -- (0.9,0.1);

\draw[->, >=latex, lightgray, dashed] (0.5,0.5) -- (1.4,0.5); 
\draw[->, >=latex, gray, thick] (1,0) -- (1.4,0.4); 

\draw[->, >=latex, gray, thick] (1,0) -- (1.9,0); 
\draw[->, >=latex, gray, thick] (1.5,0.5) -- (1.9,0.1);

\draw[->, >=latex, gray, thick] (1,1) -- (1.4,0.6);
\draw[->, >=latex, gray, thick] (1,1) -- (1.9,1); 
\draw[->, >=latex, gray, thick] (1.5,0.5) -- (1.9,0.9); 

\draw[->, >=latex, gray, thick] (1,1) -- (1.4,1.4);
\draw[->, >=latex, gray, thick] (1.5,1.5) -- (1.9,1.1);

\draw[->, >=latex, gray, thick] (2,1) -- (2.4,1.4); 
\draw[->, >=latex, gray, thick] (1.5,1.5) -- (2.4,1.5); 

\draw[->, >=latex, gray, thick] (2,1) -- (2.9,1);
\draw[->, >=latex, gray, thick] (2.5,1.5) -- (2.9,1.1);

\draw[dashed, gray, thick] (2.5,1.5) -- (3,2);

%:::::::::::::::::::::::::::::::::::::::::::::::::::::::::::::::::::::::::::::::::
\node (1) at (0,0) {$\bullet$};
\node (1) at (2,0) {$\bullet$};
\node (1) at (1,1) {$\bullet$};
\node (1) at (1.5,1.5) {$\bullet$};
\node (1) at (2.5,1.5) {$\bullet$};
\node (1) at (1,0) {$\bullet$};
\node (1) at (0.5,0.5) {$\bullet$};
\node (1) at (1.5,0.5) {$\bullet$};
\node (1) at (2,1) {$\bullet$};
\node (1) at (3,1) {$\bullet$};
\node (1) at (-0.25,0) {$s_1$};
\node (1) at (2.25,0) {$t_1$};
\node (1) at (0.75,1) {$s_2$};
\node (1) at (3.25,1) {$t_2$};

\filldraw(0.5,0.24)     node[anchor=center] {$\textcolor{lightgray}{1}$};
\filldraw(1,0.26)     node[anchor=center] {$\textcolor{lightgray}{2}$};
\filldraw(1.5,0.24)     node[anchor=center] {$3$};
\filldraw(1.5,0.76)     node[anchor=center] {$4$};
\filldraw(1.5,1.26)     node[anchor=center] {$5$};
\filldraw(2,1.26)     node[anchor=center] {$6$};
\filldraw(2.5,1.26)     node[anchor=center] {$7$};

\end{tikzpicture} \hspace{0.6cm} \begin{tikzpicture}[domain=0:4, scale=1.4]

\draw[->, >=latex, lightgray, dashed] (0,0) -- (0.9,0); 
\draw[->, >=latex, red, thick] (0,0) -- (0.4,0.4); 
\draw[->, >=latex, red, thick] (0.5,0.5) -- (0.9,0.1);

\draw[->, >=latex, lightgray, dashed] (0.5,0.5) -- (1.4,0.5); 
\draw[->, >=latex, gray, thick] (1,0) -- (1.4,0.4); 

\draw[->, >=latex, gray, thick] (1,0) -- (1.9,0); 
\draw[->, >=latex, gray, thick] (1.5,0.5) -- (1.9,0.1);

\draw[->, >=latex, gray, thick] (1,1) -- (1.4,0.6);
\draw[->, >=latex, gray, thick] (1,1) -- (1.9,1); 
\draw[->, >=latex, gray, thick] (1.5,0.5) -- (1.9,0.9); 

\draw[->, >=latex, gray, thick] (1,1) -- (1.4,1.4);
\draw[->, >=latex, gray, thick] (1.5,1.5) -- (1.9,1.1);

\draw[->, >=latex, gray, thick] (2,1) -- (2.4,1.4); 
\draw[->, >=latex, gray, thick] (1.5,1.5) -- (2.4,1.5); 

\draw[->, >=latex, gray, thick] (2,1) -- (2.9,1);
\draw[->, >=latex, gray, thick] (2.5,1.5) -- (2.9,1.1);

\draw[dashed, gray, thick] (2.5,1.5) -- (3,2);

%:::::::::::::::::::::::::::::::::::::::::::::::::::::::::::::::::::::::::::::::::
\node (1) at (0,0) {$\bullet$};
\node (1) at (2,0) {$\bullet$};
\node (1) at (1,1) {$\bullet$};
\node (1) at (1.5,1.5) {$\bullet$};
\node (1) at (2.5,1.5) {$\bullet$};
\node (1) at (1,0) {$\bullet$};
\node (1) at (0.5,0.5) {$\bullet$};
\node (1) at (1.5,0.5) {$\bullet$};
\node (1) at (2,1) {$\bullet$};
\node (1) at (3,1) {$\bullet$};
\node (1) at (-0.25,0) {$s_1$};
\node (1) at (2.25,0) {$t_1$};
\node (1) at (0.75,1) {$s_2$};
\node (1) at (3.25,1) {$t_2$};

\filldraw(0.5,0.24)     node[anchor=center] {$\textcolor{lightgray}{1}$};
\filldraw(1,0.26)     node[anchor=center] {$\textcolor{lightgray}{2}$};
\filldraw(1.5,0.24)     node[anchor=center] {$3$};
\filldraw(1.5,0.76)     node[anchor=center] {$4$};
\filldraw(1.5,1.26)     node[anchor=center] {$5$};
\filldraw(2,1.26)     node[anchor=center] {$6$};
\filldraw(2.5,1.26)     node[anchor=center] {$7$};

\end{tikzpicture} \hspace{0.6cm} \begin{tikzpicture}[domain=0:4, scale=1.4]

\draw[->, >=latex, lightgray, dashed] (0,0) -- (0.9,0); 
\draw[->, >=latex, red, thick] (0,0) -- (0.4,0.4); 
\draw[->, >=latex, lightgray, dashed] (0.5,0.5) -- (0.9,0.1);

\draw[->, >=latex, red, thick] (0.5,0.5) -- (1.4,0.5); 
\draw[->, >=latex, lightgray, dashed] (1,0) -- (1.4,0.4); 

\draw[->, >=latex, lightgray, dashed] (1,0) -- (1.9,0); 
\draw[->, >=latex, red, thick] (1.5,0.5) -- (1.9,0.1);

\draw[->, >=latex, lightgray, dashed] (1,1) -- (1.4,0.6);
\draw[->, >=latex, gray, thick] (1,1) -- (1.9,1); 
\draw[->, >=latex, lightgray, dashed] (1.5,0.5) -- (1.9,0.9); 

\draw[->, >=latex, gray, thick] (1,1) -- (1.4,1.4);
\draw[->, >=latex, gray, thick] (1.5,1.5) -- (1.9,1.1);

\draw[->, >=latex, gray, thick] (2,1) -- (2.4,1.4); 
\draw[->, >=latex, gray, thick] (1.5,1.5) -- (2.4,1.5); 

\draw[->, >=latex, gray, thick] (2,1) -- (2.9,1);
\draw[->, >=latex, gray, thick] (2.5,1.5) -- (2.9,1.1);

\draw[dashed, gray, thick] (2.5,1.5) -- (3,2);

%:::::::::::::::::::::::::::::::::::::::::::::::::::::::::::::::::::::::::::::::::
\node (1) at (0,0) {$\bullet$};
\node (1) at (2,0) {$\bullet$};
\node (1) at (1,1) {$\bullet$};
\node (1) at (1.5,1.5) {$\bullet$};
\node (1) at (2.5,1.5) {$\bullet$};
\node (1) at (1,0) {$\bullet$};
\node (1) at (0.5,0.5) {$\bullet$};
\node (1) at (1.5,0.5) {$\bullet$};
\node (1) at (2,1) {$\bullet$};
\node (1) at (3,1) {$\bullet$};
\node (1) at (-0.25,0) {$s_1$};
\node (1) at (2.25,0) {$t_1$};
\node (1) at (0.75,1) {$s_2$};
\node (1) at (3.25,1) {$t_2$};

\filldraw(0.5,0.24)     node[anchor=center] {$\textcolor{lightgray}{1}$};
\filldraw(1,0.26)     node[anchor=center] {$\textcolor{lightgray}{2}$};
\filldraw(1.5,0.24)     node[anchor=center] {$\textcolor{lightgray}{3}$};
\filldraw(1.5,0.76)     node[anchor=center] {$\textcolor{lightgray}{4}$};
\filldraw(1.5,1.26)     node[anchor=center] {$5$};
\filldraw(2,1.26)     node[anchor=center] {$6$};
\filldraw(2.5,1.26)     node[anchor=center] {$7$};

\end{tikzpicture}    
\end{center}
The first two cases (from left to right) correspond to the number of routes in the triangular snake graph associated to $\mathcal{G}_v[\underbrace{3,3,\cdots,3,3}_{(2k-2)-times}]$, while the remaining case corresponds to the number of routes in the triangular snake graph associated to $\mathcal{G}_h[\underbrace{3,3,\cdots,3,3}_{(2k-3)-times}]$. Therefore, the number of routes in 
$\mathcal{T}_{\mathcal{G}}$ (or perfect matchings in $\mathcal{G}$) satisfies the same recurrence relation as the Pell numbers. Similarly, if we consider the continued fraction $[\underbrace{2,2,\cdots,2,2}_{(2k-1)-times}]$, the associated path matrix is $M'_k$, whose determinant corresponds to the previously mentioned Pell number. Thus, the desired result is proved.
\end{proof}

\begin{example} Consider the snake graphs $\mathcal{G}'$ and $\mathcal{G}$ associated to the continued fractions $[2,2,2,2,2]$ and $[2,2,2,2,2,2]$, respectively. The number of perfect matchings of $\mathcal{G}'$ and $\mathcal{G}$ are given by
\[ 
m(\mathcal{G}')=\det M_3'=\left| \begin{array}{ccc}
 F_5 & F_4F_3 & F_4 \\ 
 F_2 & F_5+F_3 &  F_4+F_2 \\
 F_0 & F_2 &  F_4
\end{array} \right| =P_6=70,
\]
and
\[ 
m(\mathcal{G})=\det M_3=\left| \begin{array}{ccc}
 F_5 & F_4F_3 & F_4F_3 \\ 
 F_2 & F_5+F_3 &  (F_4+F_2)F_3  \\
 F_0 & F_2 &  F_5+F_3
\end{array} \right| =P_7=169.
\]
\end{example}

\section{Future work}
This work has showed that, for all snake graphs, the number of perfect matchings can be expressed as a sum of products of Fibonacci numbers. Additionally, we have explored how certain well-known sequences, such as the Fibonacci and Pell sequences, can be characterized through determinants of matrices with Fibonacci number entries. A natural question arises: can other sequences be described similarly? For instance, the Fibonacci and Pell numbers are particular cases within the broader Markov sequence. It would be interesting to study whether there are families of matrices that can effectively describe this sequence in an analogous manner.\par 
Furthermore, previous works such as \cite{CS13} and \cite{CS18} established a relationship between the continued fraction associated with a snake graph, the number of perfect matchings, and the framework of cluster algebras. Specifically, if $\mathcal{G}$ is the snake graph of a cluster variable within a cluster algebra arising from a surface, then each cluster variable $\textbf{x}$ is represented by a formula involving the perfect matchings of 
$\mathcal{G}$. Through the bijections introduced in this paper, we can reinterpret these expressions as products of arrow labels within paths of the $k$-route associated with a perfect matching, combined with labels of contraction edge vertices that are not part of any path in the route. This suggests a potential avenue for further exploration, where such combinatorial interpretations may offer new insights into the algebraic structures underlying cluster variables and related sequences.
%%%%%%%%%%%%%%%%%%%%%%%%%%%%%%%%%%%%%%%%%%%%%%%%%
% The following optional unnumbered section is where you put personal acknowledgements,
% research grant support, and similar things.  Do not put them on the front page.
\subsection*{Acknowledgements}

I would like to express my sincere gratitude to Alfredo Nájera Chávez for his invaluable comments, suggestions, and corrections on earlier drafts of this work. Additionally, I thank Timothy Magee for his insightful observations, particularly regarding the use of Hankel notation in Proposition~\ref{prop:Catalan_Fibonacci}. Their contributions have significantly improved the quality of this paper.

%BIBLIOGRAPHY
% You do not have to use the same format for your references, but 
%    include everything in this file.
% If you use BibTeX to create a bibliography, copy the .bbl file into here.
% We recommend you use \doi{...} and \arxiv{...} like the examples below,
% as they give a short display form with an active link to the full URL.

\end{document}